\documentclass[11pt]{article}
\usepackage{amssymb,latexsym,times}
\usepackage{mathptmx}
\usepackage{amscd}
\usepackage{color}
\usepackage[all]{xy}

\usepackage{graphicx} 
\usepackage{tikzit}

\tikzstyle{smat}=[fill=none, draw=none, shape=rectangle, tikzit fill={rgb,255: red,191; green,191; blue,191}]
\tikzstyle{gmat}=[fill=white, draw=none, shape=rectangle, tikzit fill={rgb,255: red,227; green,255; blue,216}]
\tikzstyle{gcoord}=[fill=white, draw=none, shape=rectangle, text=green, tikzit fill={rgb,255: red,110; green,255; blue,122}, scale=0.7, inner sep=1pt]
\tikzstyle{bcoord}=[fill=white, draw=none, shape=rectangle, tikzit fill={rgb,255: red,191; green,191; blue,191}, scale=0.7, inner sep=1pt]
\tikzstyle{rcoord}=[fill=white, draw=none, shape=rectangle, text=red, scale=0.7, inner sep=1pt, tikzit fill=red]
\tikzstyle{bcoord2}=[fill=white, draw=none, shape=rectangle, inner sep=1pt, scale=1.0]
\tikzstyle{gco2}=[fill=white, draw=none, shape=rectangle, text=green, tikzit fill={rgb,255: red,110; green,255; blue,122}, scale=0.9, inner sep=1pt]
\tikzstyle{bco2}=[fill=white, draw=none, shape=rectangle, tikzit fill={rgb,255: red,191; green,191; blue,191}, scale=0.9, inner sep=1pt]
\tikzstyle{rco2}=[fill=white, draw=none, shape=rectangle, text=red, scale=0.9, inner sep=1pt, tikzit fill=red]

\tikzstyle{arrow 1}=[->]
\tikzstyle{dotted}=[-, draw={rgb,255: red,191; green,191; blue,191}, fill=none]
\tikzstyle{sred}=[-, draw=red]

\newtheorem{Thm}{Theorem}[section]
\newtheorem{Lem}[Thm]{Lemma}
\newtheorem{Cor}[Thm]{Corollary}
\newtheorem{Prop}[Thm]{Proposition}
\newtheorem{Conj}[Thm]{Conjecture}
\newtheorem{example}[Thm]{Example}
\newtheorem{remark}[Thm]{Remark}
\newtheorem{Def}[Thm]{Definition}

\newcommand{\E}{\mathcal{E}}
\newcommand{\Z}{\mathbb{Z}}

\newcommand{\C}{\mathbb{C}}

\newcommand{\M}{{\mathcal M}}

\usepackage{xcolor}
\usepackage{version}

\excludeversion{NB}

\newcommand{\bsm}{\begin{smallmatrix}}
\newcommand{\esm}{\end{smallmatrix}}

\newcommand{\bg}{\mathbf{g}}

\def\cqfd{\hfill $\Box$ \medskip}
\def\CC{{\mathcal C}}
\def\resp{{\em resp.\ }}

\def\<{\langle\,}
\def\>{\,\rangle}

\def\P{{\mathbb P}}

\def\g{\mathfrak g}

\def\<{\langle}
\def\>{\rangle}

\def\ra{\rightarrow}

\def\1{\mathbf 1}

\def\P{{\mathcal P}}

\def\a{\alpha}

\def\L{\Lambda}
\def\mod{{\rm mod}\,}

\def\id{{\rm id}}

\def\de{\delta}
\def\De{\Delta}
\def\Sl{\mathfrak{sl}}
\def\la{\lambda}
\def\AA{\mathcal{A}}

\def\YY{\mathcal{Y}}

\def\M{{\mathfrak M}}
\def\O{\mathcal{O}}
\def\1{{\mathbf 1}}

\def\hg{\widehat{\mathfrak{g}}}

\def\G{\Gamma}
\def\tG{\widetilde{\Gamma}}

\def\be{\mathbf e}

\def\bi{\mathbf{i}}
\def\bj{\mathbf{j}}

\def\be{\mathbf{e}}

\def\Si{\Sigma}
\def\SS{\mathcal{S}}

\def\ds{\displaystyle}

\def\br{\mathbf{r}}
\def\hatsl{\widehat{\mathfrak{sl}}}

\def\mod{\mathrm{mod}}

\def\QQ{\mathcal{Q}}

\def\bQ{\underline{Q}}

\newcommand{\cQ}{\mathbf{Q}}

\def\be{\mathbf{e}}

\newcommand{\red}{\color{red}}
\newcommand{\blue}{\color{blue}}
\newcommand{\green}{\color{green}}

\newcommand{\df}{\colon}
\newcommand{\mugreen}{\mu_{\mathrm{grn}}}
\newcommand{\tmugreen}{\widetilde{\mu}_{\mathrm{grn}}}
\newcommand{\Mgreen}{M_{\mathrm{grn}}}
\newcommand{\Igrn}{I_{\mathrm{grn}}}

\newcommand{\bt}{\mathbf{t}}
\newcommand{\bs}{\mathbf{s}}
\newcommand{\bS}{\mathbf{S}}
\newcommand{\bT}{\mathbf{T}}

\newcommand{\uSS}{\underline{\mathcal{S}}}
\newcommand{\BB}{\mathcal{B}}

\newcommand{\tb}{\mathbf{\mathfrak{t}}}
\newcommand{\Psib}{{\Psi}}
\newcommand{\mfr}{\mathfrak{r}}
\newcommand{\tT}{\widetilde{\Theta}}
\newcommand{\tTheta}{\widetilde{\Theta}}
\newcommand{\Tp}{\Theta'}
\newcommand{\VV}{\mathcal{V}}

\def\bQQ{\underline{\mathcal{Q}}}
\def\uP{\underline{P}}

\begin{document}

\title{\bf Representations of shifted quantum affine algebras \\ 
and cluster algebras I. The simply-laced case}
\author{C. Geiss, D. Hernandez and B. Leclerc}

\date{}


\maketitle

\begin{abstract}
We introduce a family of cluster algebras of infinite rank associated with root systems of type $A$, $D$, $E$. 
We show that suitable completions of these cluster algebras are isomorphic to the Grothendieck rings of the 
categories $\mathcal{O}_\mathbb{Z}$ of the corresponding shifted quantum affine algebras. 
The cluster variables of a class of distinguished initial seeds are certain formal power series defined by E. Frenkel and the second author, which satisfy a system of functional relations called $QQ$-system. We conjecture that all cluster monomials are classes of simple
objects of $\mathcal{O}_\mathbb{Z}$. In the final section, we show that these cluster algebras contain infinitely many cluster subalgebras
isomorphic to the coordinate ring of the open double Bruhat cell of the corresponding simple simply-connected algebraic group.
This explains the similarity between $QQ$-system relations and certain generalized minor identities discovered by Fomin and Zelevinsky.
\end{abstract}


\setcounter{tocdepth}{2}
{\small \tableofcontents}

\section{Introduction}

Let $\g$ be a simple Lie algebra over $\C$, and let $U_q(\hg)$ be the corresponding untwisted quantum affine algebra for a 
generic quantum parameter $q$. In recent years, cluster algebras have become a powerful new tool for studying the tensor structure of the category of finite-dimensional modules over $U_q(\hg)$ \cite{HL0, HL1, HL3, Q, KKOP1, KKOP, BC1, BC2}.

Shifted quantum affine algebras $U_q^\mu(\hg)$ are a larger class of algebras introduced by Finkelberg and Tsymbaliuk \cite{FT} in their study of quantized $K$-theoretic Coulomb branches of $3d$ $N = 4$ SUSY quiver gauge theories. They depend on an integral coweight $\mu$ of $\g$. When $\mu = 0$, the algebra $U_q^0(\hg)$ is a central extension of $U_q(\hg)$ and it has essentially the same representation theory. When $\mu \not = 0$ the representation theory becomes very different, and for instance if $\mu$ is anti-dominant, then $U_q^\mu(\hg)$ does not have any non-trivial finite-dimensional representation. In \cite{H}, the second author has started a systematic study of the representation theory of $U_q^\mu(\hg)$. 
He has introduced a category $\O^\mu$ containing infinite-dimensional representations, and shown that the Grothendieck group of $\O^{\rm sh} := \bigoplus_{\mu \in P^\vee} \O^\mu$ has a natural ring structure coming from an operation on representations called fusion product. 

The aim of this paper is to show that the Grothendieck ring $K_0(\O_\Z)$ has an explicit cluster algebra structure. Here $\O_\Z$ denotes a full subcategory of $\O^{\rm sh}$ defined by certain integrality conditions on the loop-weights of the representations.

\begin{figure}[t]
\[
\def\objectstyle{\scriptstyle}
\def\lablestyle{\scriptstyle}
\xymatrix@-0.8pc{
&{}\save[]+<0cm,1.5ex>*{\vdots}\restore&{}\save[]+<0cm,1.5ex>*{\vdots}\restore  
&{}\save[]+<0cm,1.5ex>*{\vdots}\restore
\\
&\bullet\ar[rd]\ar[u]&& \ar[ld] \bullet\ar[u] 
\\
&&\ar[ld] \bullet \ar[uu]\ar[rd]&&
\\
&\ar[uu]{\red\bullet}\ar[rd]&&\ar[ld]\ar[uu] \bullet 
\\
&&\ar[ld]\ar[uu] {\red\bullet} \ar[rd]&&
\\
& {\red\bullet}\ar[uu] \ar[rd]&&\ar[ld]\ar[uu] {\red\bullet} 
\\
&&\ar[ld]\ar[uu] {\red\bullet} \ar[rd]&&
\\
&\ar[uu] {\red\bullet}\ar[rd] &&\ar[ld]\ar[uu] \bullet 
\\
&&\ar[ld]\ar[uu] \bullet\ar[rd] &&
\\
&\ar[uu] \bullet &\ar[u]&\ar[uu] \bullet 
\\
&{}\save[]+<0cm,3ex>*{\vdots}\restore&{}\save[]+<0cm,3ex>*{\vdots}\restore  
&{}\save[]+<0cm,3ex>*{\vdots}\restore
}
\qquad
\xymatrix@-1,5pc{
&{}\save[]+<0cm,1.5ex>*{\vdots}\restore&{}\save[]+<0cm,1.5ex>*{\vdots}\restore  
&{}\save[]+<0cm,1.5ex>*{\vdots}\restore
\\
&\bullet\ar[rd]\ar[u]&& \ar[ld]\ar[u] \bullet 
\\
&&\ar[ld] \bullet \ar[uu]\ar[rd]&&
\\
&\ar[uu]{\red\bullet}\ar[d]&&\ar[ldd]\ar[uu] \bullet 
\\
&{\green\bullet}\ar[rd]
\\
&&\ar[uuu]\ar[d] {\red\bullet} &&
\\
&&\ar[ld] {\green\bullet} \ar[rd]&&
\\
&\ar[uuu] {\red\bullet} \ar[d]&&\ar[d]\ar[uuuu] {\red\bullet} 
\\
& {\green\bullet} \ar[rd]&&\ar[ld] {\green\bullet} 
\\
&&\ar[d]\ar[uuu] {\red\bullet}&&
\\
&&\ar[ld] {\green\bullet} \ar[rdd]&&
\\
&\ar[uuu] {\red\bullet}\ar[d] &&  
\\
& {\green\bullet}\ar[rd] &&\ar[ld]\ar[uuuu] \bullet 
\\
&&\ar[ld]\ar[uuu] \bullet\ar[rd] &&
\\
&\ar[uu] \bullet &\ar[u]&\ar[uu] \bullet 
\\
&{}\save[]+<0cm,1.5ex>*{\vdots}\restore&{}\save[]+<0cm,1.5ex>*{\vdots}\restore  
&{}\save[]+<0cm,1.5ex>*{\vdots}\restore
}
\]
\caption{\label{fig0} {\it The quiver $\G_e$ (with its red subquiver $G_\cQ$), and the quiver $\G_{w_0}$ in type $A_3$.}}
\end{figure}
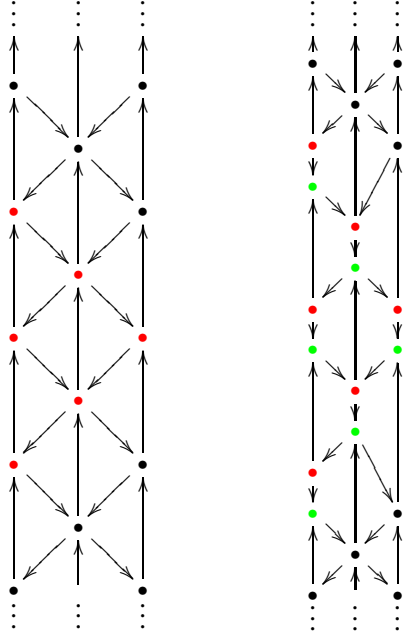

More precisely, in Section~\ref{Sect-CA} we introduce a class of infinite rank cluster algebras $\AA_w$ labelled by elements $w$ of the Weyl group $W$ of $\g$. The cluster algebra $\AA_e$ associated with the unit element $e\in W$ was already considered in \cite{HL2}, in connection with certain subcategories $\O^+$ and $\O^-$ of the Hernandez-Jimbo category $\O$ of a Borel subalgebra $U_q(\mathfrak{b})$ of $U_q(\hg)$ \cite{HJ}. 
Its defining quiver $\G$ is a doubly-infinite quiver, which in the $A$, $D$, $E$ cases coincides with the Auslander-Reiten quiver of the derived category $D^b(K\cQ)$ of a Dynkin quiver $\cQ$ of the same type, with added vertical up arrows corresponding to the Auslander-Reiten translation. For $w = w_0$, the longest element of~$W$, an easy way of describing the quiver $\G_{w_0}$ of an initial seed of $\AA_{w_0}$ is as follows (see \S\ref{subsec-Coxeter}). Let $G_{\cQ}$ denote the finite subquiver of $\G$ corresponding to the Auslander-Reiten quiver of the abelian module category $\mod(K\cQ)$. Replace each vertex of $G_\cQ$ by a pair of a red and a green vertex connected by a down arrow, and rearrange the incident arrows so that $3$-cycles of $\G$ involving at least two vertices of $G_\cQ$ become $4$-cycles in the new quiver~$\G_{w_0}$. 
This is illustrated in Figure~\ref{fig0} for an equi-oriented quiver  $\cQ$ of type~$A_3$.

A remarkable feature of $\G_{w_0}$, which plays a central rôle in our constructions, is that if we perform a sequence 
of quiver mutations at all red vertices (\resp at all green vertices), we get an isomorphic quiver in which the middle finite part consisting of red and green vertices has been shifted one step up (\resp one step down). (These red (\resp green) mutations commute with each other, so the sequence order is irrelevant.) Iterating infinitely many times this sequence of, 
say, green mutations, we can shift down infinitely many times the middle red-green part of $\G_{w_0}$, and obtain in the limit the quiver $\G_e$. This will allow us to regard $\G_e$ as a ``reference seed at infinity'' for the cluster algebra $\AA_{w_0}$. We will give a precise meaning to this in Section~\ref{sect-stable-g}, where we will attach certain ``stabilized $g$-vectors'' to the cluster variables of $\AA_{w_0}$.

In Section~\ref{sect-formal-power-series}, following \cite{FH2}, we introduce rings of formal power series endowed with an action of the Weyl group $W$. In fact, for our purposes it is convenient to consider a ring $\Pi'$ a bit larger than the ring $\Pi$ of \cite{FH2}, and to extend to $\Pi'$ the action of $W$ defined in \cite{FH2}. Then in Section~\ref{sec-QQ}, following \cite{FH3}, we define elements $\mathcal{Q}_{w(\varpi_i), a}\in\Pi'$ by acting with $W$ on certain generators of $\Pi'$. Here $w\in W$, $a\in \C^*$, and $\varpi_i$ denotes a fundamental weight of $\g$. We then quote from \cite{FH3} an important system of algebraic identities satisfied by the elements $\mathcal{Q}_{w(\varpi_i), a}\in\Pi'$, named $QQ$-system. This $QQ$-system (also called full $QQ$-system) has a long history, starting from work of Bazhanov-Lukyanov and Zamolodchikov \cite{BLZ} on quantum KdV systems, continuing with work of Masoero-Raimondo \cite{MR} on the corresponding opers introduced by Feigin and Frenkel, and pursued in related work of Frenkel and the second author on XXZ-type models \cite{FH,FH3}, and in work of Koroteev, Sage and Zeitlin on $q$-opers \cite{KSZ} (we refer the reader to \cite[\S 6.2]{FKSZ} and \cite{FH3} for a detailed historical account). For us the $QQ$-system served as a guiding principle for designing appropriate initial cluster seeds.

After these preparations, we introduce in Section~\ref{subsection-KZ} a topological subring $K_\Z$ of a component of $\Pi'$. 
By construction, $K_\Z$ is topologically generated by certain elements $\bQ_{w(\varpi_i), q^r}$. Here $(i,r)$ runs through the canonical set of labels of the vertices of $\G_{w_0}$, and $\bQ_{w(\varpi_i), q^r}$ is a suitable rescaling of a component of $\mathcal{Q}_{w(\varpi_i), q^r}$, which satisfies a coefficient-free version of the $QQ$-system relations (Proposition~\ref{Prop-renorm-QQ}). We can then formulate and prove our first main theorem (Theorem~\ref{main-Thm}), which describes an explicit injective ring homomorphism $F\colon \AA_{w_0} \to K_\Z$ such that
the topological closure of the image is equal to $K_\Z$. The images under $F$ of the cluster variables of the initial seed with quiver $\G_{w_0}$ are all of the form $\bQ_{w(\varpi_i), q^r}$, and the initial exchange relations at red or green vertices are mapped by $F$ to instances of the $QQ$-system relations.

In Section~ \ref{sub-sec-shift-quantum}, we recall following \cite{FT} and \cite{H} the necessary background on shifted quantum affine algebras and their category $\O^{\mathrm{sh}}$, and 
we explain that the ring $K_\Z$ is isomorphic to the Gro\-then\-dieck ring $K_0(\O_\Z)$. 
It follows from Theorem~\ref{main-Thm} that $\AA_{w_0}$ can also be regarded as a subring of $K_0(\O_\Z)$, whose topological closure is equal to $K_0(\O_\Z)$. Our main conjecture (Conjecture~\ref{main-conj}) then states that every cluster monomial of $\AA_{w_0}$ is the class of a simple object of $\O_\Z$. In the rest of Section~\ref{sect-shift-quantum} we collect evidences supporting this conjecture. In particular we prove that it holds when $\g = \mathfrak{sl}_2$. 
In this simple case, one can give an explicit list of cluster variables and clusters, and show that every simple object in $\O_\Z$ has a unique factorization into a fusion product of prime simple objects, thus generalizing the seminal results of $\cite{CP}$. We also note that, for every $\g$, one can regard the cluster algebra $\AA_e$ as a subalgebra of $\AA_{w_0}$ isomorphic to the Grothendieck ring of the subcategory $\CC_\Z$ of finite-dimensional modules of $\O_\Z$. Then, putting together results of \cite{HL2, KKOP, H}, we see that the conjecture holds for all cluster monomials of $\AA_e$, that is for finite-dimensional modules of $\O_\Z$.   

As explained in \cite{H}, there are strong relations between the category $\O^{\mathrm{sh}}$ and the Hernandez-Jimbo category $\O$ for $U_q(\mathfrak{b})$. For instance their parametrizing set of simple objects are identical. Yet, it does not seem possible to extend to the full category $\O$ the results of \cite{HL2} for its subcategories $\O^+$ and $\O^-$. One reason is that in the case of $U_q(\mathfrak{b})$, the tensor product of a positive and a negative prefundamental representations is never simple, in contrast with what happens in $\O^{\mathrm{sh}}$. Another notable difference is indicated in Remark~\ref{Rem-sh-Borel}.

A large part of our construction works equally well whether $\g$ is of simply-laced type or not. However, when $\g$ is simply-laced, we have additional results relating the cluster structure of $K_\Z$ to the cluster structure of the coordinate ring of the open double Bruhat cell of a simply-connected algebraic group $G$ with Lie algebra $\g$. 
Moreover, the combinatorial description of the quiver of an initial seed of $K_\Z$ is significantly simpler when $\g$ is simply-laced. This is why in this paper we restrict our attention to the simply-laced case, and defer the non simply-laced case to a forthcoming paper. 

In Section~\ref{sect-Bruhat-Wronskian}, where the assumption that $\g$ is simply-laced is necessary, 
we show that the cluster algebra $\AA_{w_0}$ is generated by an infinite family of cluster subalgebras, all isomorphic to
the Berenstein-Fomin-Zelevinsky cluster algebra structure on the coordinate ring of the open double Bruhat cell $G^{w_0,w_0}$. More precisely, for each choice of an orientation $\cQ$ of the Dynkin diagram of~$\g$, or equivalently, for each choice of a Coxeter element $c$ of $W$, we have a one-parameter family of such subalgebras parametrized by $q^{2\Z}$, which are related with each other by the sequences of mutations at green or red vertices mentioned above. It then follows from Theorem~\ref{main-Thm} that one can identify the cluster algebra $\C[G^{w_0,w_0}]$ with each member of an infinite series of finite rank cluster subalgebras of $\C\otimes K_\Z$, in such a way that the generalized minors $\Delta_{v(\varpi_i),w(\varpi_i)}$, which are the cluster variables of the standard initial seeds of $\C[G^{w_0,w_0}]$, get identified with elements of $K_\Z$ of the form
$\bQ_{w(\varpi_i),q^r}$. In these identifications, the generalized minor identities of \cite[Theorem 1.17]{FZ-Bruhat} translate into instances of $QQ$-system relations.

A similar connection between generalized minor identities and $QQ$-system relations has already been pointed out by
Koroteev and Zeitlin \cite{KZ}. They have shown that the $QQ$-systems occurring in the theory of $q$-opers \cite{KSZ, FKSZ} emerge as the relations between generalized minors of certain sections $\mathcal{G}(z)$, called $(G,q)$-Wronskians, of a principal $G$-bundle on the projective line.
Imitating their result, we present in \S \ref{subsec-qW} an analogous construction of quantum $(G,c)$-Wronskians $(g_c(z))$ in the Grothendieck ring $K_\Z$. Here $c\in W$ is a fixed Coxeter element, $z\in q^{2\Z}$, and $g_c(z)$ is an element of the algebraic group $G(K_\Z)$ satisfying quantum Wronskian identities of the form:
\begin{equation}\label{eq-qW}
\Delta_{c^k(\varpi_i),\,c^\ell(\varpi_i)}(g_c(z)) = \Delta_{c^{k-1}(\varpi_i),\,c^\ell(\varpi_i)}(g_c(q^2z)),
\quad (i\in I,\ 1\le k\le m_i,\ 0\le \ell \le m_i,\ z\in q^{2\Z}),
\end{equation}
(see Section~\ref{sect-Bruhat-Wronskian} below for unexplained notation). 
In our setting, the formulas defining $g_c(z)$ and the fact that they solve the system of equations (\ref{eq-qW}) appear as direct consequences of the cluster algebra structure of $K_\Z$, with its sequence of cluster subalgebras isomorphic to $\C[G^{w_0,w_0}]$. 

\medskip
{\bf Acknowledgements.} This research was started  in November 2021 during a Research In Pair stay at the Mathematisches For\-schungs\-institute Oberwolfach. We want to thank MFO for excellent working conditions. 
C.G. and B.L. acknowledge partial support from ERC grant QAffine, and from PAPIIT (UNAM) grant IN116723.  
C.G. thanks LMNO (Université de Caen Normandie) for an invitation in June 2023, and acknowleges partial support from 
the Simons Foundation and the Mathematisches Forschungsinstitut Oberwolfach for his combined visit to MFO and Paris in January 2024. We are also grateful to E. Frenkel for allowing us to use some of the results of \cite{FH3} before their publication. Finally, we thank an anonymous referee for reading the manuscript in detail and preparing a very helpful list of questions and corrections.

\section{Basic definitions}

\subsection{Cartan matrix and Dynkin diagram of a simple Lie algebra}\label{ssecCartan}
Let $\g$ be a simple Lie algebra over $\C$, with Cartan generators $e_i, f_i, h_i\ (i\in I)$.
We denote by $n:= |I|$ the rank of $\g$.
Let $C=(c_{ij})_{i,j\in I}$ be the Cartan matrix of $\g$, which encodes the Serre presentation.
In this paper we will assume that $\g$ is of simply-laced type, that is of type $A$, $D$, $E$ in the Cartan-Killing classification. This means that $C$ is a symmetric matrix.
The Dynkin diagram of $\g$ is the unoriented graph with vertex set $I$ and incidence matrix $2\,\mathrm{Id} - C$.

\subsection{Weights, roots, and Weyl group}

Let $\mathfrak{h} := \oplus_{i\in I}\C h_i$ denote the Cartan subalgebra of $\g$, 
and let $\mathfrak{h}^*$ be its dual vector space. 
We denote by $(\varpi_i\mid i\in I)$ the basis of $\mathfrak{h}^*$ dual to 
$(h_i\mid i\in I)$, and call its elements the fundamental weights.
We denote by $\a_i\ (i\in I)$ the simple roots of $\g$, defined by
\begin{equation}\label{fundw_roots}
\a_i = \sum_{j\in I} c_{ji} \varpi_j,\qquad (i,j\in I). 
\end{equation}
They span respectively the root lattice and the weight lattice:
\[
 Q := \bigoplus_{i\in I} \Z\a_i\ \subset \  P := \bigoplus_{i\in I} \Z\varpi_i
 \ \subset \ \mathfrak{h}^*.
\]
Let $Q_+ := \bigoplus_{i\in I} \Z_{\ge 0}\a_i $.
The weight lattice $P$ is endowed with a partial ordering given by:
\[
 \la \ge \mu \quad \Longleftrightarrow \quad \la-\mu\in Q_+,\qquad (\la,\mu \in P).
\]

The Weyl group $W$ of $\g$ is generated by the simple reflexions $s_i$, which act on $\mathfrak{h}^*$ by
\begin{equation}\label{action-s-fw-sr}
 s_i(\la) = \la - \la(h_i) \a_i,\qquad (i\in I,\ \la\in \mathfrak{h}^*).
\end{equation}
In particular, we have
\[
s_i(\varpi_j) = \varpi_j - \de_{ij}\a_i,\quad s_i(\a_j) = \a_j - c_{ij}\a_i,\qquad (i,j\in I).
\]
Let $w_0$ be the longest element of $W$. It induces an involution $\nu : I \to I$ defined via the equality
\[
 w_0(\alpha_i) = -\alpha_{\nu(i)},\quad (i\in I).
\]

\subsection{Basic infinite quiver}\label{ssec1-4}

\begin{figure}[t]
\[
\def\objectstyle{\scriptstyle}
\def\lablestyle{\scriptstyle}
\xymatrix@-1.0pc{
&{}\save[]+<0cm,1.5ex>*{\vdots}\restore&{}\save[]+<0cm,1.5ex>*{\vdots}\restore  
&{}\save[]+<0cm,1.5ex>*{\vdots}\restore
\\
&{(1,2)}\ar[rd]\ar[u]&
&\ar[ld] (3,2) \ar[u]
\\
&&\ar[ld] (2,1) \ar[rd]\ar[uu]&&
\\
&\ar[uu]{(1,0)}\ar[rd]&
&\ar[ld] (3,0) \ar[uu]
\\
&&\ar[uu]\ar[ld] (2,-1) \ar[rd]&&
\\
&\ar[uu](1,-2) \ar[rd] &&\ar[ld] (3,-2)\ar[uu]
\\
&&\ar[ld] \ar[uu](2,-3) \ar[rd]&&
\\
&\ar[uu](1,-4) \ar[rd] &&\ar[ld] (3,-4)\ar[uu]
\\
&&\ar[ld] \ar[uu](2,-5) \ar[rd]&&
\\
&\ar[uu](1,-6) &&\ar[uu] (3,-6) 
\\
&{}\save[]+<0cm,0ex>*{\vdots}\ar[u]\restore&{}\save[]+<0cm,0ex>*{\vdots}\ar[uu]\restore  
&{}\save[]+<0cm,0ex>*{\vdots}\ar[u]\restore
\\
}
\]
\caption{\label{Fig0} {\it The basic quiver $\G$ in type $A_3$.}}
\end{figure}
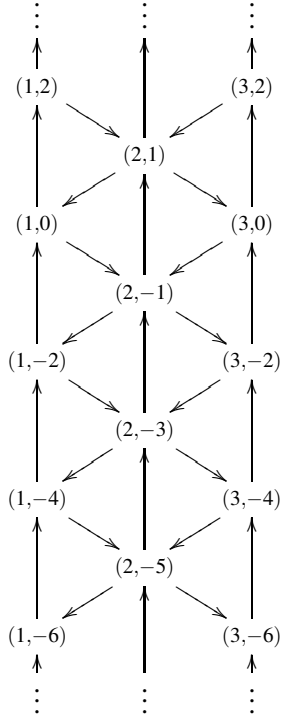

Following \cite{HL1}, we attach an infinite quiver to the Cartan matrix $C$.
Put $\widetilde{V} = I \times \Z$.
We introduce a quiver $\tG$ with vertex set $\widetilde{V}$.
The arrows of $\tG$ are given by
\[
(i,r) \to (j,s) 
\quad \Longleftrightarrow \quad
(c_{ij}\not = 0 
\quad \mbox{and} \quad
s=r+c_{ij}).
\]
It is easy to check that the oriented graph $\tG$ has two isomorphic connected components. 
We pick one of them and call it $\Gamma$. The vertex set of $\Gamma$ is denoted by $V$. 
We call $\Gamma = \Gamma_C$ the \emph{basic infinite quiver of $C$}.
An example in type $A_3$ is shown in Figure~\ref{Fig0}.

\section{Quivers and cluster algebras}\label{Sect-CA}

Let $U_q(\widehat{\g})$ denote the quantum affinization of the Lie algebra $\g$.
In \cite{HL2}, the cluster algebra $\AA_\G$ defined by the basic infinite quiver $\G$ was introduced, and it was shown that a completion of $\AA_\G$ is isomorphic to the Grothendieck rings of two tensor subcategories $\O^+_{\Z}$ and $\O^-_{\Z}$ of the category $\O$ of representations of the Borel subalgebra $U_q(\widehat{\mathfrak{b}})$ of $U_q(\widehat{\g})$ defined in \cite{HJ}. More recently, it was shown in \cite{H} that $\AA_\G$ is isomorphic to the Grothendieck ring of the subcategory $\mathcal{C}_\Z$ of $\O^{\rm sh}$ whose objects are the finite-dimensional representations of $\O_\Z$, see below Theorem~\ref{fdrep}. 
The aim of this section is to introduce a wider class of cluster algebras $\AA_w$ labelled by $w\in W$, where $\AA_e := \AA_\G$ corresponds to the unit element $e$ of $W$.

\subsection{Type $A_1$}\label{sect-3.1}

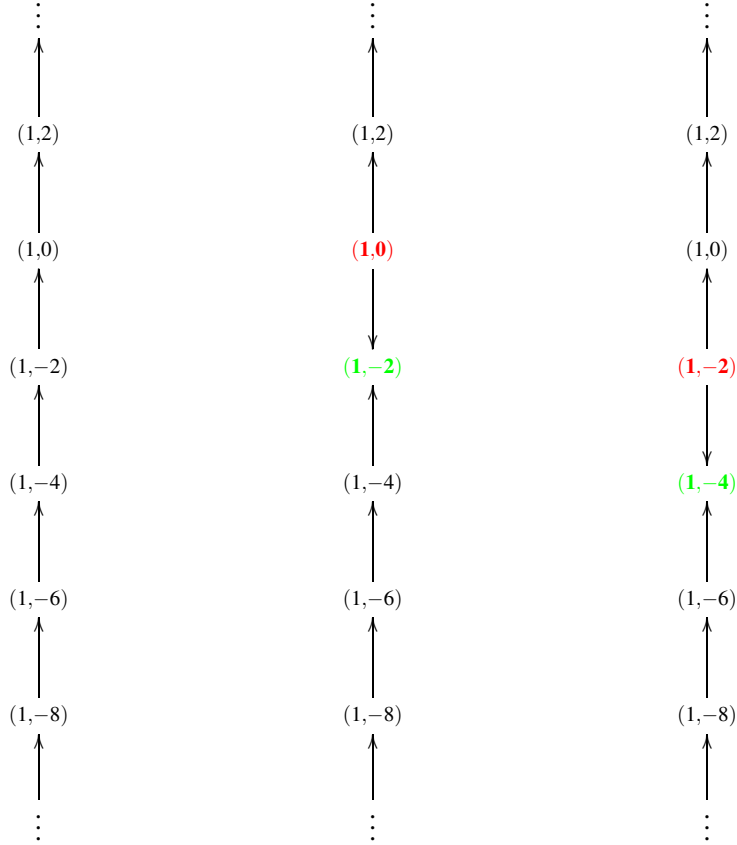
\begin{figure}[t]\label{fig-A1}
\[
\def\objectstyle{\scriptstyle}
\def\lablestyle{\scriptstyle}
\xymatrix@-1.0pc{
&{}\save[]+<0cm,1.5ex>*{\vdots}\restore&
\\
&
\\
&{(1,2)}\ar[uu]&
&
\\
&&
\\
&\ar[uu]{(1,0)}
\\
&&
\\
&\ar[uu](1,-2) 
\\
&&
\\
&\ar[uu](1,-4) 
\\
&&
\\
&\ar[uu](1,-6) 
\\
&&
\\
&\ar[uu](1,-8) 
\\
&&
\\
&{}\save[]+<0cm,0ex>*{\vdots}\ar[uu]\restore&
\\
}
\qquad\qquad
\xymatrix@-1.0pc{
&{}\save[]+<0cm,1.5ex>*{\vdots}\restore&
\\
&
\\
&{(1,2)}\ar[uu]&
&
\\
&&
\\
&\ar[uu]{\mathbf{\red(1,0)}}\ar[dd]
\\
&&
\\
&\mathbf{\green(1,-2)}
\\
&&
\\
&\ar[uu](1,-4)
\\
&&
\\
&\ar[uu](1,-6) 
\\
&&
\\
&\ar[uu](1,-8) 
\\
&&
\\
&{}\save[]+<0cm,0ex>*{\vdots}\ar[uu]\restore&
\\
}
\qquad\qquad
\xymatrix@-1.0pc{
&{}\save[]+<0cm,1.5ex>*{\vdots}\restore&
\\
&
\\
&{(1,2)}\ar[uu]&
&
\\
&&
\\
&\ar[uu]{(1,0)}
\\
&&
\\
&\ar[uu]{\mathbf{\red(1,-2)}}\ar[dd]
\\
&&
\\
&\mathbf{\green(1,-4)}
\\
&&
\\
&\ar[uu](1,-6) 
\\
&&
\\
&\ar[uu](1,-8) 
\\
&&
\\
&{}\save[]+<0cm,0ex>*{\vdots}\ar[uu]\restore&
\\
}
\]
\caption{\label{Fig1} {\it The quivers $\G_e$, $\G_{s_1,0}$ and $\G_{s_1,-2}$ in type $A_1$.}}
\end{figure}

The construction is very simple when $\g = \mathfrak{sl}_2$, but since it is the basis of the general construction we first present it in some detail. 

In this case, the quiver $\G = \G_e$ is simply an equi-oriented quiver of type $A_\infty$. Let $r\in 2\Z$. To obtain the new quiver $\G_{s_1,\,r}$, we just change the orientation of the arrow 
$(1,r) \leftarrow (1,r-2)$ and replace it by a down arrow $(1,r) \to (1,r-2)$. 
For instance, the quivers $\G_e$, $\G_{s_1,\,0}$ and $\G_{s_1,\,-2}$ are displayed in Figure~\ref{Fig1}.

Note that the new quiver $\G_{s_1,\,r}$ is again of type $A_\infty$, but that it is no longer equi-oriented. Note also that if we perform on $\G_{s_1,\,r}$ a quiver mutation at vertex $(1,r)$ 
then we get exactly the quiver $\G_{s_1,\,r+2}$. Similarly, if we perform on $\G_{s_1,\,r}$ a quiver mutation at vertex $(1,r-2)$ 
then we get the quiver $\G_{s_1,\,r-2}$. Hence, by induction, all the quivers $\G_{s_1,\,r}\ (r\in 2\Z)$ are mutation-equivalent. Therefore they define the same cluster algebra, which we denote by $\AA_{s_1}$.

Next, we remark that, for a finite $n\in\Z_{>0}$, two arbitrary quivers of type $A_n$ are mutation-equivalent. Hence there is a single cluster algebra of cluster-type $A_n$, up to isomorphism. In contrast, the quivers $\G$ and $\G_{s_1,\,r}$ cannot be obtained from each other through a \emph{finite} sequence of mutations, hence the cluster algebras $\AA_{e}$ and $\AA_{s_1}$ are different.

However, if we perform in $\G_{s_1,\,r}$ an \emph{infinite} sequence of mutations at successive vertices 
\[
(1,r-2), (1,r-4), (1,r-6),\ \ldots,
\]
we will obtain in the limit the quiver $\G$ of the cluster algebra $\AA_e$. On the other hand, in contrast with a finite sequence of mutations, this infinite sequence of mutations cannot be reversed (what should be the first mutation of the reversed sequence ?). 

To emphasize the special role of the two vertices $(1,r)$ and $(1,r-2)$ of $\G_{s_1,\,r}$, we will paint the source vertex $(1,r)$ in red and the sink vertex $(1,r-2)$ in green. In order to generalize to higher rank root systems, it is better to think of $\G_{s_1,\,r}$ as being obtained from $\G$ in three steps:
\begin{enumerate}
 \item[(i)] insert a new vertex $*$ between $(1,r)$ and $(1,r-2)$;
 \item[(ii)] replace the arrow $(1,r) \leftarrow (1,r-2)$ by a pair of arrows 
$(1,r) \to * \leftarrow (1,r-2)$;
 \item[(iii)] change the labels of the lower half as follows : 
 \[
 * \mapsto (1,r-2), \qquad (1,r-2k) \mapsto (1,r-2k-2),\ (k \ge 1).
 \]
\end{enumerate}

\subsection{Type $A, D, E$ with $w=s_i$}
\label{subsec2-2}

\begin{figure}[t]
\[
\def\objectstyle{\scriptstyle}
\def\lablestyle{\scriptstyle}
\xymatrix@-1.0pc{
&{}\save[]+<0cm,1.5ex>*{\vdots}\restore&{}\save[]+<0cm,1.5ex>*{\vdots}\restore  
&{}\save[]+<0cm,1.5ex>*{\vdots}\restore
\\
&{(1,2)}\ar[rd]\ar[u]&
&\ar[ld] (3,2) \ar[u]
\\
&&\ar[ld] (2,1) \ar[rd]\ar[uu]&&
\\
&\ar[uu]{\mathbf{\red(1,0)}}\ar[d]&&\ar[ldd] (3,0) \ar[uu]
\\
&{\mathbf{\green(1,-2)}}\ar[rd]
\\
&&\ar[uuu]\ar[ld] (2,-1) \ar[rd]&&
\\
&\ar[uu](1,-4) \ar[rd] &&\ar[ld] (3,-2)\ar[uuu]
\\
&&\ar[ld] \ar[uu](2,-3) \ar[rd]&&
\\
&\ar[uu](1,-6) \ar[rd] &&\ar[ld] (3,-4)\ar[uu]
\\
&&\ar[ld] \ar[uu](2,-5) \ar[rd]&&
\\
&\ar[uu](1,-8) &&\ar[uu] (3,-6) 
\\
&{}\save[]+<0cm,0ex>*{\vdots}\ar[u]\restore&{}\save[]+<0cm,0ex>*{\vdots}\ar[uu]\restore  
&{}\save[]+<0cm,0ex>*{\vdots}\ar[u]\restore
\\
}
\qquad
\xymatrix@-1.0pc{
&{}\save[]+<0cm,1.5ex>*{\vdots}\restore&{}\save[]+<0cm,1.5ex>*{\vdots}\restore  
&{}\save[]+<0cm,1.5ex>*{\vdots}\restore
\\
&{(1,2)}\ar[rd]\ar[u]&
&\ar[ld] (3,2) \ar[u]
\\
&&\ar[ld] (2,1) \ar[rd]\ar[uu]&&
\\
&\ar[uu]{(1,0)}\ar[rd]&
&\ar[ld] (3,0) \ar[uu]
\\
&&\ar[uu] \mathbf{\red(2,-1)}\ar[d] &&
\\
&&\ar[ld] \mathbf{\green(2,-3)} \ar[rd]&&
\\
&\ar[uuu](1,-2) \ar[rd] &&\ar[ld] (3,-2)\ar[uuu]
\\
&&\ar[ld] \ar[uu](2,-5) \ar[rd]&&
\\
&\ar[uu](1,-4) \ar[rd] &&\ar[ld] (3,-4)\ar[uu]
\\
&&\ar[ld] \ar[uu](2,-5) \ar[rd]&&
\\
&\ar[uu](1,-6) &&\ar[uu] (3,-6) 
\\
&{}\save[]+<0cm,0ex>*{\vdots}\ar[u]\restore&{}\save[]+<0cm,0ex>*{\vdots}\ar[uu]\restore  
&{}\save[]+<0cm,0ex>*{\vdots}\ar[u]\restore
\\
}
\]
\caption{\label{Fig2} {\it The quivers $\G_{s_1,\,0}$ and $\G_{s_2,\,-1}$ in type $A_3$.}}
\end{figure}
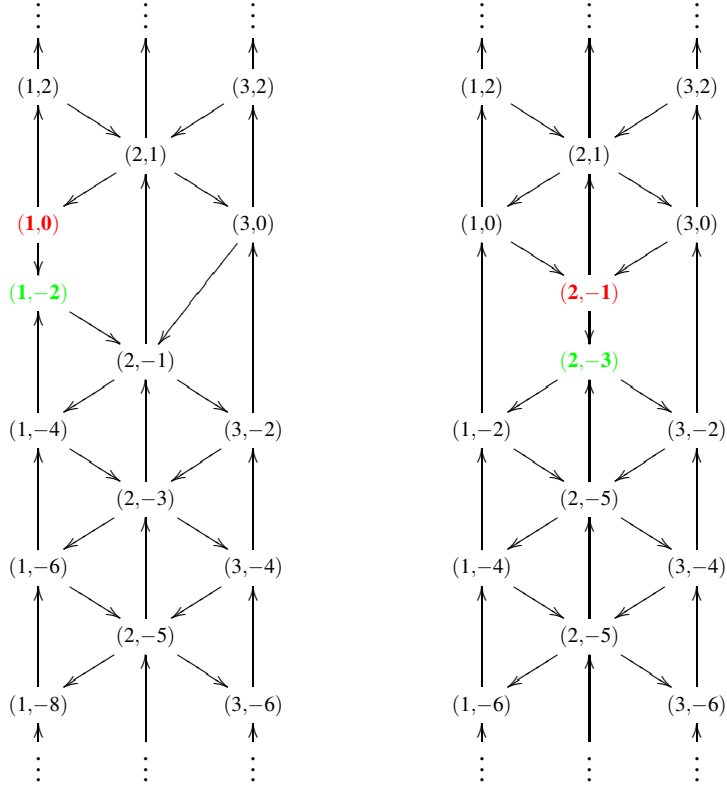

Let us now take $\g$ of type $A, D, E$, and let us define the cluster algebras $\AA_{s_i}$ associated with the simple reflections $s_i$. We start by defining the quivers $\G_{s_i,\,r}$.

\begin{Def}\label{Def2-1}
Let $(i,r)\in V$, the vertex set of $\G = \G_e$. The quiver $\G_{s_i,\,r}$ is obtained from 
$\G_e$ by performing the following operations:
\begin{enumerate}
 \item[(i)] insert a new vertex $*$ between vertices $(i,r)$ and $(i,r-2)$;
 \item[(ii)] replace the arrow $(i,r) \leftarrow (i,r-2)$ by a pair of arrows 
$(i,r) \to * \leftarrow (i,r-2)$;
 \item[(iii)] for $j$ with $c_{ij}<0$, replace the arrow $(i,r) \to (j,r+c_{ij})$ by an arrow $* \to (j,r + c_{ij})$;
 \item[(iv)] change the labels of the vertices on the lower half of column $i$ as follows:
\[
 * \mapsto (i,r-2), \qquad (i,r-2k) \mapsto (i,r-2k-2),\ (k \ge 1).
 \] 
\end{enumerate}
\end{Def}

For example, the quivers $\G_{s_1,\,0}$ and $\G_{s_2,\,-1}$ in type $A_3$ are displayed in Figure~\ref{Fig2}.
The proof of the following lemma is elementary, but it is a key property for our constructions.

\begin{figure}[t]
\[
\def\objectstyle{\scriptstyle}
\def\lablestyle{\scriptstyle}
\xymatrix@-1.0pc{
&{}\save[]+<0cm,1.5ex>*{\vdots}\restore&{}\save[]+<0cm,1.5ex>*{\vdots}\restore  
&{}\save[]+<0cm,1.5ex>*{\vdots}\restore
\\
&{(1,2)}\ar[rd]\ar[u]&
&\ar[ld] (3,2) \ar[u]
\\
&&\ar[ld] (2,1) \ar[rd]\ar[uu]&&
\\
&\ar[uu]{(1,0)}\ar[rd]&
&\ar[ld] (3,0) \ar[uu]
\\
&&\ar[uu] \mathbf{\red(2,-1)}\ar[d] &&
\\
&&\ar[ld] \mathbf{\green(2,-3)} \ar[rd]&&
\\
&\ar[uuu](1,-2) \ar[rd] &&\ar[ld] (3,-2)\ar[uuu]
\\
&&\ar[ld] \ar[uu](2,-5) \ar[rd]&&
\\
&\ar[uu](1,-4)  && (3,-4)\ar[uu]
\\
&{}\save[]+<0cm,0ex>*{\vdots}\ar[u]\restore&{}\save[]+<0cm,0ex>*{\vdots}\ar[uu]\restore  
&{}\save[]+<0cm,0ex>*{\vdots}\ar[u]\restore
\\
}
\xymatrix@-1.0pc{
&{}\save[]+<0cm,1.5ex>*{\vdots}\restore&{}\save[]+<0cm,1.5ex>*{\vdots}\restore  
&{}\save[]+<0cm,1.5ex>*{\vdots}\restore
\\
&{(1,2)}\ar[rd]\ar[u]&
&\ar[ld] (3,2) \ar[u]
\\
&& \mathbf{\red(2,1)} \ar[uu] \ar[dd]&&
\\
&\ar[uu]{(1,0)}\ar[rdd]&
& (3,0) \ar[uu]\ar[ldd]
\\
&& \ar[ul]\mathbf{\green
(2,-1)}\ar[ur] &&
\\
&&\ar[ld] (2,-3)\ar[u] \ar[rd]&&
\\
&\ar[uuu](1,-2) \ar[rd] &&\ar[ld] (3,-2)\ar[uuu]
\\
&&\ar[ld] \ar[uu](2,-5) \ar[rd]&&
\\
&\ar[uu](1,-4)  && (3,-4)\ar[uu]
\\
&{}\save[]+<0cm,0ex>*{\vdots}\ar[u]\restore&{}\save[]+<0cm,0ex>*{\vdots}\ar[uu]\restore  
&{}\save[]+<0cm,0ex>*{\vdots}\ar[u]\restore
\\
}
\xymatrix@-1.0pc{
&{}\save[]+<0cm,1.5ex>*{\vdots}\restore&{}\save[]+<0cm,1.5ex>*{\vdots}\restore  
&{}\save[]+<0cm,1.5ex>*{\vdots}\restore
\\
&{(1,2)}\ar[rd]\ar[u]&
&\ar[ld] (3,2) \ar[u]
\\
&& \mathbf{\red(2,1)} \ar[uu] \ar[d]&&
\\
&& \ar[dl]\mathbf{\green(2,-1)}\ar[dr] &&
\\
&\ar[uuu]{(1,0)}\ar[rd]&
& (3,0) \ar[uuu]\ar[ld]
\\
&&\ar[ld] (2,-3)\ar[uu] \ar[rd]&&
\\
&\ar[uu](1,-2) \ar[rd] &&\ar[ld] (3,-2)\ar[uu]
\\
&&\ar[ld] \ar[uu](2,-5) \ar[rd]&&
\\
&\ar[uu](1,-4)  && (3,-4)\ar[uu]
\\
&{}\save[]+<0cm,0ex>*{\vdots}\ar[u]\restore&{}\save[]+<0cm,0ex>*{\vdots}\ar[uu]\restore  
&{}\save[]+<0cm,0ex>*{\vdots}\ar[u]\restore
\\
}
\]
\caption{\label{Fig3} {\it Mutation of $\G_{s_2,\,-1}$ at vertex $(2,-1)$ produces $\G_{s_2,\,1}$.}}
\end{figure}
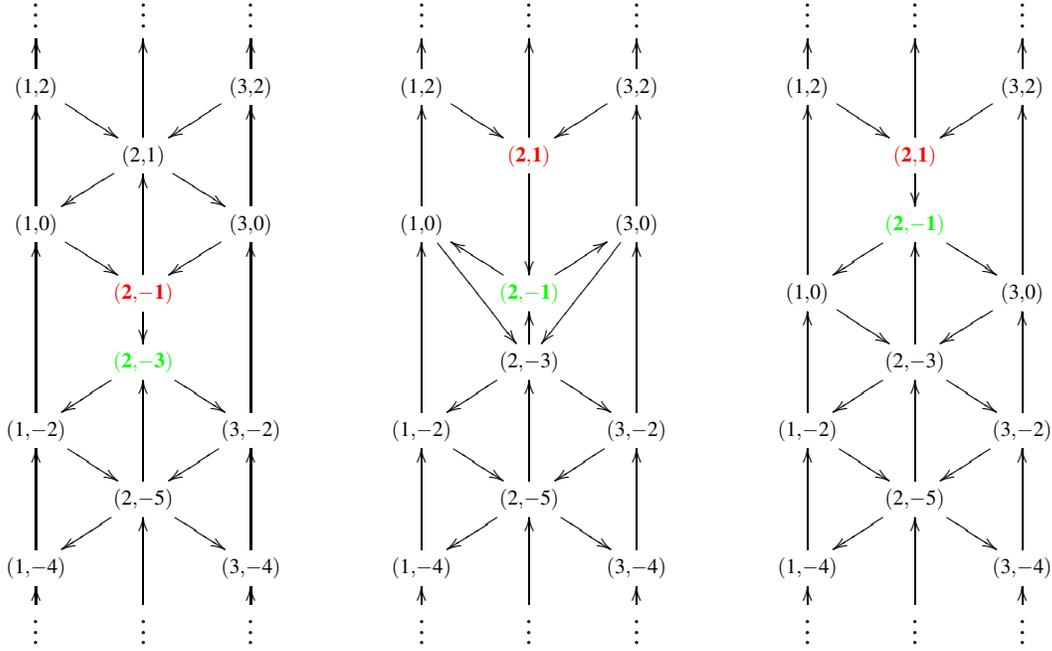

\begin{Lem} \label{Lem22} Let $\g$ be of type $A, D, E$.
\begin{enumerate}
 \item[\rm (i)] If we mutate the quiver $\G_{s_i,\,r}$ at vertex $(i,r)$, 
 then we obtain the quiver $\G_{s_i,\,r+2}$.
 \item[\rm (ii)] If we mutate the quiver $\G_{s_i,\,r}$ at vertex $(i,r-2)$, 
 then we obtain the quiver $\G_{s_i,\,r-2}$.
\end{enumerate}
\end{Lem}

\begin{proof}
This is a straightforward application of the rules of quiver mutation.
Given the explicit description of $\G_{s_i,r}$, we see that we can reduce the proof to an independent verification for every rank 2 root subsystem containing $\a_i$, and this is immediate. An illustration in type $A_3$ is given in Figure~\ref{Fig3}, where the left quiver is $\G_{s_2,\,-1}$,
the middle quiver is obtained by mutation at vertex $(2,-1)$, and
the right quiver is $\G_{s_2,\,1}$, identical to the middle one up to a vertical shift of vertex $(2,-1)$.  \cqfd
\end{proof}

Hence, for a fixed $i\in I$, all the quivers $\G_{s_i,\,r}$ are mutation-equivalent, and they define the same cluster algebra.

\begin{Def}\label{Def2-3}
Let $\g$ be of type $A, D, E$ and let $i\in I$. The cluster algebra $\AA_{s_i}$ is the cluster algebra with quiver $\G_{s_i,\,r}$ for an arbitrary $(i,r)\in V$. 
\end{Def}

\subsection{Type $A, D, E$ with arbitrary $w$}\label{subsec2-3}

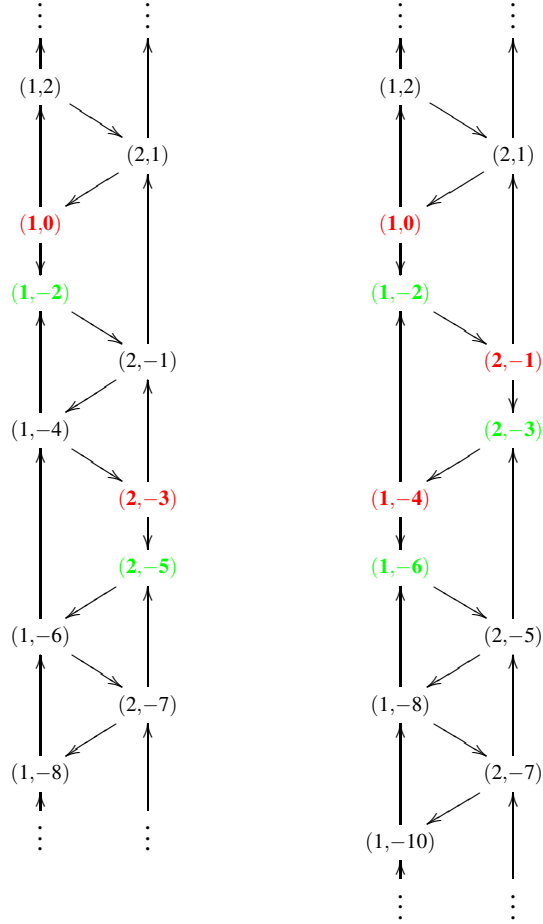
\begin{figure}[t]
\[
\def\objectstyle{\scriptstyle}
\def\lablestyle{\scriptstyle}
\xymatrix@-1.0pc{
&{}\save[]+<0cm,1.5ex>*{\vdots}\restore&{}\save[]+<0cm,1.5ex>*{\vdots}\restore  
\\
&{(1,2)}\ar[rd]\ar[u]&
\\
&&\ar[ld] (2,1) \ar[uu]&&
\\
&\ar[uu]{\mathbf{\red(1,0)}}\ar[d]&
\\
&{\mathbf{\green(1,-2)}}\ar[rd]
\\
&&\ar[uuu]\ar[ld] (2,-1) &&
\\
&\ar[uu](1,-4) \ar[rd] 
\\
&& \ar[uu]\mathbf{\red(2,-3)}\ar[d] &&
\\
&& \ar[ld] \mathbf{\green(2,-5)}
\\
&\ar[uuu](1,-6) \ar[rd] 
\\
&&\ar[ld] \ar[uu](2,-7) &&
\\
&\ar[uu](1,-8) 
\\
&{}\save[]+<0cm,0ex>*{\vdots}\ar[u]\restore&{}\save[]+<0cm,0ex>*{\vdots}\ar[uu]\restore  
\\
}
\quad
\xymatrix@-1.0pc{
&{}\save[]+<0cm,1.5ex>*{\vdots}\restore&{}\save[]+<0cm,1.5ex>*{\vdots}\restore  
\\
&{(1,2)}\ar[rd]\ar[u]&
\\
&&\ar[ld] (2,1) \ar[uu]&&
\\
&\ar[uu]{\red\mathbf{(1,0)}}\ar[d]&
\\
&{\mathbf{\green(1,-2)}}\ar[rd]
\\
&&\ar[uuu] \mathbf{\red(2,-1)}\ar[d] &&
\\
&& \ar[ld]\mathbf{\green(2,-3)}
\\
&\ar[uuu]\mathbf{\red(1,-4)} \ar[d] 
\\
&{\mathbf{\green(1,-6)}}\ar[rd]
\\
&& \ar[uuu](2,-5)\ar[ld] &&
\\
&\ar[uu](1,-8) \ar[rd] 
\\
&&\ar[ld] \ar[uu](2,-7) &&
\\
&\ar[uu](1,-10) 
\\
&{}\save[]+<0cm,0ex>*{\vdots}\ar[u]\restore&{}\save[]+<0cm,0ex>*{\vdots}\ar[uu]\restore  
\\
}
\]
\caption{\label{Fig4} {\it The quivers $\G_{(1,2),(0,-3)}$ and $\G_{(1,2,1),(0,-1,-4)}$ in type $A_2$.}}
\end{figure}

Let $w\in W$, and pick a reduced decomposition $w = s_{i_1}\cdots s_{i_k}$. We construct the quiver of an initial
seed of $\AA_w$ by induction on $k=\ell(w)$.

First, we choose a vertex of $\G$ of the form $(i_1, r_1)\in V$, and construct $\G_{s_{i_1},\, r_1}$ as in \S\ref{subsec2-2}. 
Note that the lower part of $\G_{s_{i_1},\,r_1}$, below vertex $(i_1, r_1 -2)$, is exactly the same as the corresponding part of $\G$ (except for the relabelling of vertices). 
So, we can pick a vertex $(i_2, r_2)\in V$ of $\G_{s_{i_1},\, r_1}$ in this lower part,
and perform at this vertex the same operations as for the construction of $\G_{s_{i_2},\, r_2}$. 
We thus get the new quiver $\G_{(i_1,\,i_2), (r_1,\,r_2)}$. To illustrate, the left quiver in Figure~\ref{Fig4} is the quiver $\G_{(1,2),(0,-3)}$ in type $A_2$. 

Obviously, we can continue in this way and construct successively quivers $\G_{(i_1,\ldots,i_l), (r_1,\ldots, r_l)}$ for $l\le k$. For example, the right quiver in Figure~\ref{Fig4} is the quiver $\G_{(1,2,1),(0,-1,-4)}$ in type $A_2$. 

Our rule for labelling the vertices is that the vertex set is always the set $V$ defined in \S\ref{ssec1-4}, and in each vertical subquiver, the second coordinates are arranged in decreasing order as we go down. If the red vertices are $(i_1,r_1), \ldots, (i_k,r_k)$ in this labelling, then we denote the resulting quiver by $\G_{(i_1,\ldots, i_k),\ (r_1,\ldots, r_k)}$. 
Clearly, for a given reduced decomposition $w = s_{i_1}\cdots s_{i_k}$, we obtain an infinite number
of quivers $\G_{(i_1,\ldots, i_k),\ (r_1,\ldots, r_k)}$ depending on the choice of $(r_1,\ldots, r_k)$.
However, the next proposition shows that they define the same cluster algebra.

\begin{Prop}\label{Prop24}
For a fixed reduced decomposition $w = s_{i_1}\cdots s_{i_k}$, the quivers $\G_{(i_1,\ldots,i_k), (r_1,\ldots, r_k)}$ are all mutation-equivalent. 
\end{Prop}

\begin{proof}
This follows immediately from Lemma~\ref{Lem22}. \cqfd 
\end{proof}

In fact, the next proposition gives a stronger result.

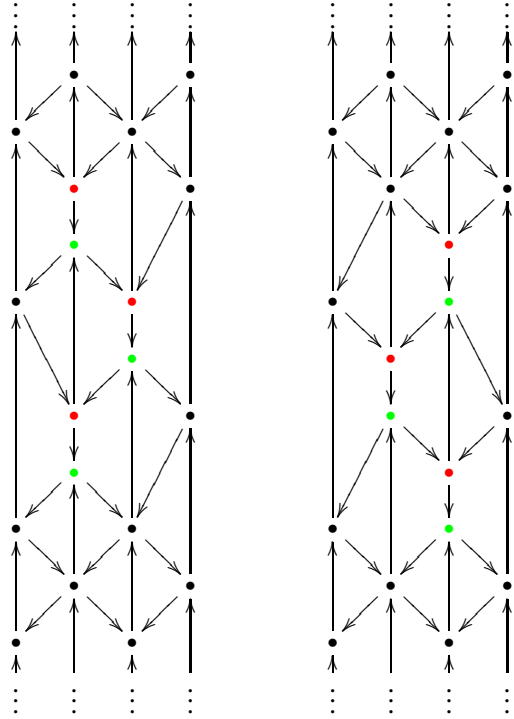
\begin{figure}[t]
\[
\def\objectstyle{\scriptstyle}
\def\lablestyle{\scriptstyle}
\xymatrix@-1.0pc{
{}\save[]+<0cm,1ex>*{\vdots}\restore{}&\save[]+<0cm,1ex>*{\vdots}\restore&{}\save[]+<0cm,1ex>*{\vdots}\restore&{}\save[]+<0cm,1ex>*{\vdots}\restore  
\\
&\bullet\ar[ld]\ar[rd]\ar[u]&& \bullet\ar[ld]\ar[u]
\\
\bullet \ar[uu]\ar[rd]&&\ar[ld] \bullet\ar[rd] \ar[uu]
\\
&\ar[uu]{\red\bullet}\ar[d]&&\ar[ldd]\bullet\ar[uu]
\\
&{\green\bullet}\ar[ld]\ar[rd]
\\
\bullet\ar[uuu]\ar[rdd]&&\ar[uuu] {\red\bullet}\ar[d] 
\\
&& \ar[ld]{\green\bullet}\ar[rd]
\\
&\ar[uuu]{\red\bullet} \ar[d] &&\bullet\ar[uuuu] \ar[ldd]
\\
&{\green\bullet}\ar[ld]\ar[rd]
\\
\bullet\ar[uuuu]\ar[rd]&& \ar[uuu]\bullet\ar[ld]\ar[rd] 
\\
&\ar[uu]\bullet\ar[ld] \ar[rd] &&\bullet\ar[uuu]\ar[ld]
\\
\bullet\ar[uu]&& \ar[uu]\bullet 
\\
{}\save[]+<0cm,0ex>*{\vdots}\ar[u]\restore{}&\save[]+<0cm,0ex>*{\vdots}\ar[uu]\restore&{}\save[]+<0cm,0ex>*{\vdots}\ar[u]\restore&{}\save[]+<0cm,0ex>*{\vdots}\ar[uu]\restore  
\\
}
\qquad\qquad
\xymatrix@-1.0pc{
{}\save[]+<0cm,1ex>*{\vdots}\restore{}&\save[]+<0cm,1ex>*{\vdots}\restore&{}\save[]+<0cm,1ex>*{\vdots}\restore&{}\save[]+<0cm,1ex>*{\vdots}\restore  
\\
&\bullet\ar[ld]\ar[rd]\ar[u]&& \bullet\ar[ld]\ar[u]
\\
\bullet \ar[uu]\ar[rd]&&\ar[ld] \bullet\ar[rd] \ar[uu]
\\
&\ar[uu]\bullet\ar[ldd]\ar[rd]&&\ar[ld]\bullet\ar[uu]
\\
&&\ar[uu] {\red\bullet}\ar[d] 
\\
\bullet\ar[uuu]\ar[rd]&& \ar[ld]{\green\bullet}\ar[rdd] 
\\
& {\red\bullet} \ar[uuu]\ar[d]
\\
&{\green\bullet}\ar[ldd] \ar[dr]&&\bullet\ar[uuuu]\ar[ld]
\\
&& \ar[uuu]\ar[d]{\red\bullet}
\\
\bullet\ar[uuuu]\ar[rd]&& {\green\bullet}\ar[ld]\ar[rd] 
\\
&\ar[uuu]\bullet \ar[ld]\ar[rd] &&\bullet\ar[uuu]\ar[ld]
\\
\bullet\ar[uu]&& \ar[uu]\bullet 
\\
{}\save[]+<0cm,0ex>*{\vdots}\ar[u]\restore{}&\save[]+<0cm,0ex>*{\vdots}\ar[uu]\restore&{}\save[]+<0cm,0ex>*{\vdots}\ar[u]\restore&{}\save[]+<0cm,0ex>*{\vdots}\ar[uu]\restore  
\\
}
\]
\caption{\label{Fig5bis} {\it A local $3$-move.}}
\end{figure}

\begin{figure}[t]
\[
\def\objectstyle{\scriptstyle}
\def\lablestyle{\scriptstyle}
\xymatrix@-1.0pc{
{}\save[]+<0cm,1ex>*{\vdots}\restore&{}\save[]+<0cm,1ex>*{\vdots}\restore  
\\
{(1,2)}\ar[rd]\ar[u]&
\\
&\ar[ld] (2,1) \ar[uu]
\\
\ar[uu]{\mathbf{\red(1,0)}}\ar[d]&
\\
\fbox{$\scriptstyle\mathbf{\green(1,-2)}$}\ar[rd]
\\
&\ar[uuu] \mathbf{\red(2,-1)}\ar[d] 
\\
& \ar[ld]\mathbf{\green(2,-3)}
\\
\ar[uuu]\mathbf{\red(1,-4)} \ar[d] 
\\
{\mathbf{\green(1,-6)}}\ar[rd]
\\
& \ar[uuu](2,-5)\ar[ld] 
\\
\ar[uu](1,-8) \ar[rd] 
\\
&\ar[ld] \ar[uu](2,-7) 
\\
\ar[uu](1,-10) 
\\
{}\save[]+<0cm,0ex>*{\vdots}\ar[u]\restore&{}\save[]+<0cm,0ex>*{\vdots}\ar[uu]\restore  
\\
}
\xymatrix@-1.0pc{
{}\save[]+<0cm,1.5ex>*{\vdots}\restore&{}\save[]+<0cm,1.5ex>*{\vdots}\restore  
\\
{(1,2)}\ar[rd]\ar[u]&
\\
&\ar[ld] (2,1) \ar[uu]
\\
\ar[uu]{\mathbf{(1,0)}}\ar[rdd]&
\\
\mathbf{(1,-2)}\ar[u]\ar[ddd]
\\
&\ar[uuu]\ar[lu] \mathbf{(2,-1)}\ar[d] 
\\
& \ar[ld]\mathbf{(2,-3)}
\\
\fbox{$\scriptstyle\mathbf{(1,-4)}$} \ar[d]\ar[ruu] 
\\
{\mathbf{(1,-6)}}\ar[rd]
\\
& \ar[uuu](2,-5)\ar[ld] 
\\
\ar[uu](1,-8) \ar[rd] 
\\
&\ar[ld] \ar[uu](2,-7) 
\\
\ar[uu](1,-10)
\\
{}\save[]+<0cm,0ex>*{\vdots}\ar[u]\restore&{}\save[]+<0cm,0ex>*{\vdots}\ar[uu]\restore  
\\
}
\xymatrix@-1.0pc{
&{}\save[]+<0cm,1.5ex>*{\vdots}\restore&{}\save[]+<0cm,1.5ex>*{\vdots}\restore  
\\
&{(1,2)}\ar[rd]\ar[u]&
\\
&&\ar[ld] (2,1) \ar[uu]
\\
&\ar[uu]{\mathbf{(1,0)}}\ar[rdd]&
\\
&\mathbf{(1,-2)}\ar[u]\ar@/_{2pc}/[dddd]
\\
&&\ar[uuu]\ar[ldd] \mathbf{(2,-1)} 
\\
&& \fbox{$\scriptstyle\mathbf{(2,-3)}$}\ar[ldd]
\\
&\mathbf{(1,-4)} \ar[uuu]\ar[ru] 
\\
&{\mathbf{(1,-6)}}\ar[rd]\ar[u]
\\
&& \ar[uuu](2,-5)\ar[ld] 
\\
&\ar[uu](1,-8) \ar[rd] 
\\
&&\ar[ld] \ar[uu](2,-7) 
\\
&\ar[uu](1,-10)
\\
&{}\save[]+<0cm,0ex>*{\vdots}\ar[u]\restore&{}\save[]+<0cm,0ex>*{\vdots}\ar[uu]\restore  
\\
}
\xymatrix@-1.0pc{
&{}\save[]+<0cm,1.5ex>*{\vdots}\restore&{}\save[]+<0cm,1.5ex>*{\vdots}\restore  
\\
&{(1,2)}\ar[rd]\ar[u]&
\\
&&\ar[ld] (2,1) \ar[uu]
\\
&\ar[uu]{\mathbf{(1,0)}}\ar[rdd]&
\\
&\mathbf{(1,-2)}\ar[u]\ar@/_{2pc}/[dddd]
\\
&&\ar[uuu]\ar[ldd] \mathbf{(2,-1)} 
\\
&& \ar[ld]\mathbf{(2,-3)}\ar[ddd]
\\
&\mathbf{(1,-4)} \ar[uuu] 
\\
&{\mathbf{(1,-6)}}\ar[ruu]
\\
&& (2,-5)\ar[ld] 
\\
&\ar[uu](1,-8) \ar[rd] 
\\
&&\ar[ld] \ar[uu](2,-7) 
\\
&\ar[uu](1,-10) 
\\
&{}\save[]+<0cm,0ex>*{\vdots}\ar[u]\restore&{}\save[]+<0cm,0ex>*{\vdots}\ar[uu]\restore  
\\
}
\xymatrix@-1.0pc{
{}\save[]+<0cm,1.5ex>*{\vdots}\restore&{}\save[]+<0cm,1.5ex>*{\vdots}\restore  
\\
{(1,2)}\ar[rd]\ar[u]&
\\
&\ar[ld] (2,1) \ar[uu]
\\
\ar[uu](1,0)\ar[rd]&
\\
&\ar[uu] \mathbf{\red(2,-1)}\ar[d] 
\\
& \ar[ld]\mathbf{\green(2,-3)} 
\\
 \mathbf{\red(1,-2)} \ar[uuu]\ar[d]
\\
\mathbf{\green(1,-4)} \ar[dr]
\\
& \ar[uuu]\ar[d]\mathbf{\red(2,-5)}
\\
& \mathbf{\green(2,-7)}\ar[ld] 
\\
\ar[uuu](1,-6) \ar[rd] 
\\
&\ar[ld] \ar[uu](2,-9) 
\\
\ar[uu](1,-8) 
\\
{}\save[]+<0cm,0ex>*{\vdots}\ar[u]\restore&{}\save[]+<0cm,0ex>*{\vdots}\ar[uu]\restore  
\\
}
\]
\caption{\label{Fig5} {\it A sequence of 3 mutations from $\G_{(1,2,1),(0,-1,-4)}$ to $\G_{(2,1,2),(-1,-2,-5)}$.}}
\end{figure}
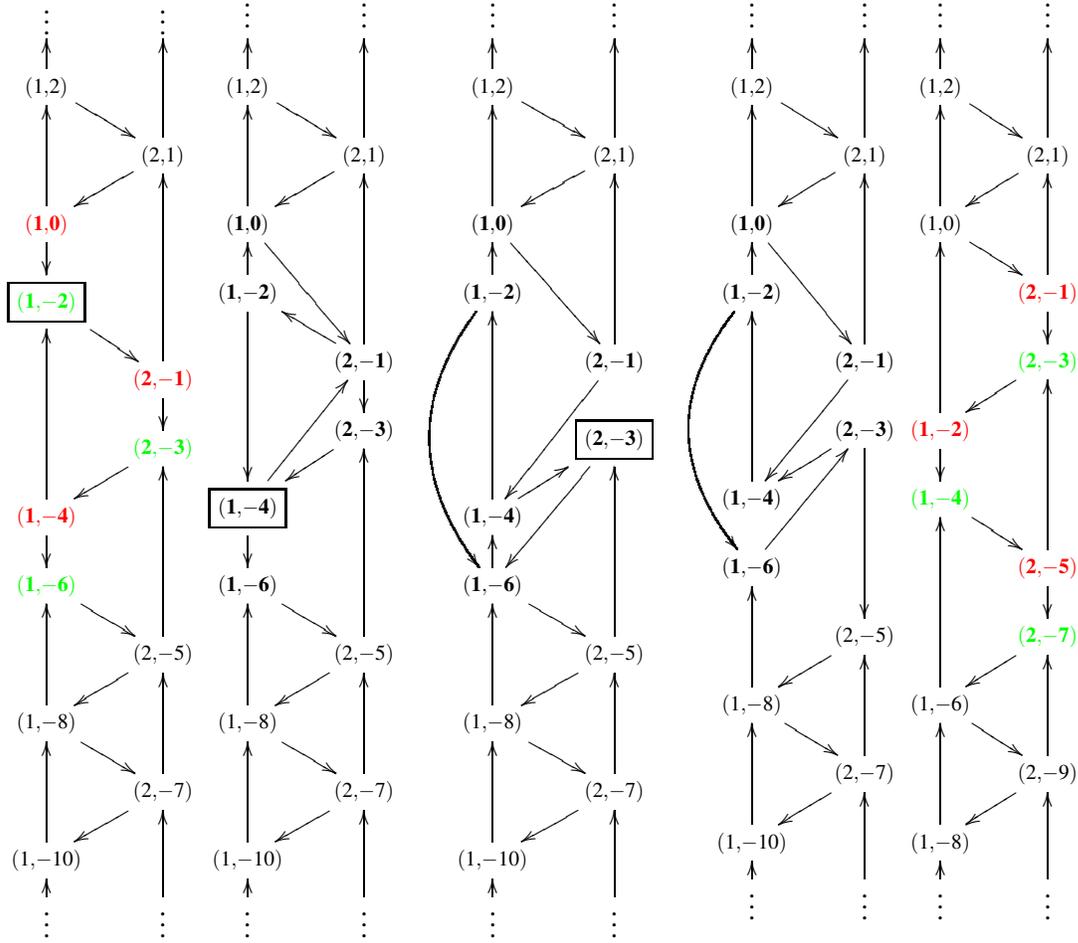

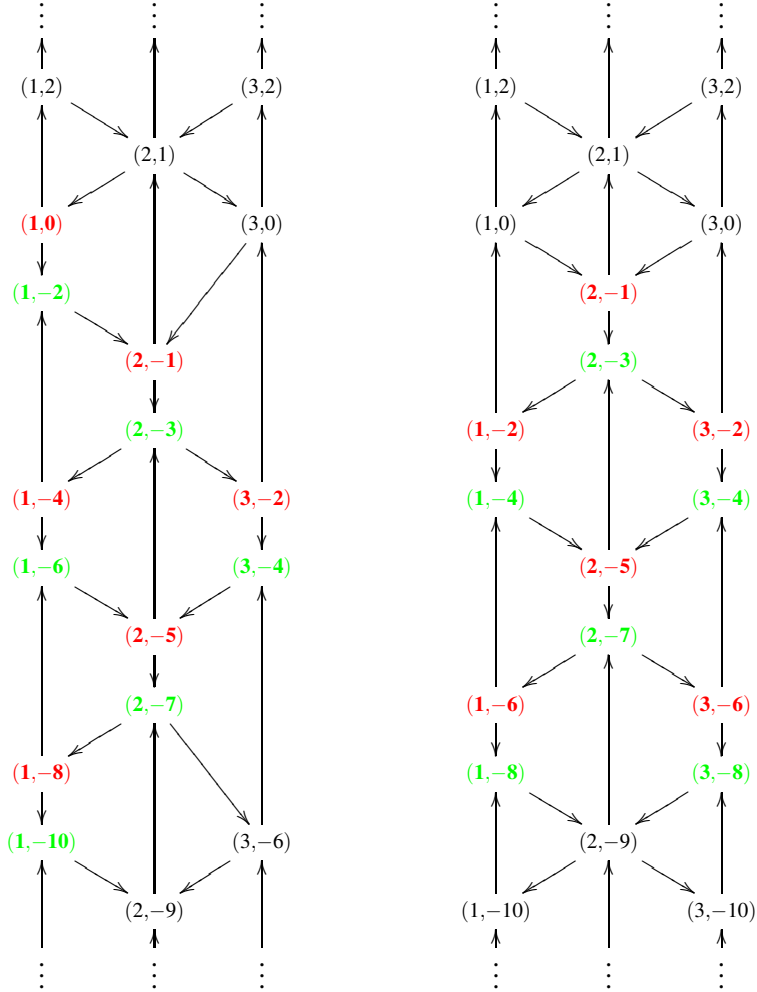
\begin{figure}[t]
\[
\def\objectstyle{\scriptstyle}
\def\lablestyle{\scriptstyle}
\xymatrix@-1.0pc{
&{}\save[]+<0cm,1.5ex>*{\vdots}\restore&{}\save[]+<0cm,1.5ex>*{\vdots}\restore  
&{}\save[]+<0cm,1.5ex>*{\vdots}\restore
\\
&{(1,2)}\ar[rd]\ar[u]&
&\ar[ld] (3,2) \ar[u]
\\
&&\ar[ld] (2,1) \ar[rd]\ar[uu]&&
\\
&{\ar[uu]}{\red\mathbf{(1,0)}}\ar[d]&&\ar[ldd] (3,0) \ar[uu]
\\
&{\green\mathbf{(1,-2)}}\ar[rd]
\\
&&\ar[uuu] \mathbf{\red(2,-1)}\ar[d]&&
\\
&&\ar[ld] \mathbf{\green(2,-3)} \ar[rd]&&
\\
&\ar[uuu]\mathbf{\red(1,-4)} \ar[d] && \ar[d] \mathbf{\red(3,-2)}\ar[uuuu]
\\
&\mathbf{\green(1,-6)} \ar[rd] &&\ar[ld] \mathbf{\green(3,-4)}
\\
&& \ar[uuu]\mathbf{\red(2,-5)}\ar[d] &&
\\
&&\ar[ld]\mathbf{\green(2,-7)} \ar[rdd]&&
\\
&\ar[uuu]\mathbf{\red(1,-8)}\ar[d]  
\\
&\mathbf{\green(1,-10)} \ar[rd]&&\ar[ld] (3,-6)\ar[uuuu]
\\
&& \ar[uuu](2,-9) &&
\\
&{}\save[]+<0cm,0ex>*{\vdots}\ar[uu]\restore&{}\save[]+<0cm,0ex>*{\vdots}\ar[u]\restore  
&{}\save[]+<0cm,0ex>*{\vdots}\ar[uu]\restore
\\
}
\qquad
\xymatrix@-1.0pc{
&{}\save[]+<0cm,1.5ex>*{\vdots}\restore&{}\save[]+<0cm,1.5ex>*{\vdots}\restore  
&{}\save[]+<0cm,1.5ex>*{\vdots}\restore
\\
&{(1,2)}\ar[rd]\ar[u]&
&\ar[ld] (3,2) \ar[u]
\\
&&\ar[ld] (2,1) \ar[rd]\ar[uu]&&
\\
&\ar[uu]{(1,0)}\ar[rd]&
&\ar[ld] (3,0) \ar[uu]
\\
&&\ar[uu] \mathbf{\red(2,-1)}\ar[d] &&
\\
&&\ar[ld]\ar[rd] \mathbf{\green(2,-3)} &&
\\
&\ar[uuu]\mathbf{\red(1,-2)}\ar[d]  && \mathbf{\red(3,-2)}\ar[uuu]\ar[d]
\\
&\mathbf{\green(1,-4)} \ar[rd] &&\ar[ld] \mathbf{\green(3,-4)}
\\
&& \ar[d]\ar[uuu]\mathbf{\red(2,-5)} &&
\\
&&\ar[ld] \mathbf{\green(2,-7)} \ar[rd]&&
\\
&\ar[uuu]\mathbf{\red(1,-6)}\ar[d]  && \mathbf{\red(3,-6)}\ar[d]\ar[uuu]
\\
&\mathbf{\green(1,-8)} \ar[rd] &&\ar[ld] \mathbf{\green(3,-8)}
\\
&& \ar[uuu](2,-9) \ar[ld]\ar[rd]&&
\\
&\ar[uu](1,-10) &&\ar[uu] (3,-10) 
\\
&{}\save[]+<0cm,0ex>*{\vdots}\ar[u]\restore&{}\save[]+<0cm,0ex>*{\vdots}\ar[uu]\restore  
&{}\save[]+<0cm,0ex>*{\vdots}\ar[u]\restore
\\
}
\]
\caption{\label{Fig6} {\it Two initial seeds for the cluster algebra $\AA_{w_0}$ in type $A_3$.}}
\end{figure}

\begin{Prop}\label{Prop25}
The cluster algebra defined by a quiver $\G_{(i_1,\ldots,i_k), (r_1,\ldots, r_k)}$ depends only on the Weyl group element $w$. 
\end{Prop}

\begin{proof}
Since $\g$ is assumed to be of type $A, D, E$, the braid relations satisfied by the generators of $W$ are all of the form 
\[
s_is_js_i = s_js_is_j,\ (c_{ij} = -1),\quad\mbox{or}\quad s_is_k = s_ks_i,\ (c_{ik} = 0).
\]
We need to show that if two words $(i_1,\ldots,i_k)$ and $(j_1,\ldots,j_k)$ are related by one of these two types of braid moves, then the corresponding quivers $\G_{(i_1,\ldots,i_k), (r_1,\ldots, r_k)}$ and 
$\G_{(j_1,\ldots,j_k), (s_1,\ldots, s_k)}$ are related by a finite sequence of mutations.

Suppose that for some $1\le l\le k-2$ we have $i_l=i_{l+2}=i$, $i_{l+1}=j$ with $c_{ij}=-1$.
By Proposition~\ref{Prop24}, we see that, mutating only at red or green vertices, we can freely move up or down any of the $k$ pairs of red and green vertices of $\G_{(i_1,\ldots,i_k), (r_1,\ldots, r_k)}$ corresponding to the indices $i_1,\ldots,i_k$, as long as we keep them in the same order when reading from top to bottom. By this procedure we can perform a sequence of mutations on  
$\G_{(i_1,\ldots,i_k), (r_1,\ldots, r_k)}$ which will isolate the red and green vertices corresponding to $i_l, i_{l+1}, i_{l+2}$ so that locally, the mutated quiver looks like the left quiver of Figure~\ref{Fig5bis}.

We claim that this left quiver  can be transformed into the right quiver of Figure~\ref{Fig5bis} by a sequence of 3 mutations. This sequence is exhibited in Figure~\ref{Fig5}, where we have restricted ourselves to quivers of type $A_2$ to save space, and we have labelled the vertices to improve readability.
The mutations are performed at vertices $(1,-2)$, $(1,-4)$, $(2,-3)$, successively. The result of this sequence of 3 mutations is the 4th quiver of Figure~\ref{Fig5}, and the 5th quiver is the target quiver $\G_{(2,1,2),(-1,-2,-5)}$. One can see that the 4th and 5th quivers are identical up to the change of labelling:
\[
 (1,-4) \mapsto (2,-3),
 \qquad (1,a)\mapsto (1,a+2),\quad (a\le -6),
 \qquad (2,b)\mapsto (2,b-2),\quad (b\le -3).
\]
This change of labelling seems artificial at this stage, but we will see in Remark~\ref{rem-mut-g-vector} below that this problem disappears when we use the canonical labelling given by ``stabilized $g$-vectors''. It is easy to check that the same sequence of mutations transforms the left quiver of Figure~\ref{Fig5bis} into the right quiver, that is, that the arrows connecting the central part to the two sides also mutate as expected.

Finally, the case of the trivial braid moves $s_is_k = s_ks_i,\ (c_{ik} = 0)$ follows again immediately 
from Lemma~\ref{Lem22}.
\cqfd 
\end{proof}

\begin{Def}\label{Def3-6}
The cluster algebra with defining quiver $\G_{(i_1,\ldots,i_k), (r_1,\ldots, r_k)}$ is denoted by $\AA_w$.  
\end{Def}

For instance, in Figure~\ref{Fig6} we display two initial seeds of the cluster algebra $\AA_{w_0}$ in type $A_3$, corresponding respectively to the two reduced expressions $s_1s_2s_1s_3s_2s_1$ and
$s_2s_1s_3s_2s_1s_3$ of $w_0$. 

\begin{remark} \label{Remark3-7}
{\rm
As noted above in the case of type $A_1$, by Proposition~\ref{Prop24} it is possible to find an \emph{infinite} sequence of mutations transforming the quiver of an initial seed of $\AA_w$ into the basic infinite quiver $\G$, that is, into the quiver of an initial seed of $\AA_e$. More precisely, mutating once at each green vertex of an initial seed of $\AA_w$ produces a new seed whose quiver is a one step downward translation of the initial one. Thus, repeating an infinite number of times this operation we can push down to infinity the middle part consisting of red and green vertices and recover the quiver~$\G$, with its canonical labelling of vertices.
}
\end{remark}

\subsection{Initial seeds of $\AA_{w_0}$ associated with a Coxeter element}\label{subsec-Coxeter}

\begin{figure}[t]
\[
\def\objectstyle{\scriptstyle}
\def\lablestyle{\scriptstyle}
\xymatrix@-1.0pc{
&{}\save[]+<0cm,1.5ex>*{\vdots}\restore&{}\save[]+<0cm,1.5ex>*{\vdots}\restore  
&{}\save[]+<0cm,1.5ex>*{\vdots}\restore
\\
&{(1,2)}\ar[rd]\ar[u]&
&\ar[ld] (3,2) \ar[u]
\\
&&\ar[ld] (2,1) \ar[rd]\ar[uu]&&
\\
&\ar[uu]\mathbf{\red(1,0)}\ar[rd]&
&\ar[ld] (3,0) \ar[uu]
\\
&&\ar[uu]\ar[ld] \mathbf{\red(2,-1)} \ar[rd]&&
\\
&\ar[uu]\mathbf{\red(1,-2)} \ar[rd] &&\ar[ld] \mathbf{\red(3,-2)}\ar[uu]
\\
&&\ar[ld] \ar[uu]\mathbf{\red(2,-3)} \ar[rd]&&
\\
&\ar[uu]\mathbf{\red(1,-4)} \ar[rd] &&\ar[ld] (3,-4)\ar[uu]
\\
&&\ar[ld] \ar[uu](2,-5) \ar[rd]&&
\\
&\ar[uu](1,-6) &&\ar[uu] (3,-6) 
\\
&{}\save[]+<0cm,0ex>*{\vdots}\ar[u]\restore&{}\save[]+<0cm,0ex>*{\vdots}\ar[uu]\restore  
&{}\save[]+<0cm,0ex>*{\vdots}\ar[u]\restore
\\
}
\qquad
\def\objectstyle{\scriptstyle}
\def\lablestyle{\scriptstyle}
\xymatrix@-1.0pc{
&{}\save[]+<0cm,1.5ex>*{\vdots}\restore&{}\save[]+<0cm,1.5ex>*{\vdots}\restore  
&{}\save[]+<0cm,1.5ex>*{\vdots}\restore
\\
&{(1,2)}\ar[rd]\ar[u]&
&\ar[ld] (3,2) \ar[u]
\\
&&\ar[ld] (2,1) \ar[rd]\ar[uu]&&
\\
&\ar[uu]{(1,0)}\ar[rd]&
&\ar[ld] (3,0) \ar[uu]
\\
&&\ar[uu]\ar[ld] \mathbf{\red(2,-1)} \ar[rd]&&
\\
&\ar[uu]\mathbf{\red(1,-2)} \ar[rd] &&\ar[ld] \mathbf{\red(3,-2)}\ar[uu]
\\
&&\ar[ld] \ar[uu]\mathbf{\red(2,-3)} \ar[rd]&&
\\
&\ar[uu]\mathbf{\red(1,-4)} \ar[rd] &&\ar[ld] \mathbf{\red(3,-4)}\ar[uu]
\\
&&\ar[ld] \ar[uu](2,-5) \ar[rd]&&
\\
&\ar[uu](1,-6) &&\ar[uu] (3,-6) 
\\
&{}\save[]+<0cm,0ex>*{\vdots}\ar[u]\restore&{}\save[]+<0cm,0ex>*{\vdots}\ar[uu]\restore  
&{}\save[]+<0cm,0ex>*{\vdots}\ar[u]\restore
\\
}
\]
\caption{\label{Fig7b} {\it The subquivers $G_c$ corresponding to the Coxeter elements $s_1s_2s_3$ and $s_2s_1s_3$ in type $A_3$.}}
\end{figure}
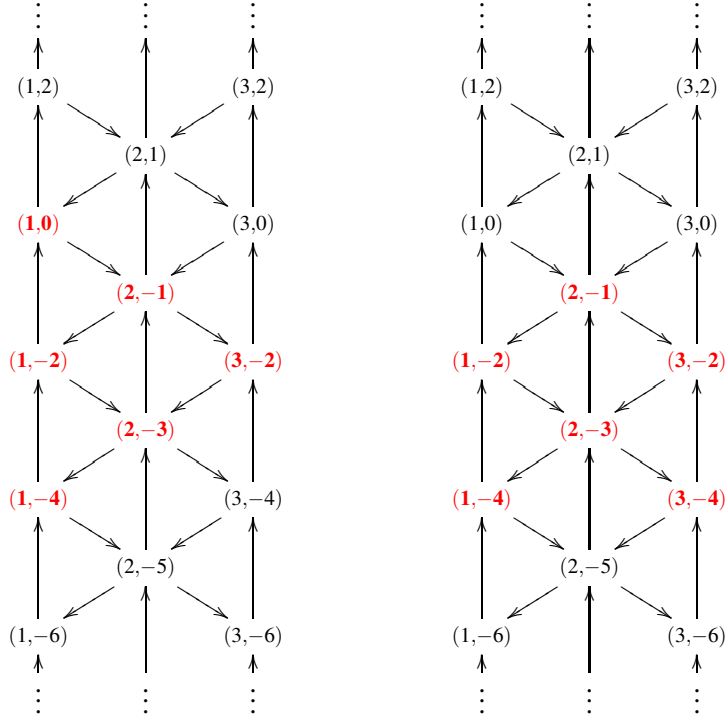

Let $\cQ$ be an orientation of the Dynkin diagram of $\g$. In other words, $\cQ$ is a Dynkin quiver of the
same Dynkin type as $\g$.
For $i \in I$, we denote by $s_i(\cQ)$ the quiver obtained from $\cQ$ by changing the orientation of every
arrow with source $i$ or target $i$.  Let $w = s_{i_1}\cdots s_{i_k} \in W$ be a reduced decomposition. We say that
$\bi = (i_1,\ldots , i_k)$ is \emph{adapted} to $\cQ$ if and only if $i_k$ is a source of $\cQ$, $i_{k-1}$ is a source of $s_{i_k} (\cQ)$, $\ldots$\,, $i_1$ is a source of $s_{i_{2}}\cdots s_{i_k} (\cQ)$. There is a unique Coxeter element having reduced expressions adapted to $\cQ$. We shall denote it by $c$. Moreover the map $\cQ \mapsto c$ is a bijection between orientations $\cQ$ of the Dynkin diagram and Coxeter elements of $W$.

Let $c = s_{i_1}\cdots s_{i_n}$ denote a fixed Coxeter element of $W$, and let $\cQ$ be the corresponding Dynkin quiver. 
Then the basic infinite quiver $\Gamma_C$ is isomorphic to the Auslander-Reiten quiver of the bounded derived category $D^b(K\cQ)$ of the path algebra $K\cQ$ over a field $K$, in which we have added arrows corresponding to the action of the Auslander-Reiten translation (these are all the vertical arrows $(i,r) \to (i,r+2)$ of $\Gamma_C$). 

We can therefore identify the Auslander-Reiten quiver of the abelian category $\mod(K\cQ)$ with a finite connected full subquiver $G_c$ of $\Gamma_C$, whose number of vertices is equal to the number of positive roots of $\g$. 
This is illustrated in type $A_3$ in Figure~\ref{Fig7b}. The two Coxeter elements $s_1s_2s_3$ and $s_2s_1s_3$ correspond to the two full subquivers $G_c$ with vertices painted in red. 

It is then easy to formulate an alternative recipe for obtaining the quiver of an initial seed of $\AA_{w_0}$. 
Pick a Coxeter element $c$ and determine the subquiver $G_c$ of $\G_C$. Then at each vertex of $G_c$ perform the procedure described in Definition~\ref{Def2-1}. The resulting quiver $\G_c$ is the quiver of an initial seed of $\AA_{w_0}$. For example, applying this recipe to the quivers of Figure~\ref{Fig7b}, we get the quivers of Figure~\ref{Fig6}. 

Strictly speaking, $\G_c$ is only determined up to a vertical translation, but because of Proposition~\ref{Prop24} this does not affect the mutation class of the quivers obtained by means of this procedure.

\section{Stabilized $g$-vectors}\label{sect-stable-g}

If we choose a reference seed, we can parametrize cluster variables of $\AA_{w_0}$ using their $g$-vectors with respect to this reference seed. This parametrization depends on the choice of reference seed. 
For instance, we can choose a reference seed coming from a Coxeter element $c$, as described in \S \ref{subsec-Coxeter}. 
But, as noted above, this is only determined up to a vertical translation. If we mutate this reference seed up or down, the $g$-vector of a given cluster variable will change. 

In this section we will show that if we consider the sequence of $g$-vectors of a fixed cluster variable $x$ with respect to a sequence of such reference seeds obtained by successive downward translations, then this sequence of $g$-vectors stabilizes after a finite number of downward translations. The stabilized limit $g$-vector can be regarded as the $g$-vector of $x$ with respect to the limit reference seed obtained by translating downward a given Coxeter initial seed an infinite number of times. As noted above (see Remark~\ref{Remark3-7}), the quiver of this limit of seeds is the basic quiver $\G_C$, which is independent of the choice of Coxeter element $c$. After studying this limiting procedure in detail, we will give explicit descriptions of stabilized $g$-vectors of the cluster variables of our Coxeter-type initial seeds.

\subsection{$g$-vectors}\label{subsec4-1} 
We start by recalling the basic facts about $g$-vectors which we will need. 

Let $x$ denote the $k$th cluster variable in a seed $S$, and let $\Sigma$ be another seed. In \cite{FZ-CA4}, Fomin and Zelevinsky have given two recursive relations for calculating the $g$-vector of $x$ with respect to the reference seed $\Sigma$. Namely, we can mutate at $x$ in the seed $S$ and get a new cluster variable $\mu_k(x)$, or mutate $\Sigma$ in direction $l$ and get a new reference seed $\mu_l(\Sigma)$. Then the first recursion expresses the $g$-vector of $x$ with respect to $\mu_l(\Si)$ in terms of the $g$-vector of $x$ with respect to $\Si$, and the second recursion expresses the $g$-vector of $\mu_k(x)$ with respect to $\Sigma$ in terms of the $g$-vector of $x$ with respect to $\Sigma$.

To state these relations we will need some more notation. Let $\G$ denote the quiver of the seed $\Si$, and let $V$ be its vertex set. Let $l\in V$ and let $\Si' := \mu_l(\Si)$. 
We denote by $\mathbf{g} := (g_v)_{v\in V}$ (\resp $\mathbf{g'} := (g'_v)_{v\in V}$) the $g$-vector of the cluster variable $x$ with respect to $\Si$ (\resp with respect to $\Si'$). Then we have
\begin{equation}\label{recursion1}
 g'_v = 
\left\{
\begin{array}{ll}
-g_v & \mbox{if $v = l$},
\\[1mm]
g_v\, +\, N_{l \to v}\, g_l & \mbox{if $v\not = l$ and $g_l \ge 0$},
\\[1mm]
g_v\, +\, N_{v \to l}\, g_l & \mbox{if $v\not = l$ and $g_l \le 0$},
\end{array}
\right. 
\end{equation}
where for $a,b\in V$, we denote by $N_{a\to b}$ the number of arrows from $a$ to $b$ in $\Gamma$. 
This relation is
\cite[Conjecture 7.12]{FZ-CA4}, and it
is now a theorem by \cite[Theorem 1.7]{DWZ2}. 

To state the second recursion relation, we also need $c$-vectors (see \cite{NZ}). Let $\mathbf{c} := (c_v)_{v\in V}$ denote the $c$-vector of $x$ with respect to $\Si$. 
By \cite[Theorem 1.7]{DWZ2} and \cite{NZ}, it is known that the coordinates $c_v$ of $\mathbf{c}$ are either all non-negative, or all non-positive. 
Accordingly we can write unambiguously $\mathbf{c} \ge 0$, or $\mathbf{c} \le 0$.
Let $x^* := \mu_k(x)$, and let 
$\mathbf{g^*}$ denote the $g$-vector of $x^*$ with respect to $\Si$.
Let $\mathbf{g}^v\,(v\not = k)$ be the $g$-vectors with respect to $\Si$ of the cluster variables $x_v$ of $S$ other than $x = x_k$. 
Then by \cite[Proposition 6.6, Equations 6.12, 6.13]{FZ-CA4}
we have, 
\begin{equation}\label{recursion2}
\mathbf{g^*} =
\left\{
\begin{array}{ll}
-\mathbf{g} +\, \sum_{v\in V} M_{v \to k}\, \mathbf{g}^v & \mbox{if $\mathbf{c} \ge 0$},
\\[2mm]
-\mathbf{g} +\, \sum_{v\in V} M_{k \to v}\, \mathbf{g}^v & \mbox{if $\mathbf{c} \le 0$},
\end{array}
\right. 
\end{equation}
where for $a,b\in V$, we denote by $M_{a \to b}$ the number of arrows from $a$ to $b$ in the quiver $G$ of the seed $S$. 

Finally, to determine in practice whether $\mathbf{c} \ge 0$ or $\mathbf{c} \le 0$, we can make use of 
\cite[Equation 2.8]{NZ}, which implies that
\begin{equation}\label{Eq4}
\sum_{v\in V} \left(M_{v \to k} - M_{k \to v}\right)\, \mathbf{g}^v = \sum_{v\in V} c_v\, \mathbf{b}_v^0,
\end{equation}
where $\mathbf{b}_v^0 := \sum_{l\in V}\left(N_{l \to v} - N_{v\to l}\right) \mathbf{e}_l$ are the column vectors of the exchange matrix $B^0$ of $\Si$. Indeed the left-hand side of (\ref{Eq4}) can be calculated from the datum of $S$, and by expanding this sum as a linear combination of the known columns of $B^0$, we can calculate the coefficients $c_v$ and check their sign.  

\subsection{An infinite sequence of mutations}\label{subsec4-2}

\begin{figure}[t]
\[
\def\objectstyle{\scriptstyle}
\def\lablestyle{\scriptstyle}
\xymatrix@-1.0pc{
&{}\save[]+<0cm,1.5ex>*{\vdots}\restore
\\
&
\\
&\mathbf{e}_{(1,2)}\ar[uu]
\\
&
\\
&\ar[uu]{\mathbf{\red e_{(1,0)}}}\ar[dd]
\\
&
\\
&\mathbf{\green e_{(1,-2)}} 
\\
&
\\
&\ar[uu] \mathbf{e}_{(1,-4)} 
\\
&
\\
&\ar[uu]\mathbf{e}_{(1,-6)} 
\\
&
\\
&\ar[uu]\mathbf{e}_{(1,-8)} 
\\
&
\\
&{}\save[]+<0cm,0ex>*{\vdots}\ar[uu]\restore
\\
}
\xymatrix@-1.0pc{
&{}\save[]+<0cm,1.5ex>*{\vdots}\restore
\\
&
\\
&{(1,2)}\ar[uu]
\\
&
\\
&\ar[uu]{\mathbf{\red {(1,0)}}}\ar[dd]
\\
&
\\
&\mathbf{\green {(1,-2)}} 
\\
&
\\
&\ar[uu] {(1,-4)} 
\\
&
\\
&\ar[uu]{(1,-6)} 
\\
&
\\
&\ar[uu]{(1,-8)} 
\\
&
\\
&{}\save[]+<0cm,0ex>*{\vdots}\ar[uu]\restore
\\
}
\quad
\xymatrix@-1.0pc{
&{}\save[]+<0cm,1.5ex>*{\vdots}\restore
\\
&
\\
&\mathbf{e}_{(1,2)}\ar[uu]
\\
&
\\
&\ar[uu]{\mathbf{\red e_{(1,0)}}}\ar[dd]
\\
&
\\
&\mathbf{\green -e_{(1,-2)}} 
\\
&
\\
&\ar[uu] \mathbf{e}_{(1,-4)} 
\\
&
\\
&\ar[uu]\mathbf{e}_{(1,-6)} 
\\
&
\\
&\ar[uu]\mathbf{e}_{(1,-8)} 
\\
&
\\
&{}\save[]+<0cm,0ex>*{\vdots}\ar[uu]\restore
\\
}
\xymatrix@-1.0pc{
&{}\save[]+<0cm,1.5ex>*{\vdots}\restore
\\
&
\\
&{(1,2)}\ar[uu]
\\
&
\\
&\ar[uu](1,0)
\\
\\
&\mathbf{\red(1,-2)}\ar[uu]\ar[dd]
\\
&
\\
&\mathbf{\green(1,-4)} 
\\
&
\\
&\ar[uu](1,-6) 
\\
&
\\
&\ar[uu](1,-8) 
\\
&
\\
&{}\save[]+<0cm,0ex>*{\vdots}\ar[uu]\restore
\\
}
\quad
\xymatrix@-1.0pc{
&{}\save[]+<0cm,1.5ex>*{\vdots}\restore
\\
&
\\
&\mathbf{e}_{(1,2)}\ar[uu]
\\
&
\\
&\ar[uu]{\mathbf{\red e_{(1,0)}}}\ar[dd]
\\
&
\\
&\mathbf{\green -e_{(1,-2)}} 
\\
&
\\
&\ar[uu] \mathbf{-e}_{(1,-4)} 
\\
&
\\
&\ar[uu]\mathbf{e}_{(1,-6)} 
\\
&
\\
&\ar[uu]\mathbf{e}_{(1,-8)} 
\\
&
\\
&{}\save[]+<0cm,0ex>*{\vdots}\ar[uu]\restore
\\
}
\xymatrix@-1.0pc{
&{}\save[]+<0cm,1.5ex>*{\vdots}\restore
\\
&
\\
&{(1,2)}\ar[uu]
\\
&
\\
&\ar[uu]{(1,0)}
\\
&
\\
&\ar[uu](1,-2)\ar[uu]
\\
\\
&\mathbf{ \red(1,-4)}\ar[uu]\ar[dd]
\\
&
\\
&\mathbf{ \green(1,-6)} 
\\
&
\\
&\ar[uu](1,-8) 
\\
&
\\
&{}\save[]+<0cm,0ex>*{\vdots}\ar[uu]\restore
\\
}
\quad
\xymatrix@-1.0pc{
&{}\save[]+<0cm,1.5ex>*{\vdots}\restore
\\
&
\\
&\mathbf{e}_{(1,2)}\ar[uu]
\\
&
\\
&\ar[uu]{\mathbf{\red e_{(1,0)}}}\ar[dd]
\\
&
\\
&\mathbf{\green -e_{(1,-2)}} 
\\
&
\\
&\ar[uu] \mathbf{-e}_{(1,-4)} 
\\
&
\\
&\ar[uu] \mathbf{-e}_{(1,-6)} 
\\
&
\\
&\ar[uu]\mathbf{e}_{(1,-8)} 
\\
&
\\
&{}\save[]+<0cm,0ex>*{\vdots}\ar[uu]\restore
\\
}
\xymatrix@-1.0pc{
&{}\save[]+<0cm,1.5ex>*{\vdots}\restore
\\
&
\\
&{(1,2)}\ar[uu]
\\
&
\\
&\ar[uu]{(1,0)}
\\
&
\\
&\ar[uu](1,-2)\ar[uu]
\\
\\
&(1,-4)\ar[uu]
\\
&
\\
&\mathbf{ \red(1,-6)}\ar[uu]\ar[dd] 
\\
&
\\
&\mathbf{ \green(1,-8)}
\\
&
\\
&{}\save[]+<0cm,0ex>*{\vdots}\ar[uu]\restore
\\
}
\]
\caption{\label{Fig7c} {\it The first steps of the sequence of $g$-vectors in type $A_1$.}}
\end{figure}
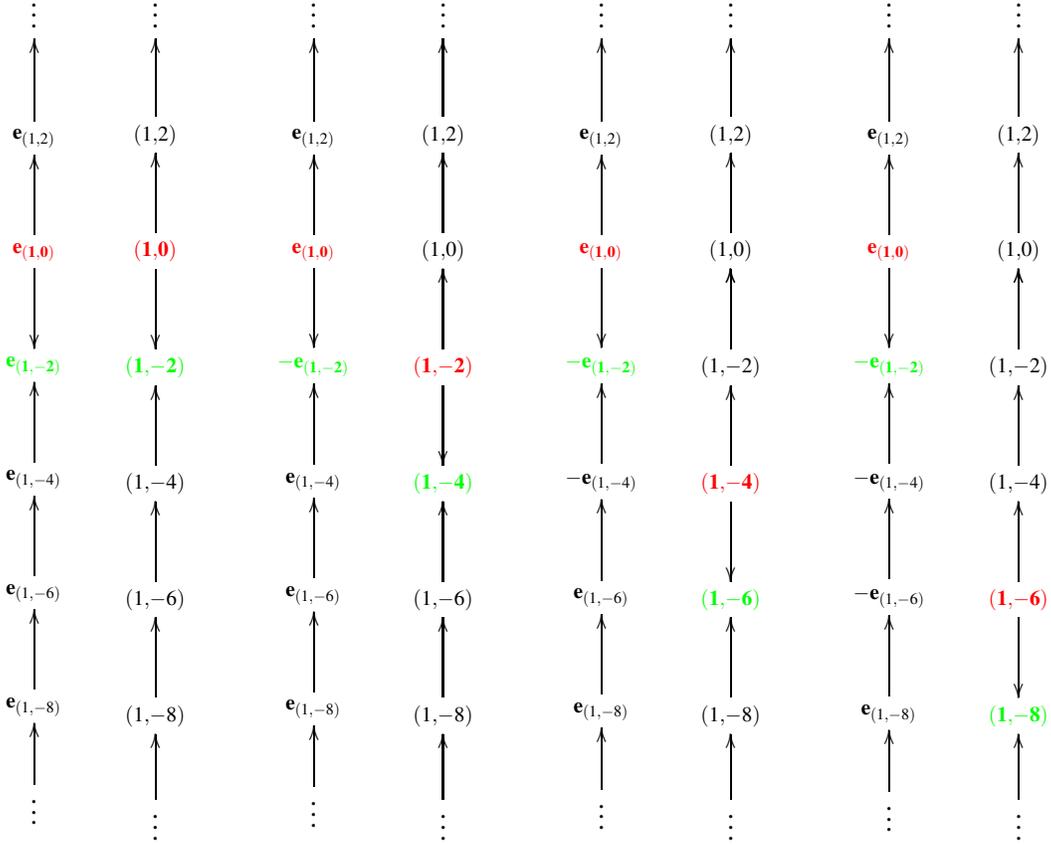

From now on, we fix a Coxeter element $c$, and we consider the quiver $\G := \G_c$ of an initial seed $\Si$ of $\AA_{w_0}$ associated with $c$, as described in \S\ref{subsec-Coxeter}. This seed is only defined up to a vertical translation, so for definiteness, let us decide that the highest red vertex of $\Gamma$ is of the form $(i,0)$, as in the left quiver of Figure~\ref{Fig6}.   

Then we define the following infinite sequence of mutations. We first mutate at all green vertices of $\Gamma$, that is, at all sinks of the vertical subquivers of type $A_\infty$ of $\Gamma$. Here the mutation order is irrelevant because all these mutations commute with each other. Let us denote by $\G^{(1)}$ the quiver of the new seed $\Si^{(1)}$ obtained after this finite sequence of $\ell(w_0)$ mutations. Then, by Lemma~\ref{Lem22}, the quiver $\G^{(1)}$ is obtained from $\G$ by a one-step downward translation. So we can iterate this procedure and mutate at every green vertex of $\G^{(1)}$ to get a new seed $\Si^{(2)}$ whose quiver $\G^{(2)}$ is obtained from $\G$ by a two-step downward translation. And so on. In this way we obtain an infinite sequence $\Si^{(m)}\, (m \ge 0)$ of seeds of $\AA_{w_0}$, with quivers $\G^{(m)}$ such that $\lim_{m\to\infty} \G^{(m)} = \G_{C}$.

Our aim is to study the sequences of $g$-vectors of all cluster variables of $\Si$ with respect to the sequence of reference seeds $\Si^{(m)}$.

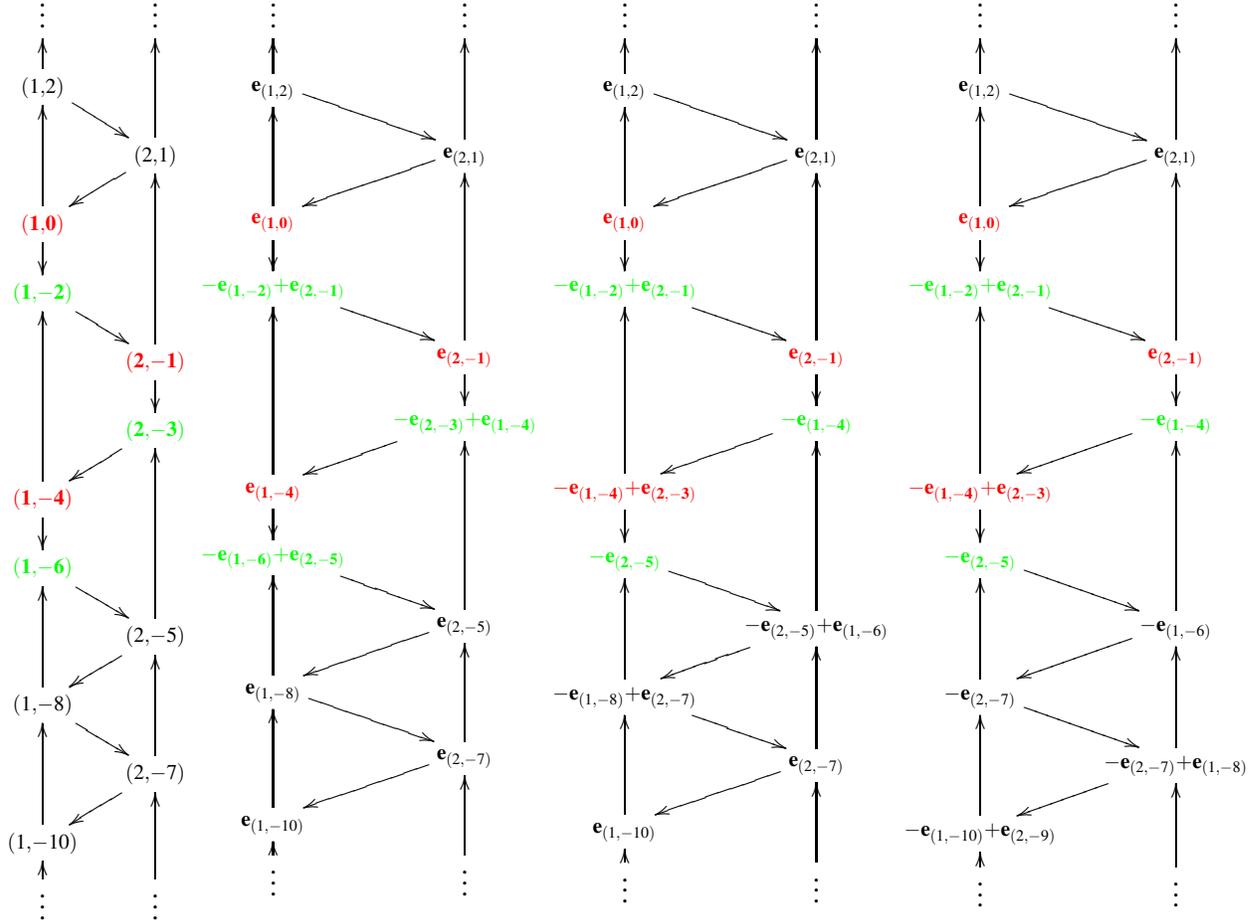
\begin{figure}[t]
\[
\def\objectstyle{\scriptstyle}
\def\lablestyle{\scriptstyle}
\xymatrix@-1.0pc{
{}\save[]+<0cm,1.5ex>*{\vdots}\restore&{}\save[]+<0cm,1.5ex>*{\vdots}\restore  
\\
{(1,2)}\ar[rd]\ar[u]
\\
&\ar[ld] (2,1) \ar[uu]
\\
\ar[uu]{\red\mathbf{(1,0)}}\ar[d]
\\
{\mathbf{\green(1,-2)}}\ar[rd]
\\
&\ar[uuu] \mathbf{\red(2,-1)}\ar[d] 
\\
& \ar[ld]\mathbf{\green(2,-3)}
\\
\ar[uuu]\mathbf{\red(1,-4)} \ar[d] 
\\
{\mathbf{\green(1,-6)}}\ar[rd]
\\
& \ar[uuu](2,-5)\ar[ld] 
\\
\ar[uu](1,-8) \ar[rd] 
\\
&\ar[ld] \ar[uu](2,-7) 
\\
\ar[uu](1,-10) 
\\
{}\save[]+<0cm,0ex>*{\vdots}\ar[u]\restore&{}\save[]+<0cm,0ex>*{\vdots}\ar[uu]\restore  
\\
}
\xymatrix@-1.0pc{
{}\save[]+<0cm,1.5ex>*{\vdots}\restore&{}\save[]+<0cm,1.5ex>*{\vdots}\restore  
\\
{\be_{(1,2)}}\ar[rd]\ar[u]
\\
&\ar[ld] \be_{(2,1)} \ar[uu]
\\
\ar[uu]{\red\mathbf{\be_{(1,0)}}}\ar[d]
\\
{\mathbf{\green-\be_{(1,-2)}+\be_{(2,-1)}}}\ar[rd]
\\
&\ar[uuu] \mathbf{\red\be_{(2,-1)}}\ar[d] 
\\
& \ar[ld]\mathbf{\green-\be_{(2,-3)}+\be_{(1,-4)}}
\\
\ar[uuu]\mathbf{\red\be_{(1,-4)}} \ar[d] 
\\
{\mathbf{\green-\be_{(1,-6)}+\be_{(2,-5)}}}\ar[rd]
\\
& \ar[uuu]\be_{(2,-5)}\ar[ld] 
\\
\ar[uu]\be_{(1,-8)} \ar[rd] 
\\
&\ar[ld] \ar[uu]\be_{(2,-7)} 
\\
\ar[uu]\be_{(1,-10)} 
\\
{}\save[]+<0cm,0ex>*{\vdots}\ar[u]\restore&{}\save[]+<0cm,0ex>*{\vdots}\ar[uu]\restore  
\\
}
\xymatrix@-1.0pc{
{}\save[]+<0cm,1.5ex>*{\vdots}\restore&{}\save[]+<0cm,1.5ex>*{\vdots}\restore  
\\
{\be_{(1,2)}}\ar[rd]\ar[u]
\\
&\ar[ld] \be_{(2,1)} \ar[uu]
\\
\ar[uu]{\red\mathbf{\be_{(1,0)}}}\ar[d]
\\
{\mathbf{\green-\be_{(1,-2)}+\be_{(2,-1)}}}\ar[rd]
\\
&\ar[uuu] \mathbf{\red\be_{(2,-1)}}\ar[d] 
\\
& \ar[ld]\mathbf{\green-\be_{(1,-4)}}
\\
\ar[uuu]\mathbf{\red-\be_{(1,-4)}+\be_{(2,-3)}} \ar[d] 
\\
{\mathbf{\green-\be_{(2,-5)}}}\ar[rd]
\\
& \ar[uuu] {-\be_{(2,-5)}+\be_{(1,-6)}}\ar[ld] 
\\
\ar[uu]{-\be_{(1,-8)} + \be_{(2,-7)}} \ar[rd] 
\\
&\ar[ld] \ar[uu]\be_{(2,-7)} 
\\
\ar[uu]\be_{(1,-10)} 
\\
{}\save[]+<0cm,0ex>*{\vdots}\ar[u]\restore&{}\save[]+<0cm,0ex>*{\vdots}\ar[uu]\restore  
\\
}
\xymatrix@-1.0pc{
{}\save[]+<0cm,1.5ex>*{\vdots}\restore&{}\save[]+<0cm,1.5ex>*{\vdots}\restore  
\\
{\be_{(1,2)}}\ar[rd]\ar[u]
\\
&\ar[ld] \be_{(2,1)} \ar[uu]
\\
\ar[uu]{\red\mathbf{\be_{(1,0)}}}\ar[d]
\\
{\mathbf{\green-\be_{(1,-2)}+\be_{(2,-1)}}}\ar[rd]
\\
&\ar[uuu] \mathbf{\red\be_{(2,-1)}}\ar[d] 
\\
& \ar[ld]\mathbf{\green-\be_{(1,-4)}}
\\
\ar[uuu]\mathbf{\red-\be_{(1,-4)}+\be_{(2,-3)}} \ar[d] 
\\
{\mathbf{\green-\be_{(2,-5)}}}\ar[rd]
\\
& \ar[uuu] {-\be_{(1,-6)}}\ar[ld] 
\\
\ar[uu]{-\be_{(2,-7)}} \ar[rd] 
\\
&\ar[ld] \ar[uu]{-\be_{(2,-7)}+\be_{(1,-8)}} 
\\
\ar[uu]{-\be_{(1,-10)}+\be_{(2,-9)} }
\\
{}\save[]+<0cm,0ex>*{\vdots}\ar[u]\restore&{}\save[]+<0cm,0ex>*{\vdots}\ar[uu]\restore  
\\
}
\]
\caption{\label{Fig7d} {\it The first steps of the sequence of $g$-vectors in type $A_2$.}}
\end{figure}

\begin{example}\label{exa-A1}
{\rm
As a warm-up example, let us consider the $A_1$-case. Here there is only one Coxeter element $c=s_1=w_0$.
In this case, all mutations in the sequence occur at vertices which are sinks, and it is easy to 
apply Equation~(\ref{recursion1}). This is illustrated in Figure~\ref{Fig7c}, where we have displayed the first
4 steps of the mutation sequence. For each step we show (left) the initial seed $\Si$ and (right) the mutated 
reference seed $\Si^{(m)}$. The $g$-vectors of the cluster variables of $\Si$ with respect to $\Si^{(m)}$ are
written at the corresponding vertices of $\Si$.
The end result is that the $g$-vectors of the cluster variables of $\Si$ with respect to the limit reference seed are 
\[
 \mathbf{g}^{(\infty)}_{(1,2k)} 
= \be_{(1,2k)},\quad 
 \mathbf{g}^{(\infty)}_{(1,2l)} 
= -\be_{(1,2l)},
\qquad (k\ge 0,\ l< 0).
\]
}
\end{example}

\begin{example}\label{exa-A2}
{\rm
Let us now consider the $A_2$-case. We choose the Coxeter element $c=s_1s_2$. 
In Figure~\ref{Fig7d}, we have displayed the $g$-vectors of the cluster variables of $\Si$ with respect to the reference seeds $\Si^{(m)}$ for $m\le 3$.
In this case it is again easy to check that the $g$-vectors stabilize. The stabilized $g$-vectors
with respect to the limit reference seed are :
\begin{eqnarray*}
 \mathbf{g}^{(\infty)}_{(1,2k)} 
&=&
\left\{
\begin{array}{ll}
\be_{(1,2k)} & \mbox{if $k \ge 0$},
\\[2mm]
-\be_{(1,2k)} +\, \be_{(2,2k+1)} & \mbox{if $-2\le k \le -1 $},
\\[2mm]
-\be_{(2,2k+1)} & \mbox{if $k \le -3 $},
\end{array}
\right.  
\\[3mm]
\mathbf{g}^{(\infty)}_{(2,2k-1)} 
&=&
\left\{
\begin{array}{ll}
\be_{(2,2k-1)} & \mbox{if $k \ge 0$},
\\[2mm]
-\be_{(1,2k-2)} & \mbox{if $k \le -1 $}.
\end{array}
\right.  
\end{eqnarray*}
}
\end{example}

\begin{figure}[t]
\[
\def\objectstyle{\scriptstyle}
\def\lablestyle{\scriptstyle}
\xymatrix@-1.0pc{
&{}\save[]+<0cm,1.5ex>*{\vdots}\restore
\\
&\mathbf{e}_{(1,2)}\ar[u]
\\
&
\\
&\ar[uu]{\mathbf{\red e_{(1,0)}}}\ar[d]
\\
&\mathbf{\green -e_{(1,-2)}} 
\\
&
\\
&\ar[uu] \mathbf{-e}_{(1,-4)} 
\\
&
\\
&\ar[uu] \mathbf{-e}_{(1,-6)} 
\\
&
\\
&\ar[uu]\mathbf{-e}_{(1,-8)} 
\\
&
\\
&\ar[uu] \mathbf{-e}_{(1,-10)} 
\\
&
\\
&\ar[uu]\mathbf{-e}_{(1,-12)} 
\\
&
\\
&\ar[uu]\mathbf{-e}_{(1,-14)} 
\\
&{}\save[]+<0cm,0ex>*{\vdots}\ar[u]\restore
\\
}
\qquad
\xymatrix@-1.0pc{
{}\save[]+<0cm,1.5ex>*{\vdots}\restore&{}\save[]+<0cm,1.5ex>*{\vdots}\restore  
\\
{\be_{(1,2)}}\ar[rd]\ar[u]
\\
&\ar[ld] \be_{(2,1)} \ar[uu]
\\
\ar[uu]{\red\mathbf{\be_{(1,0)}}}\ar[d]
\\
{\mathbf{\green-\be_{(1,-2)}+\,\be_{(2,-1)}}}\ar[rd]
\\
&\ar[uuu] \mathbf{\red\be_{(2,-1)}}\ar[d] 
\\
& \ar[ld]\mathbf{\green-\be_{(1,-4)}}
\\
\ar[uuu]\mathbf{\red-\be_{(1,-4)}+\be_{(2,-3)}} \ar[d] 
\\
{\mathbf{\green-\be_{(2,-5)}}}\ar[rd]
\\
& \ar[uuu] {-\be_{(1,-6)}}\ar[ld] 
\\
\ar[uu]{-\be_{(2,-7)}} \ar[rd] 
\\
&\ar[ld] \ar[uu]{-\be_{(1,-8)}} 
\\
\ar[uu]{-\be_{(2,-9)} }\ar[rd]
\\
&\ar[uu]{-\be_{(1,-10)}} \ar[ld]
\\
\ar[uu]{-\be_{(2,-11)} }
\\
{}\save[]+<0cm,0ex>*{\vdots}\ar[u]\restore&{}\save[]+<0cm,0ex>*{\vdots}\ar[uu]\restore  
\\
}
\xymatrix@-1.0pc{
&{}\save[]+<0cm,1.5ex>*{\vdots}\restore&{}\save[]+<0cm,1.5ex>*{\vdots}\restore  
&{}\save[]+<0cm,1.5ex>*{\vdots}\restore
\\
&{\be_{(1,2)}}\ar[rd]\ar[u]&
&\ar[ld] \be_{(3,2)} \ar[u]
\\
&&\ar[ld] \be_{(2,1)} \ar[rd]\ar[uu]&&
\\
&{\ar[uu]}{\red\be_{(1,0)}}\ar[d]&&\ar[ldd] \be_{(3,0)} \ar[uu]
\\
&{\green\mathbf{-e_{(1,-2)}+\,e_{(2,-1)}}}\ar[rd]
\\
&&\ar[uuu] \mathbf{\red e_{(2,-1)}}\ar[d]&&
\\
&&\ar[ld] \mathbf{\green -e_{(1,-4)}+\, e_{(3,-2)}} \ar[rd]&&
\\
&\ar[uuu]\mathbf{\red -e_{(1,-4)}+\,e_{(2,-3)}} \ar[d] && \ar[d] \mathbf{\red e_{(3,-2)}}\ar[uuuu]
\\
&\mathbf{\green -e_{(2,-5)} +\,e_{(3,-4)}} \ar[rd] &&\ar[ld] \mathbf{\green -e_{(1,-6)}}
\\
&& \ar[uuu]\mathbf{\red -e_{(1,-6)} + \, e_{(3,-4)}}\ar[d] &&
\\
&&\ar[ld]\mathbf{\green -e_{(2,-7)}} \ar[rdd]&&
\\
&\ar[uuu]\mathbf{\red -e_{(2,-7)} + \, e_{(3,-6)}}\ar[d]  
\\
&\mathbf{\green -e_{(3,-8)}} \ar[rd]&&\ar[ld] {-\be_{(1,-8)}}\ar[uuuu]
\\
&& \ar[uuu]{-\be_{(2,-9)}} \ar[ld]\ar[rd]&&
\\
&\ar[uu] {-\be_{(3,-10)}} &&\ar[uu] {-\be_{(1,-10)}} 
\\
&{}\save[]+<0cm,0ex>*{\vdots}\ar[u]\restore&{}\save[]+<0cm,0ex>*{\vdots}\ar[uu]\restore  
&{}\save[]+<0cm,0ex>*{\vdots}\ar[u]\restore
\\
}
\]
\caption{\label{Fig7e} {\it The stabilized $g$-vectors of an initial seed in type $A_n\ (n\le 3)$.}}
\end{figure}
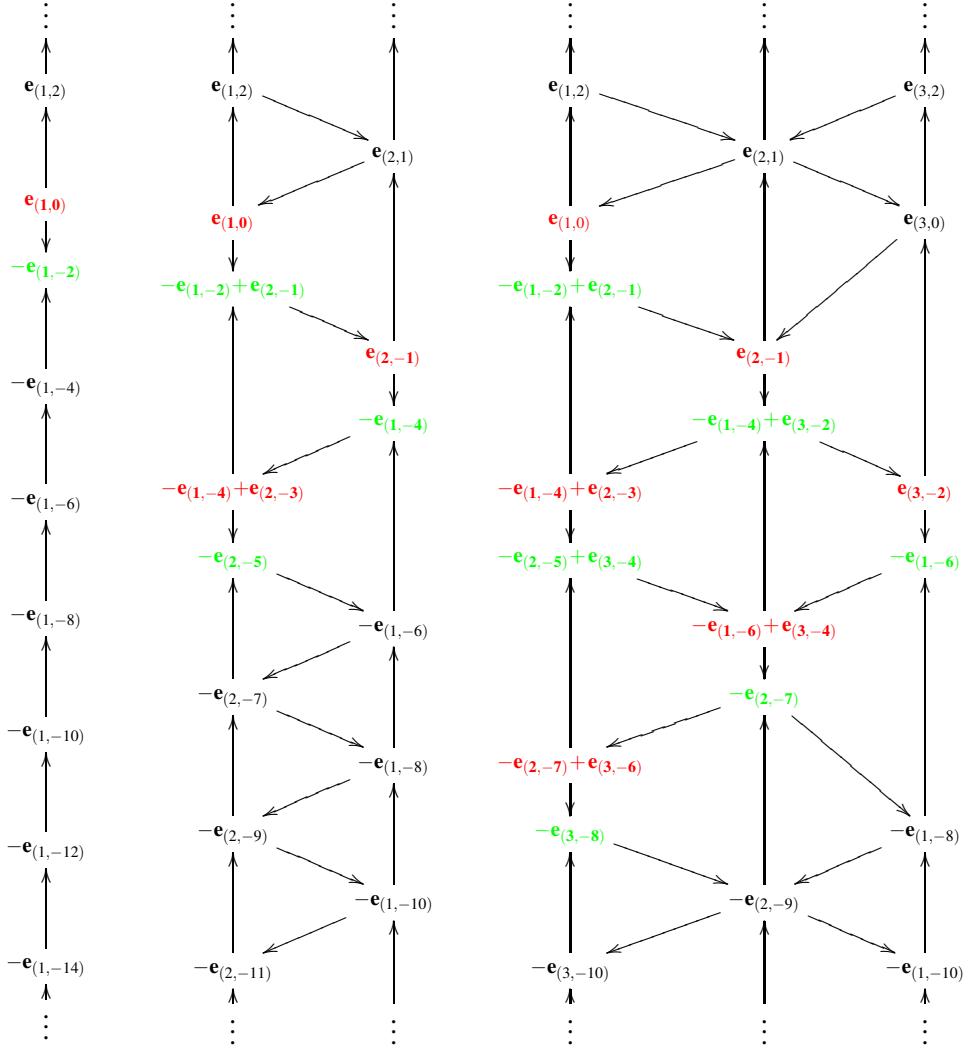

\begin{example}\label{examples-stables}
{\rm
The stabilized $g$-vectors of the cluster variables of a typical initial seed in types $A_1$, $A_2$ and $A_3$
are displayed in Figure~\ref{Fig7e}.
}
\end{example}

\subsection{Calculation of stabilized $g$-vectors}\label{subsec4-3}

In this section we obtain an explicit expression of the above sequence of $g$-vectors, and show that this sequence stabilizes. To do this we first introduce a partition of the quiver $\G$ into horizontal slices. We then show that the $G$-matrices whose column vectors are the $g$-vectors have a block diagonal form, where the diagonal blocks correspond to the slices of the partition. These blocks can be explicitly calculated using the action of the Weyl group $W$ on the weight lattice.

\subsubsection{Slices for $\G$}\label{subsec4-3-1}

Let $\cQ$ be a Dynkin quiver with vertex set $I$, as in \S\ref{subsec-Coxeter}. Let $c$ be the corresponding Coxeter element, and $\G = \G_c$ the corresponding quiver of an initial seed of $\AA_{w_0}$.  We will use the height function 
$l_c : I\ra\Z_{\ge 0}$, which is uniquely determined by the following two
properties:
\begin{itemize}
    \item $0\in\mathrm{Im}(l_c)$,
    \item If $i \ra j$ is an arrow in $\cQ$ then $l_c(i)=l_c(j)+1$.
\end{itemize}
This is well-defined since the underlying Dynkin diagram is a tree. Note that if $l_c(i)=0$, then $i$ is a sink of $\cQ$. 

With this notation, the vertex set $V$ of $\G$ can be written as 
\[
V = \bigcup_{i\in I}  \{(i,k)\in I\times\Z \mid k\in l_c(i)+2\Z\}.
\]
An arrow $(i,m)\ra (j,n)$ in $\G$ is called \emph{vertical} if $i=j$, otherwise it is called \emph{oblique}. For every vertical down arrow $(i,k)\ra (i,k-2)$  we
draw the vertex $(i,k)$ \emph{red} and $(i,k-2)$ \emph{green}.  All other vertices remain black.

\begin{remark} \label{rem:slices}
{\rm
(1) If we delete all oblique arrows from $\G$,  the red vertices  are
precisely the sources, and the green vertices are precisely the sinks  of the remaining quiver.

(2) The vertical down arrows are in natural bijection with the vertices of
the Auslander-Reiten quiver of the path algebra $K\cQ$.

(3) We have for each $m\in \Z_{\ge 0}$ an embedding
$\iota_m : \cQ\ra \G$ defined by:
\[
 i \mapsto (i, l_c(i)+2m), \mbox{ and }  
                      (i\ra j) \mapsto ((i, l_c(i)+2m) \ra (j, l_c(j)+2m)).
\]
}
\end{remark}

The next Lemma is easy to check.
\begin{Lem} \label{lem:slices}
Let $\alpha\df (i,m)\ra (j,n)$ be an oblique arrow in $\G$.
\begin{itemize}
    \item[(a)]  
    The start point $s(\alpha)=(i,m)$ of $\alpha$ is either green or black. 
    The endpoint $t(\alpha)=(j,n)$ of $\alpha$ is either red or black.
    \item[(b1)]
    If $s(\alpha)$ is green,  there exists an \emph{opposite} oblique arrow
    \[
    \alpha^+\df (j, n+2)\ra (i, m+2) \mbox{ and } t(\alpha^+)=(i, m+2) \mbox{ is red}.
    \]
    \item[(b2)] 
    If $s(\alpha)$ is black,  there exists a \emph{parallel} oblique arrow 
    \[\alpha^+\df (i,m+2)\ra (j, n+2),\]
    except when $(j, n+2)$  is green.  
    \item[(c1)] 
    If $t(\alpha)$ is red,  there exists an \emph{opposite} oblique arrow
    \[
    \alpha^-\df (j, n-2)\ra (i, m-2) \mbox{ and  } s(\alpha^-)=(j, n-2)  \mbox{ is green.}
    \]
    \item[(c2)]
    If $t(\alpha)$ is black,  there exists a \emph{parallel} oblique arrow 
    \[\alpha^-\df (i,m-2)\ra (j, n-2),\]
    except when $(i, m-2)$ is red. 
\end{itemize} 
\end{Lem}

\begin{Def}
For $m\in\Z$, let  
\[
   I(m):=\{(i, l_c(i)+2m)\mid i\in I\}\subset V, 
\]
and denote by $\cQ(m)$ the full subquiver of  $\G$ with vertex set $I(m)$.
\end{Def}

Note, that $\cQ(m)$ contains no vertical arrows of $\G$, and for $m\geq 0$,
the quiver $\cQ(m)$ is canonically isomorphic to $\cQ$. We regard the subquivers $\cQ(m)$ 
as \emph{horizontal slices} of $\G$. This is illustrated in Figure~\ref{Fig7f}, in which
for the quiver 
\[
 \cQ : 1 \leftarrow 2 \leftarrow 3 \ra 4
\]
of type $A_4$, the slices $\cQ(0), \cQ(-1), \ldots,  \cQ(-6)$ of $\G$ are surrounded by grey lines.
The next proposition is easy to check.
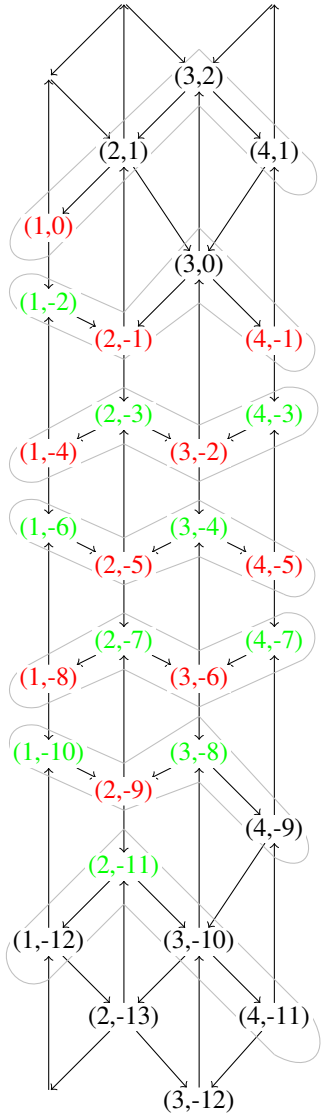
\begin{figure}
\begin{center} 
\begin{tikzpicture}
	\begin{pgfonlayer}{nodelayer}
		\node [style=bcoord] (0) at (2, 2) {};
		\node [style=bcoord] (1) at (-2, 2) {};
		\node [style=bco2] (2) at (0, 0) {(3,2)};
		\node [style=bco2] (3) at (2, -2) {(4,1)};
		\node [style=bcoord] (4) at (-4, 0) {};
		\node [style=bco2] (5) at (-2, -2) {(2,1)};
		\node [style=rco2] (6) at (-4, -4) {(1,0)};
		\node [style=gco2] (7) at (-4, -6) {(1,-2)};
		\node [style=rco2] (9) at (-2, -7) {(2,-1)};
		\node [style=bco2] (10) at (0, -5) {(3,0)};
		\node [style=gco2] (11) at (-2, -9) {(2,-3)};
		\node [style=rco2] (12) at (0, -10) {(3,-2)};
		\node [style=gco2] (13) at (0, -12) {(3,-4)};
		\node [style=rco2] (14) at (2, -7) {(4,-1)};
		\node [style=gco2] (15) at (2, -9) {(4,-3)};
		\node [style=rco2] (16) at (-4, -10) {(1,-4)};
		\node [style=gco2] (17) at (-4, -12) {(1,-6)};
		\node [style=rco2] (18) at (2, -13) {(4,-5)};
		\node [style=gco2] (19) at (2, -15) {(4,-7)};
		\node [style=rco2] (20) at (0, -16) {(3,-6)};
		\node [style=gco2] (21) at (0, -18) {(3,-8)};
		\node [style=rco2] (22) at (-2, -13) {(2,-5)};
		\node [style=gco2] (23) at (-2, -15) {(2,-7)};
		\node [style=rco2] (24) at (-4, -16) {(1,-8)};
		\node [style=gco2] (25) at (-4, -18) {(1,-10)};
		\node [style=rco2] (26) at (-2, -19) {(2,-9)};
		\node [style=gco2] (27) at (-2, -21) {(2,-11)};
		\node [style=bco2] (28) at (2, -20) {(4,-9)};
		\node [style=bco2] (29) at (-4, -23) {(1,-12)};
		\node [style=bco2] (30) at (0, -23) {(3,-10)};
		\node [style=bco2] (31) at (-2, -25) {(2,-13)};
		\node [style=bco2] (32) at (2, -25) {(4,-11)};
		\node [style=bco2] (33) at (0, -27.25) {(3,-12)};
		\node [style=bcoord] (34) at (-4, -27) {};
		\node [style=none] (35) at (-4.75, -22.75) {};
		\node [style=none] (36) at (-2, -20) {};
		\node [style=none] (37) at (-4, -24) {};
		\node [style=none] (38) at (-2, -22) {};
		\node [style=none] (39) at (2, -26) {};
		\node [style=none] (40) at (3, -25) {};
		\node [style=none] (41) at (2, -20.75) {};
		\node [style=none] (42) at (2.75, -20) {};
		\node [style=none] (43) at (0, -17) {};
		\node [style=none] (44) at (-2, -18.25) {};
		\node [style=none] (45) at (-4.25, -17.25) {};
		\node [style=none] (46) at (-4.75, -18.25) {};
		\node [style=none] (47) at (-2, -19.5) {};
		\node [style=none] (48) at (0, -18.75) {};
		\node [style=none] (49) at (-4.75, -15.75) {};
		\node [style=none] (50) at (-2, -14.25) {};
		\node [style=none] (51) at (0, -15.25) {};
		\node [style=none] (52) at (2.25, -14.25) {};
		\node [style=none] (53) at (2.75, -15.5) {};
		\node [style=none] (54) at (0, -16.75) {};
		\node [style=none] (55) at (-2, -15.75) {};
		\node [style=none] (56) at (-4.25, -16.75) {};
		\node [style=none] (57) at (-4.75, -9.75) {};
		\node [style=none] (58) at (-2, -8.25) {};
		\node [style=none] (59) at (0, -9.25) {};
		\node [style=none] (60) at (2.25, -8.25) {};
		\node [style=none] (61) at (2.75, -9.5) {};
		\node [style=none] (62) at (0, -10.75) {};
		\node [style=none] (63) at (-2, -9.75) {};
		\node [style=none] (64) at (-4.25, -10.75) {};
		\node [style=none] (65) at (0, 0.75) {};
		\node [style=none] (66) at (3, -2.25) {};
		\node [style=none] (67) at (2.25, -3) {};
		\node [style=none] (68) at (0, -0.75) {};
		\node [style=none] (69) at (-4, -4.75) {};
		\node [style=none] (70) at (-4.75, -3.75) {};
		\node [style=none] (71) at (-4.25, -11.25) {};
		\node [style=none] (72) at (-2, -12.25) {};
		\node [style=none] (73) at (0, -11.25) {};
		\node [style=none] (74) at (2.75, -12.75) {};
		\node [style=none] (75) at (2.25, -13.75) {};
		\node [style=none] (76) at (0, -12.75) {};
		\node [style=none] (77) at (-2, -13.5) {};
		\node [style=none] (78) at (-4.75, -12.25) {};
		\node [style=none] (79) at (-4.25, -5.25) {};
		\node [style=none] (80) at (-2, -6.25) {};
		\node [style=none] (81) at (0, -4) {};
		\node [style=none] (82) at (2.75, -6.75) {};
		\node [style=none] (83) at (2.25, -7.75) {};
		\node [style=none] (84) at (0, -6) {};
		\node [style=none] (85) at (-2, -7.5) {};
		\node [style=none] (86) at (-4.75, -6.25) {};
	\end{pgfonlayer}
	\begin{pgfonlayer}{edgelayer}
		\draw [style=arrow 1] (1) to (4);
		\draw [style=arrow 1] (4) to (5);
		\draw [style=arrow 1] (5) to (1);
		\draw [style=arrow 1] (1) to (2);
		\draw [style=arrow 1] (2) to (3);
		\draw [style=arrow 1] (2) to (5);
		\draw [style=arrow 1] (5) to (6);
		\draw [style=arrow 1] (6) to (4);
		\draw [style=arrow 1] (6) to (7);
		\draw [style=arrow 1] (9) to (5);
		\draw [style=arrow 1] (7) to (9);
		\draw [style=arrow 1] (5) to (10);
		\draw [style=arrow 1] (10) to (2);
		\draw [style=arrow 1] (3) to (10);
		\draw [style=arrow 1] (10) to (9);
		\draw [style=arrow 1] (10) to (14);
		\draw [style=arrow 1] (14) to (15);
		\draw [style=arrow 1] (15) to (12);
		\draw [style=arrow 1] (12) to (10);
		\draw [style=arrow 1] (9) to (11);
		\draw [style=arrow 1] (11) to (12);
		\draw [style=arrow 1] (11) to (16);
		\draw [style=arrow 1] (16) to (7);
		\draw [style=arrow 1] (12) to (13);
		\draw [style=arrow 1] (16) to (17);
		\draw [style=arrow 1] (17) to (22);
		\draw [style=arrow 1] (22) to (11);
		\draw [style=arrow 1] (13) to (22);
		\draw [style=arrow 1] (13) to (18);
		\draw [style=arrow 1] (18) to (15);
		\draw [style=arrow 1] (14) to (3);
		\draw [style=arrow 1] (18) to (19);
		\draw [style=arrow 1] (19) to (20);
		\draw [style=arrow 1] (20) to (13);
		\draw [style=arrow 1] (22) to (23);
		\draw [style=arrow 1] (23) to (20);
		\draw [style=arrow 1] (23) to (24);
		\draw [style=arrow 1] (24) to (17);
		\draw [style=arrow 1] (24) to (25);
		\draw [style=arrow 1] (25) to (26);
		\draw [style=arrow 1] (26) to (23);
		\draw [style=arrow 1] (20) to (21);
		\draw [style=arrow 1] (21) to (26);
		\draw [style=arrow 1] (26) to (27);
		\draw [style=arrow 1] (21) to (28);
		\draw [style=arrow 1] (28) to (19);
		\draw [style=arrow 1] (27) to (29);
		\draw [style=arrow 1] (29) to (25);
		\draw [style=arrow 1] (27) to (30);
		\draw [style=arrow 1] (30) to (21);
		\draw [style=arrow 1] (28) to (30);
		\draw [style=arrow 1] (29) to (31);
		\draw [style=arrow 1] (31) to (27);
		\draw [style=arrow 1] (30) to (31);
		\draw [style=arrow 1] (30) to (32);
		\draw [style=arrow 1] (32) to (28);
		\draw [style=arrow 1] (32) to (33);
		\draw [style=arrow 1] (33) to (30);
		\draw [style=arrow 1] (31) to (33);
		\draw [style=arrow 1] (31) to (34);
		\draw [style=arrow 1] (34) to (29);
		\draw [style=arrow 1] (0) to (2);
		\draw [style=arrow 1] (3) to (0);
		\draw [style=dotted] (35.center) to (36.center);
		\draw [style=dotted] (36.center) to (40.center);
		\draw [style=dotted, bend left=75, looseness=2.00] (40.center) to (39.center);
		\draw [style=dotted] (38.center) to (39.center);
		\draw [style=dotted] (38.center) to (37.center);
		\draw [style=dotted, bend right=90, looseness=1.75] (35.center) to (37.center);
		\draw [style=dotted] (45.center) to (44.center);
		\draw [style=dotted] (44.center) to (43.center);
		\draw [style=dotted] (43.center) to (42.center);
		\draw [style=dotted, bend left=75, looseness=2.00] (42.center) to (41.center);
		\draw [style=dotted] (48.center) to (41.center);
		\draw [style=dotted] (48.center) to (47.center);
		\draw [style=dotted] (47.center) to (46.center);
		\draw [style=dotted, bend left=75, looseness=1.75] (46.center) to (45.center);
		\draw [style=dotted] (49.center) to (50.center);
		\draw [style=dotted] (50.center) to (51.center);
		\draw [style=dotted] (51.center) to (52.center);
		\draw [style=dotted, bend left=75, looseness=1.50] (52.center) to (53.center);
		\draw [style=dotted] (53.center) to (54.center);
		\draw [style=dotted] (54.center) to (55.center);
		\draw [style=dotted] (55.center) to (56.center);
		\draw [style=dotted, bend right=75, looseness=1.75] (49.center) to (56.center);
		\draw [style=dotted] (57.center) to (58.center);
		\draw [style=dotted] (58.center) to (59.center);
		\draw [style=dotted] (59.center) to (60.center);
		\draw [style=dotted, bend left=75, looseness=1.75] (60.center) to (61.center);
		\draw [style=dotted] (61.center) to (62.center);
		\draw [style=dotted] (62.center) to (63.center);
		\draw [style=dotted] (63.center) to (64.center);
		\draw [style=dotted, bend right=75, looseness=1.75] (57.center) to (64.center);
		\draw [style=dotted] (70.center) to (65.center);
		\draw [style=dotted] (65.center) to (66.center);
		\draw [style=dotted, bend left=75, looseness=1.75] (66.center) to (67.center);
		\draw [style=dotted] (67.center) to (68.center);
		\draw [style=dotted] (68.center) to (69.center);
		\draw [style=dotted, bend left=75, looseness=2.00] (69.center) to (70.center);
		\draw [style=dotted] (71.center) to (72.center);
		\draw [style=dotted] (72.center) to (73.center);
		\draw [style=dotted] (73.center) to (74.center);
		\draw [style=dotted, bend left=90, looseness=1.75] (74.center) to (75.center);
		\draw [style=dotted] (75.center) to (76.center);
		\draw [style=dotted] (76.center) to (77.center);
		\draw [style=dotted] (77.center) to (78.center);
		\draw [style=dotted, bend left=75, looseness=1.75] (78.center) to (71.center);
		\draw [style=dotted] (79.center) to (80.center);
		\draw [style=dotted] (80.center) to (81.center);
		\draw [style=dotted] (81.center) to (82.center);
		\draw [style=dotted, bend left=90, looseness=1.75] (82.center) to (83.center);
		\draw [style=dotted] (83.center) to (84.center);
		\draw [style=dotted] (84.center) to (85.center);
		\draw [style=dotted] (85.center) to (86.center);
		\draw [style=dotted, bend left=75, looseness=1.75] (86.center) to (79.center);
	\end{pgfonlayer}
\end{tikzpicture}
\end{center} 
\caption{\label{Fig7f} {\it The horizontal slices of $\G$ in type $A_4$ for $c=s_1s_2s_4s_3$.}}
\end{figure}

\begin{Prop} \label{prop:slices}
 Let $m\in\Z$. 
\begin{itemize}
\item[(a)]
 Each green vertex of $\cQ(m)$ is a source and each red vertex of $\cQ(m)$ is a sink.   
\item[(b)]
Let $I_{\,\mathrm{red}}(m)$ denote the subset of red vertices of $\cQ(m)$. Then the shift map
\[
I(m)\ra I(m-1),\quad  (i, l(i)+2m)\mapsto (i, l(i)+2m-2)\quad (i\in I),
\]
induces an isomorphism of quivers
\[
\left(\prod_{v\in I_{\,\mathrm{red}}(m)} s_v\right)\  \cQ(m) \stackrel{\sim}{\longrightarrow} \cQ(m-1),
\]
where $s_v$  denotes the quiver reflection which changes the orientation of all arrows 
incident with $v$.  
In particular, the underlying unoriented graph of $\cQ(m)$ is isomorphic to the Dynkin diagram of $\g$ for all $m\in\Z$.
\item[(c)]
Dually, let $I_{\,\mathrm{grn}}(m)$ denote the subset of green vertices of $\cQ(m)$. Then the shift map
\[
I(m)\ra I(m+1),\quad  (i, l(i)+2m)\mapsto (i, l(i)+2m+2)\quad (i\in I),
\]
induces an isomorphism of quivers
\[
\left(\prod_{v\in I_{\,\mathrm{grn}}(m)} s_v\right)\  \cQ(m) \stackrel{\sim}{\longrightarrow} \cQ(m+1).
\]
\item[(d)] 
Let $\nu\df I\ra I$ be the Nakayama permutation defined by $w_0(\a_i) = -\a_{\nu(i)}$. 
Suppose that for
some $m<0$ we have $I_{\mathrm{red}}(m)=\emptyset$.  Then the twisted shift
\[
I(0) \ra I(m), \quad  (i, l_c(i)) \mapsto (\nu(i), l_c(\nu(i))+2m)\quad (\forall i\in I)
\]
induces an isomorphism $\cQ \stackrel{\sim}{\longrightarrow}  \cQ(m)$ of quivers.
\end{itemize}
\end{Prop}

\subsubsection{Transformation of the G-matrix under green mutations}

From now we fix $I=\{1, 2, \ldots, n\}$.  We define a 
total ordering on the 
set $V$ of vertices of $\G$ by the following rule:
\begin{equation}  \label{eq:Gam-order}
(i, l_c(i)+2a) < (j, l_c(j)+ 2b) \Longleftrightarrow (a<b)  \mbox{ or } (a=b \mbox{ and } i<j)
\end{equation}
In the sequel, when we consider a matrix with rows and columns indexed by $V$, we always assume that its rows and columns are ordered using this ordering.

Recall the sequence of quivers $\G^{(k)}\ (k\ge 0)$ defined in \S\ref{subsec4-2}. 
All these quivers are isomorphic to 
$\G^{(0)} = \G$ by substituting every vertex $(i,l)$ by $(i,l-2k)$.
In particular, every quiver $\G^{(k)}$ has again $N$ red and $N$ green vertices given by the end points
of the $N$ down arrows, where 
$N$ is the number of positive roots of the Dynkin diagram.
Let $V^{(k)}_{\,\mathrm{grn}}\subset V$ denote the subset of green vertices of $\G^{(k)}$.
We define the following sequence of mutations of $\G^{(k)}$:
\[
\mugreen^{(k)}:=  \prod_{(i,l)\in V^{(k)}_{\,\mathrm{grn}}} \mu_{(i,l)}.
\]
Thus $\mugreen^{(k)}$ consists of $N$ pairwise commuting mutations.
As explained in \S\ref{subsec4-2}, we have 
\[
\mugreen^{(k)}(\G^{(k)})=\G^{(k+1)}.
\]

Let 
\[
G^{(0)}= \left(g^{(0)}_{x,y}\right)_{x, y\in V}\in\Z^{V\times V}
\]
be the \emph{initial} g-matrix with $g^{(0)}_{x,y}=\delta_{x,y}$,
where $\delta$ is the Kronecker symbol.  
The convention is that the $g$-vectors are the column vectors of the $G$-matrix.

Now, we define recursively 
\[
G^{(k+1)}:=\tmugreen^{(k)}(G^{(k)}),
\]
where $\tmugreen^{(k)}$ consists in applying to the column vectors of $G^{(k)}$ the transformations of \S\ref{subsec4-1} (\ref{recursion1}) for all green vertices $l$ of $\G^{(k)}$. Thus, the columns of $G^{(k)}$ are the $g$-vectors
of the cluster variables of the initial seed $\mathbf{x}$ attached to $\G^{(0)}$ with respect 
to the reference seed $\mugreen^{(k-1)}\circ\cdots\circ\mugreen^{(0)}(\mathbf{x})$ attached to $\G^{(k)}$.

For each $i\in I$ let $\bt_i\in\Z^{I\times I}$ be the matrix
with entries 
\begin{equation}  \label{eq:t_i}
    t^{(i)}_{jk}= 
\left\{
\begin{array}{cl}
\delta_{jk} &\mbox{if } k\neq i,\\
\delta_{ji}-c_{ji} &\mbox{if } k=i.   
\end{array}
\right.
\end{equation}

\begin{remark} \label{rem:refl}
{\rm
Comparing with Eq.\,(\ref{fundw_roots}), (\ref{action-s-fw-sr}), we see that $\bt_i$ is the matrix of  the simple reflection $s_i\in W$ with respect to the basis of fundamental weights. 
Note that left multiplication of a matrix $M$
by $\bt_i$ can be seen as a row transformation:  row $i$ of $M$ is added to each row $j$ of $M$ such that $j$ is connected to $i$ by an edge in the Dynkin diagram, and then row $i$ of $M$ is multiplied by $-1$.

Observe also that the transposed matrix $\bs_i:=\bt_i^t$ is the matrix of the same reflection
$s_i\in W$ with respect to the basis of simple roots.
}
\end{remark}

\begin{Lem} \label{lem:green1}
    Suppose that in each row of $G^{(k)}$ which corresponds to a 
    green vertex of $\G^{(k)}$, all entries are non-negative.  
    Then, the transformation $\tmugreen^{(k)}$ of $G^{(k)}$  can be
    realized by left multiplication by a block diagonal matrix
    $\Mgreen^{(k)}$, where each diagonal block $\Mgreen^{(k)}(m)$
    corresponds to the subset $I(m)\subset V$.
More precisely,  in this case we have
\[
\Mgreen^{(k)}(m):=  \prod_{i\in I^{(k)}_{\mathrm{grn}}(m)} \bt_i,
\]    
where
\[
\Igrn^{(k)}(m):=\{i\in I \mid (i, l_c(i)+2m)\in\G^{(k)}\ 
\mathrm{ is\ green}\},
\]
and we identify each $(i, l_c(i)+2m)\in I(m)$ with $i\in I$.
\end{Lem}

Note that by construction $\Igrn^{(k+1)}(m)=\Igrn^{(k)}(m+1)$.

\begin{proof}
    Recall, that $\G^{(k)}$ is obtained from $\G=\G^{(0)}$ by 
    relabelling the vertices. More precisely we have to substitute 
    $(j, l)$ by $(j,l-2k)$  for all $(j,l)\in V$,  and keep the colouring.
    In particular, the subquivers $\cQ(m)$  can also be considered for $\G^{(k)}$.
    Up to the corresponding shift they have the same properties as in $\G$.
    
    Thus, by Remark~\ref{rem:slices} and Proposition~\ref{prop:slices}, 
    two different green vertices of $\G^{(k)}$ are never joined by an arrow.  
    The sign hypothesis on $G^{(k)}$ implies that we are in the second case of Equation~(\ref{recursion1}).
    It follows that the transformation $\tilde{\mu}_{(i, l_c(i)+2m)}$ of $G^{(k)}$ corresponding to a single green vertex $(i, l_c(i)+2m)\in \cQ(m)\subset\G^{(k)}$  is a row transformation. 
    Moreover, transformations for different green vertices commute.
    For this reason it is sufficient to show that the matrix for the row 
    transformation corresponding to a single green vertex has the alluded block diagonal 
    form.  In fact,  by Proposition~\ref{prop:slices},
    for a green vertex  $(i, l_c(i)+2m)\in \cQ(m)$ as above, the arrows in $\G^{(k)}$
    which start at $(i, l_c(i)+2m)$ are precisely of the form 
    \[
    (i, l_c(i)+2m) \ra (j, l_c(j)+2m) \mbox{ with } c_{i,j}<0.
    \]
    This means that, by Equation~(\ref{recursion1}) and Remark~\ref{rem:refl}, the corresponding transformation is given by putting the 
    matrix $\bt_i$ in the diagonal block corresponding to $\cQ(m)$, whilst 
    all other diagonal blocks are identity matrices. \cqfd      
\end{proof}

Let $\cQ'$ be a Dynkin quiver with vertex set $I=\{1,2,\ldots,n\}$, and let $N$ denote the number of its positive roots. Recall that we say that a sequence $\bi=(i_r, i_{r-1}, \ldots, i_1)$ of vertices 
 of $\cQ'$ is \emph{adapted} to $\cQ'$ if $r\leq N$, and for every $k=0, 1, \ldots, r-1$ we have 
 that $i_{k+1}$ is a source of $s_{i_k}\cdots s_{i_2}s_{i_1}(\cQ')$.
    Then $(i_r, \ldots, i_1)$ is a reduced expression for $w= s_{i_r}\cdots s_{i_1}\in W$, and we will
    also say that $s_{i_r}\cdots s_{i_1}$ is adapted to $\cQ'$. 

\begin{Lem} \label{lem:reflt}
Let $\bi=(i_r, i_{r-1}, \ldots, i_1)$ be adapted to the Dynkin quiver $\cQ'$.
    Then, with the notation from~(\ref{eq:t_i}), the matrix
    \[
    T:=\bt_{i_{r-1}}\cdots \bt_{i_2}\bt_{i_1}\in \Z^{I\times I}
    \]
    has the following properties:
    \begin{itemize}
        \item Each row of $T$ is the expansion of a certain root in terms of the simple roots.
        \item Row $i_r$ of $T$ corresponds to a \emph{positive} root.
    \end{itemize}
\end{Lem}

\begin{proof}
By construction, $(i_r, \ldots, i_1)$ is a reduced expression for $w= s_{i_r}\cdots s_{i_1}\in W$. 
    Let $(\be_i)_{i\in I}$ be the coordinate basis of $\Z^I$. Then for $j\in I$ we have, in view of 
    Remark~\ref{rem:refl}, 
 \[
 (\be_j^t (\bt_{i_{r-1}}\cdots \bt_{i_1}))^t= (\bs_{i_1}\cdots \bs_{i_{r-1}})\be_j,  
 \]   
which is the expansion of a root in terms of the simple roots.

It is classical  that in this situation  
\[
\alpha_{i_1},\ s_{i_1}(\alpha_{i_2}),\ s_{i_1}s_{i_2}(\alpha_{i_3}), \ldots,\  
s_{i_1}s_{i_2}\cdots s_{i_{r-1}}(\alpha_{i_r})
\]
is precisely a list of the \emph{positive} roots which are sent by $w$ to negative roots.  
Thus, in particular
\[
\be_{i_r}^t(\bt_{i_{r-1}}\cdots \bt_{i_2}\bt_{i_1})
\]
is the expansion of a positive root in terms of the simple roots. \cqfd
\end{proof}

\smallskip
Using the notation of Lemma~\ref{lem:green1}, we define for each $m\in\Z$ :    
\[
\begin{array}{lcl}
\bT_m         &:=&  \prod_{i\in \Igrn^{(0)}(m)}\! \bt_i \in\Z^{I\times I},\\[2mm]
\bS_m         &:=& \prod_{i\in \Igrn^{(0)}(m)}\! \bs_i \in\Z^{I\times I},\\[2mm]
S_m           &:=& \prod_{i\in \Igrn^{(0)}(m)}\! s_i \in W.
\end{array}
\]
All factors of these products pairwise commute since different green vertices are never
joined by an arrow. Let
\[
h_c:=  \min\left\{m\in\Z\mid \Igrn^{(0)}(m)\neq\emptyset\right\}.
\]
Note that $h_c<0$.
If $m\not \in \{-1,-2,\ldots, h_c\}$ then $\Igrn^{(0)}(m) = \emptyset$, and it is understood that $\bT_m = \mathrm{Id}_n$,
the identity matrix of size $n$.
It follows from the construction of $\G$ that $S_{-1}\cdots S_{-2} \cdots S_{h_c}$ is a reduced expression for the longest element $w_0\in W$  which is adapted to $ \cQ(h_c)$. 

\begin{Thm} \label{thm:g-vect1}
    For each $k \ge 0$, the $G$-matrix $G^{(k)}$ is of block diagonal 
    form $\mathrm{diag}(G^{(k)}(m))\mid m\in\Z)$, where each diagonal block $G^{(k)}(m)$ corresponds to
    the vertex set $I(m)\subset V$.  More precisely,  we have 
\[
G^{(k)}(m)=     \bT_{m+k-1}\cdots \bT_{m+1}\cdot \bT_m
\]
under the usual identifications.  In particular,  for $k\gg |m|$ we have
\begin{equation}\label{eq:Gkgg0}
G^{(k)}(m)=
\left\{
\begin{array}{cl}
    \mathrm{Id}_n &\mbox{if } m>0,\\
    \bT_{-1} \bT_{-2}\cdots \bT_m  &\mbox{if } h_c\leq m \leq -1,\\
    \bT_{-1} \bT_{-2}\cdots \bT_{h_c} &\mbox{if } m\leq h_c,
\end{array}
\right.
\end{equation}
which shows that the sequence of $G$-matrices $(G^{(k)})_{k\ge 0}$ has a well-defined 
limit 
\[
G^{(\infty)} := \lim_{k\to +\infty} G^{(k)}.
\]
\end{Thm}

\begin{proof}
In order to prove the first claim  we proceed by induction on $k$.  Trivially,  all rows of $G^{(0)}$ are non-negative.   
Thus, by Lemma~\ref{lem:green1},  the matrix $G^{(1)}$ is of the desired shape. 
For the induction step we may assume that we have in each diagonal block
\[
G^{(k)}(m)= \bT_{m+k-1}\cdots \bT_m = \bt_{i_r}\cdots\bt_{i_2}\bt_{i_1},
\]
where $(i_r, \ldots, i_2, i_1)$ is the sequence obtained by reading  consecutively the first components
of the green vertices in the slices $\cQ(m+k-1), \cQ(m+k-2), \ldots, \cQ(m)$.  Now, the sequence 
$(i_r, \ldots, i_1)$
is adapted to  $\cQ(m)$ by Proposition~\ref{prop:slices}(c).
Without loss
of generality we may assume that $j_1, j_2, \ldots, j_g$ are the first components of the green vertices of $\cQ(m+k)$.  Now,  we can apply Lemma~\ref{lem:reflt}  to the sequence
$(j_a, i_r, i_{r-1}, \ldots, i_1)$ for $1\leq a\leq g$ and $\cQ'= \cQ(m)$, and conclude that 
row $j_a$ of $G^{(k)}(m)$ is non-negative.  It follows by Lemma~\ref{lem:green1}, that
$G^{(k+1)}(m)$ also has the desired shape. 

The second claim follows since $\bT_i=\mathrm{Id}_n$ for  $i\geq 0$  or $i < h_c$.  \cqfd
\end{proof}

\begin{remark}\label{rem4-13}
{\rm
We deduce from Theorem~\ref{thm:g-vect1}  the following additional information about the columns
of $G^{(k)}$. Let $(i, l)\in I(m)\subset V$.
 \begin{itemize}
     \item 
     The support of $\bg_{(i,l)}^{(k)}$ is contained in $I(m)$.
     \item 
     If $m>0$,  we have $\bg_{(i,l)}^{(k)}=\be_{(i,l)}$.
     Note, that $\bT_{-1} \bT_{-2}\cdots \bT_{h_c}$  represents $w_0$, thus it is the negative of the permutation matrix for the Nakayama permutation $\nu$. Therefore, for $m\leq h_c$ and $k\gg 0$, we have
\[
     \bg_{(i, l_c(i)+2m)}^{(k)}=-\be_{(\nu(i), l_c(\nu(i))+2m)}.
\]     
 \end{itemize}   
}
\end{remark}

\begin{figure}[t]
\[
\def\objectstyle{\scriptstyle}
\def\lablestyle{\scriptstyle}
\xymatrix@-1.0pc{
{}\save[]+<0cm,1ex>*{\vdots}\restore&{}\save[]+<0cm,1ex>*{\vdots}\restore  
\\
{\be_{(1,2)}}\ar[rd]\ar[u]&
\\
&\ar[ld] \be_{(2,1)} \ar[uu]
\\
\ar[uu]{\mathbf{\red\be_{(1,0)}}}\ar[d]&
\\
\mathbf{\green-\be_{(1,-2)}+\be_{(2,-1)}}\ar[rd]
\\
&\ar[uuu] \mathbf{\red\be_{(2,-1)}}\ar[d] 
\\
& \ar[ld]\mathbf{\green-\be_{(1,-4)}}
\\
\ar[uuu]\mathbf{\red-\be_{(1,-4)}+\be_{(2,-3)}} \ar[d] 
\\
{\mathbf{\green-\be_{(2,-5)}}}\ar[rd]
\\
& \ar[uuu]{-\be_{(1,-6)}}\ar[ld] 
\\
\ar[uu]{-\be_{(2,-7)}} \ar[rd] 
\\
&\ar[ld] \ar[uu]{-\be_{(1,-8)}} 
\\
\ar[uu]{-\be_{(2,-9)}} 
\\
{}\save[]+<0cm,0ex>*{\vdots}\ar[u]\restore&{}\save[]+<0cm,0ex>*{\vdots}\ar[uu]\restore  
\\
}
\quad
\xymatrix@-1.0pc{
&{}\save[]+<0cm,1.5ex>*{\vdots}\restore&{}\save[]+<0cm,1.5ex>*{\vdots}\restore  
\\
&{\be_{(1,2)}}\ar[rd]\ar[u]&
\\
&&\ar[ld] \be_{(2,1)} \ar[uu]
\\
&\ar[uu]{{\be_{(1,0)}}}\ar[rdd]&
\\
&\mathbf{\blue{\be_{(1,-2)}}}\ar[u]\ar@/_{2pc}/[dddd]
\\
&&\ar[uuu]\ar[ldd] {\be_{(2,-1)}} 
\\
&& \ar[ld]\mathbf{\blue{-\be_{(2,-5)}+\be_{(1,-4)}}}\ar[ddd]
\\
&\mathbf{-\be_{(2,-3)}+\be_{(1,-2)}} \ar[uuu] 
\\
&{\blue{-\be_{(2,-5)}}}\ar[ruu]
\\
&& {-\be_{(1,-6)}}\ar[ld] 
\\
&\ar[uu] {-\be_{(2,-7)}} \ar[rd] 
\\
&&\ar[ld] \ar[uu] {-\be_{(1,-8)}} 
\\
&\ar[uu]{-\be_{(2,-9)}} 
\\
&{}\save[]+<0cm,0ex>*{\vdots}\ar[u]\restore&{}\save[]+<0cm,0ex>*{\vdots}\ar[uu]\restore  
\\
}
\quad
\xymatrix@-1.0pc{
{}\save[]+<0cm,1.5ex>*{\vdots}\restore&{}\save[]+<0cm,1.5ex>*{\vdots}\restore  
\\
{\be_{(1,2)}}\ar[rd]\ar[u]&
\\
&\ar[ld] \be_{(2,1)} \ar[uu]
\\
\ar[uu]\be_{(1,0)}\ar[rd]&
\\
&\ar[uu] \mathbf{\red\be_{(2,-1)}}\ar[d] 
\\
& \ar[ld]\mathbf{\green-\be_{(2,-3)}+\be_{(1,-2)}} 
\\
 \mathbf{\red\be_{(1,-2)}} \ar[uuu]\ar[d]
\\
\mathbf{\green-\be_{(2,-5)}} \ar[dr]
\\
& \ar[uuu]\ar[d]\mathbf{\red-\be_{(2,-5)}+\be_{(1,-4)}}
\\
& \mathbf{\green-\be_{(1,-6)}}\ar[ld] 
\\
\ar[uuu]{-\be_{(2,-7)}} \ar[rd] 
\\
&\ar[ld] \ar[uu] {-\be_{(1,-8)}}
\\
\ar[uu] {-\be_{(1,-9)}} 
\\
{}\save[]+<0cm,0ex>*{\vdots}\ar[u]\restore&{}\save[]+<0cm,0ex>*{\vdots}\ar[uu]\restore  
\\
}
\]
\caption{\label{Fig5b} {\it Mutation of stabilized $g$-vectors from $\G_{(1,2,1),(0,-1,-4)}$ to $\G_{(2,1,2),(-1,-2,-5)}$.}}
\end{figure}
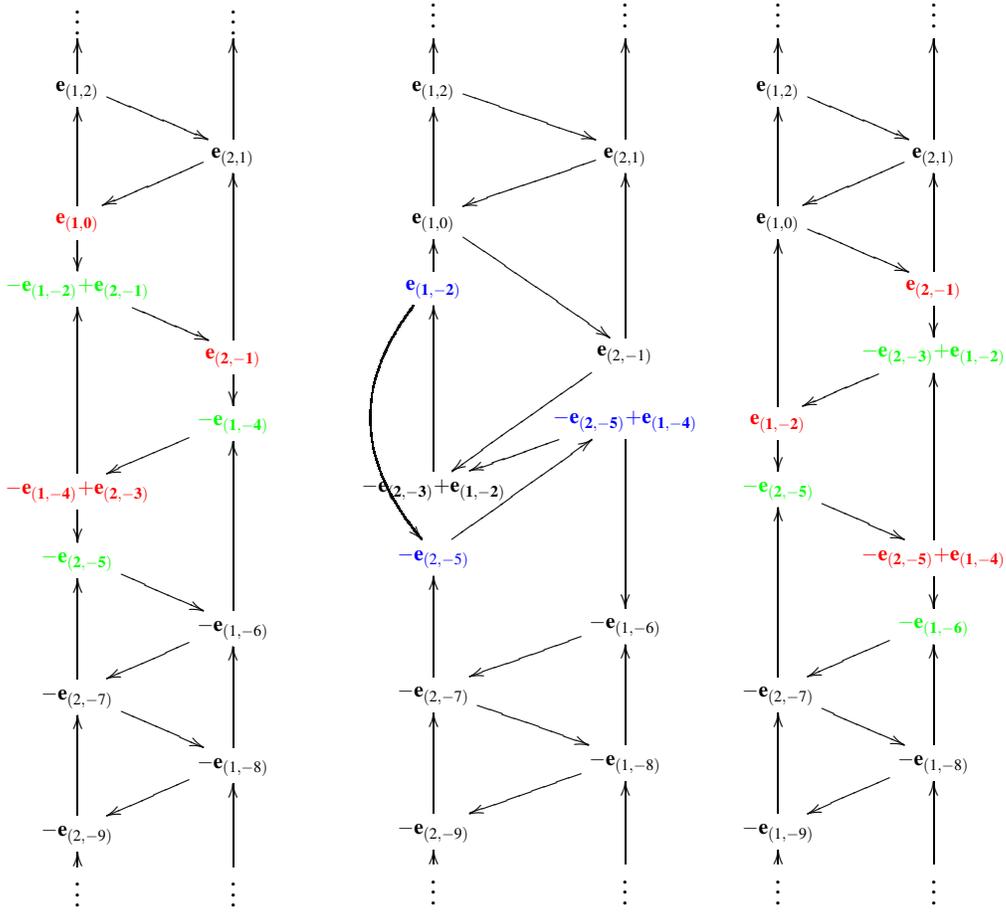

\begin{example}
{\rm
We illustrate Theorem~\ref{thm:g-vect1} in type $A_2$. This is a continuation of Example~\ref{exa-A2}.
Here the quiver $\cQ$ is $1 \leftarrow 2$, and $h_c=-3$. For $m\in \Z$, we have $I(m) = \{(1, 2m), (2,2m+1)\}$. 
We have
\[
 \bt_1 = 
 \pmatrix{-1 & 0 \cr 1 & 1},
 \qquad
 \bt_2 = 
 \pmatrix{1 & 1 \cr 0 & -1},
\]
and 
\[
\bT_m = 
\left\{
\begin{array}{cl}
\mathrm{Id}_2 & \mbox{if\ \ } m \ge 0, \\
\bt_1 & \mbox{if\ \ } m = -1, \\
\bt_2 & \mbox{if\ \ } m = -2, \\
\bt_1 & \mbox{if\ \ } m = -3, \\
\mathrm{Id}_2 & \mbox{if\ \ } m \le -4. \\
\end{array}
\right.
\]
Hence,
\[
\bT_{-1}\bT_{-2} = \pmatrix{-1 & -1 \cr 1 & 0},
 \qquad
\bT_{-2}\bT_{-3} = \pmatrix{0 & 1 \cr -1 & -1},
\qquad
\bT_{-1}\bT_{-2}\bT_{-3} = \pmatrix{0 & -1 \cr -1 & 0}.
\]
The non-trivial diagonal blocks of $G^{(1)}$ are:
\begin{eqnarray*}
G^{(1)}(-1) &=& \bT_{-1}, \\
G^{(1)}(-2) &=& \bT_{-2}, \\
G^{(1)}(-3) &=& \bT_{-3}. 
\end{eqnarray*}
The non-trivial diagonal blocks of $G^{(2)}$ are:
\begin{eqnarray*}
G^{(2)}(-1) &=& \bT_0\bT_{-1} = \bT_{-1}, \\
G^{(2)}(-2) &=& \bT_{-1}\bT_{-2}, \\
G^{(2)}(-3) &=& \bT_{-2}\bT_{-3}, \\
G^{(2)}(-4) &=& \bT_{-3}\bT_{-4} = \bT_{-3}. 
\end{eqnarray*}
The non-trivial diagonal blocks of $G^{(3)}$ are:
\begin{eqnarray*}
G^{(3)}(-1) &=& \bT_1\bT_0\bT_{-1} = \bT_{-1}, \\
G^{(3)}(-2) &=& \bT_0\bT_{-1}\bT_{-2} = \bT_{-1}\bT_{-2}, \\
G^{(3)}(-3) &=& \bT_{-1}\bT_{-2}\bT_{-3}, \\
G^{(3)}(-4) &=& \bT_{-2}\bT_{-3}\bT_{-4} = \bT_{-2}\bT_{-3},\\
G^{(3)}(-5) &=& \bT_{-3}\bT_{-4}\bT_{-5} = \bT_{-3}.
\end{eqnarray*}
And so on. Thus, we see that in the limit the non trivial blocks of the stabilized $G$-matrix 
$G^{(\infty)}$ are
\begin{eqnarray*}
G^{(\infty)}(-1) &=& \bT_{-1}, \\
G^{(\infty)}(-2) &=& \bT_{-1}\bT_{-2}, \\
G^{(\infty)}(m) &=& \bT_{-1}\bT_{-2}\bT_{-3},\ (m\le -3),
\end{eqnarray*}
in agreement with the left quiver of Figure~\ref{Fig5b}.
}
\end{example}

\begin{remark}\label{rem-mut-g-vector}
{\rm
Now that we have determined the stabilized $g$-vectors of the cluster variables of an initial seed, we can use Equation~(\ref{recursion2}) to calculate the stabilized $g$-vector of a cluster variable obtained from this initial seed by any explicit sequence of mutations. 

For instance, if we replace in Figure~\ref{Fig5} the original labelling of the vertices by the corresponding stabilized $g$-vectors we get Figure~\ref{Fig5b}. In Figure~\ref{Fig5b}, the middle seed is obtained from the left seed by the same sequence of three mutations as in Figure~\ref{Fig5}. The stabilized $g$-vectors of the three new cluster variables, calculated using Equation~(\ref{recursion2}), are painted in blue. The right seed is identical to the middle one : it is obtained by just moving the lower blue vertex in the left column to the right column and adjusting correspondingly the arrows. Note that now there is no need of relabelling the vertices like we did in Figure~\ref{Fig5}.
}
\end{remark}

\subsubsection{A knitting algorithm} \label{ssec:dshift}

We will now deduce from Theorem~\ref{thm:g-vect1} a simple algorithm for calculating 
the stabilized $g$-vectors.  

For $\mathbf{v}=\left(v_{(i,l)}\right)_{(i,l)\in V}\in\Z^{V}$ and  
$s\in\Z$, we define the \emph{shifted vector}
$\mathbf{v}[s]=\left(v'_{(i,l)}\right)_{(i,l)\in V}$ by:
\[
v'_{i,l}:= v_{i, l-2s},\qquad  ((i,l)\in V).
\]

\begin{Thm} \label{thm:g-vect2}
Recall the notation of Theorem~\ref{thm:g-vect1}. For $(i,l)\in V$, denote by 
    $\bg^{(k)}_{(i,l)}\in\Z^{V}$ the column vector of $G^{(k)}$ in column $(i,l)$.
Suppose that $k\gg 0$.    
    
\begin{itemize}
 \item[(a)]
 If there is a vertical \emph{up}-arrow $(i,l)\ra (i,l+2)$ in $\G$, then we have: 
    \begin{equation}\label{eq-knit1}
    \bg^{(k)}_{(i,l)}=\bg^{(k)}_{(i,l+2)}[-1]. 
    \end{equation}
 \item[(b)]
If there is a vertical \emph{down}-arrow $(i,l+2)\ra (i,l)$ in $\G$, then we have:  
\begin{equation}\label{eq-knit2}
    \bg^{(k)}_{(i,l)}
    = -\bg^{(k)}_{(i,l+2)}[-1] + \sum_{(i,l)\to (j,s)}\bg^{(k)}_{(j,s)}.
\end{equation}
\end{itemize}
\end{Thm}

\begin{proof}
    Suppose that $(i,l)\in I(m)$. 
    If $m \ge 0$ then we are always in case (a), and  $\bg^{(k)}_{(i,l)} = \be_{(i,l)}$ by construction, so Equation~(\ref{eq-knit1}) is trivially verified. 
    Similarly, if $m < h$ then we are also in case (a) and Equation~(\ref{eq-knit1}) follows obviously from the second point of Remark~\ref{rem4-13}. 
    
    So we can assume that $h\le m <0$. By Eq.~(\ref{eq:Gkgg0}), we have 
    \[
    G^{(k)}(m)= G^{(k)}(m+1)\cdot \bT_{m}.
    \]
    
    If we are in case (a) then $(i,l)$ is \emph{not} a green vertex of $\G$, thus by definition of $\bT_{m}$,
    column $i$ of $\bT_{m}$ is the unit vector $\be_i$, so the matrices $G^{(k)}(m)$ and 
    $G^{(k)}(m+1)$ have the same column $i$, which is precisely the content of Equation~(\ref{eq-knit1}). 
    
    If we are in case (b) then $(i,l)$ is a green vertex of $\G$, and again by definition of $\bT_{m}$, column $i$ of $\bT_{m}$ has an entry $-1$ on row $i$, an entry $+1$ on every row $j$ such that $c_{ij} = -1$, and $0$ on all other rows. By Proposition~\ref{prop:slices}, $(i,l)$ is a source of $ \cQ(m)$ and $(i,l+2)$ is a red vertex which is a sink of $\cQ(m+1)$. 
    Thus we obtain
    \[
    \bg^{(k)}_{(i,l)}=\left(-\bg^{(k)}_{(i,l+2)} + \sum_{(j,s+2)\to (i,l+2)}\bg^{(k)}_{(j,s+2)}\right)[-1] 
    = -\bg^{(k)}_{(i,l+2)}[-1] + \sum_{(i,l)\to (j,s)}\bg^{(k)}_{(j,s)},
    \]
    where the second equality follows from case (a).
    \cqfd
\end{proof}

Using Theorem~\ref{thm:g-vect2}, we can calculate very easily the stabilized $g$-vectors $\bg^{(\infty)}_{(i,l)}$, starting
from the initial datum $\bg^{(\infty)}_{(i,l)} = \be_{(i,l)}$ for all vertices $(i,l)$ of the 
upper slices $\cQ(m)\ (m\ge 0)$, and then going down using alternately equations (\ref{eq-knit1}) and 
(\ref{eq-knit2}). This algorithm resembles the knitting algorithm of Auslander-Reiten theory. 
More precisely, remember that the green vertices of $\G$ are in one-to-one correspondence with the 
vertices of the Auslander-Reiten quiver of $\cQ$. The following easy corollary of Theorem~\ref{thm:g-vect2}, which involves only green vertices, can be regarded as an analogue of the classical \emph{mesh relations}.

\begin{Cor}
Suppose that $(i,l)$ and $(i,l-4)$ are green vertices of $\G$. Then for every (oblique) arrow $(i,l) \to (j,s)$, vertex $(j,s-2)$ is green, and we have
\[
 \bg^{(\infty)}_{(i,l-4)} + \bg^{(\infty)}_{(i,l)}[-2] = \sum_{(i,l)\to (j,s)} \bg^{(\infty)}_{(j,s-2)}[-1].
\]
\end{Cor}

\subsubsection{Reformulation in terms of a braid group action}\label{subsec434}

We now give a reformulation of the previous calculations in terms of a braid group action on $\Z^V$. 
\begin{Def}
For $i\in I$ we define an automorphism $\theta_i$ of the free $\Z$-module $\Z^V$ by:
\[
\theta_i\left(\be_{(j,a)}\right) :=
\left\{
\begin{array}{ll}
\be_{(j,a)} &\mbox{if } j\neq i,\\[2mm]
-\be_{(i,a-2)} + \ds\sum_{k\, :\, c_{ik}=-1} \be_{(k,a-1)} &\mbox{if } k=i,   
\end{array}
\right.
\qquad ((j,a)\in V).
\]
\end{Def}

The following lemma shows that the $\theta_i\ (i\in I)$ define an action on $\Z^V$ of the braid group attached to the Weyl group $W$. This is similar to the braid group action defined by Chari in 
\cite{C}. 
\begin{Lem}
The automorphisms $\theta_i\ (i\in I)$ satisfy the braid relations. 
\end{Lem}

\begin{proof}
Let $i,j$ be two distinct elements of $I$.
If $k\not\in \{i,j\}$ then 
$\theta_i(\be_{(k,a)}) = \theta_j(\be_{(k,a)}) = \be_{(k,a)}$ for every $(k,a)\in V$, 
and the braid relations are trivially verified on $\be_{(k,a)}$.

If $c_{ij}=0$, then 
\[
\theta_i\theta_j(\be_{(i,a)}) 
= \theta_i(\be_{(i,a)}) 
= -\be_{(i,a-2)} + \sum_{k\, :\, c_{ik}=-1} \be_{(k,a-1)}
= \theta_j\theta_i(\be_{(i,a)})
\]
since all indices $k$ occurring in the sum are different from $j$.

If $c_{ij}=-1$, then 
\[
\theta_j\theta_i\theta_j(\be_{(i,a)}) 
= \theta_j\left(-\be_{(i,a-2)} + \sum_{k\, :\, c_{ik}=-1} \be_{(k,a-1)}\right)
= -\be_{(j,a-3)} + \sum_{k\not = j\, :\, c_{ik}=-1} \be_{(k,a-1)}
+ \sum_{l\not = i\, :\, c_{jl}=-1} \be_{(l,a-2)}.
\]
Since this sum does not contain any vector of the form $\be_{(i,b)}$, it is invariant under $\theta_i$.
Hence 
\[
 \theta_j\theta_i\theta_j(\be_{(i,a)}) = \theta_j\theta_i(\be_{(i,a)}) = \theta_i\theta_j\theta_i(\be_{(i,a)}),
\]
so the braid relations are also satisfied on $\be_{(i,a)}$. 
\cqfd
\end{proof}

Recall that reading the first indices of the green vertices of $\G$ from top to bottom we get 
a sequence $\bi = (i_1,\ldots ,i_N)\in I^N$ such that $s_{i_1}\cdots s_{i_N}$ is a reduced decomposition of $w_0\in W$. For example, in Figure~\ref{Fig7f} we have $N= 10$ and
$\bi = (1,2,4,1,3,2,4,1,3,2)$.

Let $((i_1,a_1)\ldots , (i_N,a_N))\in V^N$ denote the sequence of green vertices 
of $\G$ in the same ordering.
For $1\le t \le N$, let
$
 s_t := \sharp\{j<t\mid i_j = i_t\}.
$
Finally, for $i\in I$, let $(i,m_i)\in V$ denote the highest red vertex of $\G$ 
in column $i$.

\begin{Prop}\label{prop-braid}
For $1\le t \le N$ we have
\[
 \bg_{(i_t,a_t)}^{(\infty)} = 
 \theta_{i_1}\theta_{i_2}\cdots \theta_{i_t}(\be_{(i_t,m_{i_t})})[-s_t].
\]
\end{Prop}
\begin{proof}
This follows readily from Theorem~\ref{thm:g-vect2} by induction on $1\le t\le N$. To avoid cumbersome notation, we will just check it below on a typical example. \cqfd 
\end{proof}

\begin{example}
{\rm
We consider again the quiver of Example~\ref{examples-stables} in type $A_3$. Thus 
$\bi = (1,2,1,3,2,1)$, and the sequence of green vertices is
$((1,-2), (2,-3), (1,-6), (3,-4), (2,-7), (1,-10))$. 
The highest red vertices are
$(1,m_1) = (1,0)$, $(2,m_2) = (2,-1)$, $(3,m_3) = (3,-2)$.
We have
\[
\begin{array}{lll}
&\theta_1(\be_{(1,0)})&= -\be_{(1,-2)} + \be_{(2,-1)} = \bg_{(1,-2)}^{(\infty)},\\[2mm]
&\theta_1\theta_2(\be_{(2,-1)})&= \theta_1\left(-\be_{(2,-3)} + \be_{(1,-2)} + \be_{(3,-2)}\right)
= -\be_{(1,-4)}+\be_{(3,-2)} = \bg_{(2,-3)}^{(\infty)},\\[2mm]
&\theta_1\theta_2\theta_1(\be_{(1,0)})[-1] &= \theta_1\theta_2\left(-\be_{(1,-4)} + \be_{(2,-3)}\right)
= \theta_1\left(-\be_{(2,-5)} + \be_{(3,-4)}\right) 
\\[2mm]
&&= -\be_{(2,-5)} + \be_{(3,-4)} = \bg_{(1,-6)}^{(\infty)},\\[2mm]
&\theta_1\theta_2\theta_1\theta_3(\be_{(3,-2)}) &= \theta_1\theta_2\left(-\be_{(3,-4)} + \be_{(2,-3)}\right) = \theta_1\left(-\be_{(2,-5)} + \be_{(1,-4)}\right)
\\[2mm]
&&= -\be_{(1,-6)} = \bg_{(3,-4)}^{(\infty)},\\[2mm]
&\theta_1\theta_2\theta_1\theta_3\theta_2(\be_{(2,-1)})[-1] &=
\theta_1\theta_2\theta_1\theta_3\left(-\be_{(2,-5)}+\be_{(1,-4)}+\be_{(3,-4)}\right)
\\[2mm]
&&= \theta_1\theta_2\left(-\be_{(1,-6)}+\be_{(2,-5)}-\be_{(3,-6)}\right)
\\[2mm]
&&= \theta_1(-\be_{(2,-7)}) = -\be_{(2,-7)} = \bg_{(2,-7)}^{(\infty)},\\[2mm]
&\theta_1\theta_2\theta_1\theta_3\theta_2\theta_1(\be_{(1,0)})[-2] &=
\theta_1\theta_2\theta_3\left(-\be_{(2,-7)} + \be_{(3,-6)}\right)
=\theta_1\theta_2(-\be_{(3,-8)})
\\[2mm]
&&= -\be_{(3,-8)} = \bg_{(1,-10)}^{(\infty)}.
\end{array}
\]
}
\end{example}

\begin{remark}
{\rm
Proposition~\ref{prop-braid} gives an expression of the stabilized $g$-vectors $\bg_{(i,a)}^{(\infty)}$ for all green vertices $(i,a)$ of $\G$. This in fact determines \emph{all} stabilized $g$-vectors since, by Equation~(\ref{eq-knit1}), all the remaining ones are obtained from these by mere degree shifts.
So for \emph{every} vertex $(i,a)\in V$ the stabilized $g$-vector $\bg^{(\infty)}_{(i,a)}$ can be written in the form
\begin{equation}\label{form-theta-g-vect}
 \bg^{(\infty)}_{(i,a)} = \theta_{i_1}\cdots\theta_{i_{t}}(\be_{(i_t,m_{i_t})})[s]
\end{equation}
for some well-defined $0\le t \le N$ and $s\in\Z$. Here it is understood that when $t=0$, the monomial in the $\theta_i$ is empty. In that case $(i,a)$ belongs to the upper part of $\G$ and $\bg^{(\infty)}_{(i,a)} = \be_{(i,a)}$, a positive degree shift of $\bg^{(\infty)}_{(i,m_i)} = \be_{(i,m_i)}$.  
}
\end{remark}


\section{Rings of formal power series and Weyl group action}\label{sect-formal-power-series}

Following \cite{FH2}, we introduce certain rings of formal power series related to the representation theory of $U_q(\widehat{\g})$ and its shifted versions $U_q^\mu(\widehat{\g})$. (These connections will be explained in Section~\ref{sect-shift-quantum} below.) 
Then, again following \cite{FH2}, we introduce automorphisms of these rings which generate an action of the Weyl group $W$ of $\g$.

\subsection{Variables}

We fix $q\in\C^*$ of infinite multiplicative order. 
We denote by $\{\Psi_{i,a} \mid i\in I,a \in \C^*\}$ an infinite set of commuting variables. To every $\la\in P$ we also attach a commutative variable $[\la]$ such that 
\[
 [\la][\mu] = [\la + \mu],\quad (\la,\mu \in P).
\]
Next we introduce for $i\in I$ and $a\in\C^*$ the following Laurent monomials:
\begin{eqnarray}
 Y_{i,a} &:=& [\varpi_i]\frac{\Psi_{i,aq^{-1}}}{\Psi_{i,aq}}, 
\\
A_{i,a} &:=&
Y_{i,aq^{-1}}Y_{i,aq}
\prod_{j :\ c_{j,i} = -1}Y_{j,a}
^{-1},
\\ \label{def-tilde-Psi}
\widetilde{\Psi}_{i,a}  &:=&  \Psi_{i,a}^{-1}\prod_{j :\ c_{ij}=-1}\Psi_{j,aq}.
\end{eqnarray}
These monomials satisfy an important relation:
\begin{Lem}\label{Lem31}
For every $i\in I$ and $a\in \C^*$, we have:
\[
Y_{i,a} A_{i,aq^{-1}}^{-1}
=
[\varpi_i - \a_i] \frac{\widetilde{\Psi}_{i,aq^{-3}}}{\widetilde{\Psi}_{i,aq^{-1}}}.
\]
\end{Lem}
\begin{proof}
It follows from the definition of $\widetilde{\Psi}_{i,a}$ that
\[
\frac{\widetilde{\Psi}_{i,aq^{-3}}}{\widetilde{\Psi}_{i,aq^{-1}}} =
\frac{\Psi_{i,aq^{-1}}}{\Psi_{i,aq^{-3}}}
\prod_{j :\ c_{ij}=-1} \frac{\Psi_{j,aq^{-2}}}{\Psi_{j,a}}
=[-\varpi_i + \alpha_i] Y_{i,aq^{-2}}^{-1}
\prod_{j :\ c_{ij}=-1} Y_{j,aq^{-1}}.
\]
The desired equality follows by comparison with the definition of $A_{i,aq^{-1}}^{-1}$. 
\cqfd 
\end{proof}

\subsection{The ring $\Pi$}\label{subsec-Pi}\label{sect3-2}

Let $\YY$ denote the ring of Laurent polynomials: 
\[
 \YY := \Z[Y_{i,a}^{\pm1};\ i\in I,\ a\in\C^*].
\]
Let $\M\subset \YY$ denote the multiplicative group of Laurent monomials in the variables $Y_{i,a}$. We denote
by $\omega$ the group homomorphism from $\M$ to $P$ defined by 
\[
 \omega(Y_{i,a}) := \varpi_i, \qquad (i\in I,\ a\in \C^*).
\]
Then we also have 
\[
 \omega(A_{i,a}) := \a_i, \qquad (i\in I,\ a\in \C^*).
\]
In \cite[Definition 2.1]{FH2}, completions $\widetilde{\YY}_w$ of this ring were introduced for every $w\in W$. The elements of $\widetilde{\YY}_w$ are formal power series of the form
\[
 \sum_{m\in S} a_m\,m,
\]
where $S$ is any subset of $\M$ such that $w(\omega(S))$ is contained in a finite union of cones of $P$ of the form
\[
 C_\la := \la - Q_+, \qquad (\la\in P),
\]
and for any $\la \in P$, the number of $a_m\not = 0$ with $\omega(m) = \la$ is finite.

\begin{example}\label{Exa3-2}
{\rm
The formal power series ring
\[
 \Z[[A_{i,a}^{-1}; i\in I,\ a\in \C^*]]
\]
is contained in $\widetilde{\YY}_e$, because $\omega(A_{i,a}^{-1}) = -\a_i \in - Q_+$ for all $i\in I$.
On the other hand, the ring 
\[
 \Z[[A_{i,a}; i\in I,\ a\in \C^*]]
\]
is contained in $\widetilde{\YY}_{w_0}$, since $w_0(\omega(A_{i,a})) = w_0(\a_i) \in - Q_+$ for all $i\in I$. 
Now fix $i\in I$ and $a\in\C^*$, and consider the series:
\begin{eqnarray}
\Si_{i,a}^+ &:=& \sum_{k\ge 0}\, \prod_{0<j\le k} A_{i,aq^{-2(j-1)}}^{-1} 
= 1 + A_{i,a}^{-1} + A_{i,a}^{-1}A_{i,aq^{-2}}^{-1} + A_{i,a}^{-1}A_{i,aq^{-2}}^{-1}A_{i,aq^{-4}}^{-1} +\cdots ,
\\
\Si_{i,a}^- &:=& -\sum_{k > 0}\, \prod_{0<j\le k} A_{i,aq^{2j}} 
= -\left(A_{i,aq^{2}} + A_{i,aq^{2}}A_{i,aq^{4}} + A_{i,aq^{2}}A_{i,aq^{4}}A_{i,aq^{6}} +\cdots\right). 
\end{eqnarray}
Then $\Si_{i,a}^+ \in \widetilde{\YY}_{w}$ for every $w\in W$ such that $w(\a_i)>0$, and 
$\Si_{i,a}^- \in \widetilde{\YY}_{w}$ for every $w\in W$ such that $w(\a_i)<0$. \cqfd
}
\end{example}

Next, one defines \cite[\S 2.5]{FH2} the ring
\[
 \Pi := \bigoplus_{w\in W} \widetilde{\YY}_w,
\]
in which addition and multiplication are performed component-wise. This comes with a diagonal embedding
$\YY \to \Pi$, and with projections $E_w : \Pi \to \widetilde{\YY}_w$ for every $w\in W$.

The following elements of $\Pi$ introduced in \cite[Definition 2.7]{FH2} play a key role.
\begin{Def}
For $i\in I$, $a\in \C^*$, and $w\in W$ put: 
\[
\Si_{i,a}^w := \Si_{i,a}^+\ \mbox{if} \ w(\a_i) > 0,\quad \mbox{and} \quad
\Si_{i,a}^w := \Si_{i,a}^- \ \mbox{if}\ w(\a_i) < 0.  
\]
Then by Example~\ref{Exa3-2} we have that $\Si_{i,a}^w \in \widetilde{\YY}_{w}$ for every $w\in W$. 
Define:
\[
 \Si_{i,a} := (\Si_{i,a}^w)_{w\in W} \in \Pi.
\]
\end{Def}

Then for every $i\in I$, $a\in \C^*$, $w\in W$ the element $\Si_{i,a}^w$ is invertible in $\widetilde{\YY}_{w}$. It follows that $\Si_{i,a}$ is invertible in $\Pi$ (see \cite[Lemma 2.8]{FH2}).

\subsection{The Weyl group action on $\Pi$}

Following \cite[\S3]{FH2}, for every $i\in I$ we introduce a ring endomorphism $\Theta_i = (\Theta_i^w)_{w\in W}$ of $\Pi$. Each component $\Theta_i^w$ is in fact a ring homomorphism from 
$\widetilde{\YY}_{w}$ to $\widetilde{\YY}_{ws_i}$ defined by :
\[
 \Theta_i^w(Y_{j,a}) = 
\left\{ 
\begin{array}{lc}
Y_{j,a} & \mbox{ if }j \not = i, \\[2mm]
Y_{i,a}A_{i,aq^{-1}}^{-1}\displaystyle\frac{\Si_{i,aq^{-3}}^{ws_i}}{\Si_{i,aq^{-1}}^{ws_i}} & \mbox{ if }j = i.
\end{array}
\right.
\]
It is shown in \cite{FH2} that these formulas combine into a well-defined continuous ring endomorphism $\Theta_i$ of $\Pi$.

\begin{example}
{\rm
We fix $\g = \mathfrak{sl}_2$. In this case $I = \{1\}$, so we can drop the index $i$ for simplicity of notation. Here $W = \{e,s\}$. Consider the Laurent polynomials:
\[
P_a := Y_a + Y_{aq^2}^{-1} \in \Z[Y_a^{\pm 1}; a\in \C^*]. 
\]
(These are the $q$-characters of the fundamental representations $L(Y_a)$ of 
$U_q(\widehat{\mathfrak{sl}}_2)$.) By the diagonal embedding of $\Z[Y_a^{\pm 1}; a\in \C^*]$ into $\Pi$, we get elements $(P_a^e, P_a^s) \in \Pi$, where $P_a^e = P_a^s = P_a$. We want to calculate their image under $\Theta$. Applying definitions, we have: 
\[
 \Theta^e(P_a^e) = Y_aA^{-1}_{aq^{-1}}\frac{\Si^-_{aq^{-3}}}{\Si^-_{aq^{-1}}} 
 + Y_{aq^2}^{-1}A_{aq}\frac{\Si^-_{aq}}{\Si^-_{aq^{-1}}}
= (\Si^-_{aq^{-1}})^{-1}\left(Y_{aq^{-2}}^{-1}\Si^-_{aq^{-3}} + Y_a\Si^-_{aq}\right)  
\in \widetilde{\YY}_{s}.
\]
Now, 
\[
 \Si^-_{aq^{-3}} = -A_{aq^{-1}} + A_{aq^{-1}}\Si^-_{aq^{-1}}, \qquad
 \Si^-_{aq} = 1 + A_{aq}^{-1}\Si^-_{aq^{-1}},
\]
hence
\[
 Y_{aq^{-2}}^{-1}\Si^-_{aq^{-3}} = -Y_a + Y_a\Si^-_{aq^{-1}}, \qquad
 Y_a\Si^-_{aq} = Y_a + Y_{aq^2}^{-1}\Si^-_{aq^{-1}},
\]
so that 
\[
 \Theta^e(P_a^e) = Y_a + Y_{aq^2}^{-1} = P_a^s \in \widetilde{\YY}_{s}.
\]
A similar calculation using the sums $\Si^+$ instead of $\Si^-$ shows that we also have:
\[
 \Theta^s(P_a^s) = Y_a + Y_{aq^2}^{-1} = P_a^e \in \widetilde{\YY}_{e}.
\]
Therefore, we get that $\Theta(P_a^e, P_a^s) = (P_a^e, P_a^s) \in \Pi$, that is, the diagonal embedding 
of $P_a$ in $\Pi$ is invariant by $\Theta$, a simple illustration of \cite[Theorem 5.1]{FH2}. \cqfd
}
\end{example}

\begin{Thm}[\cite{FH2}]\label{Thm-Weyl-relations}
The endomorphisms $\Theta_i$ satisfy the relations of the Coxeter generators $s_i\in W$. Hence the assignment $s_i \mapsto \Theta_i$ defines an action of the Weyl group $W$ on the ring $\Pi$ by ring automorphisms. 
\end{Thm}

We quote the following important formulas proved in \cite{FH2}:
\begin{equation}\label{form-7}
\Theta_i(A_{i,a}^{-1}) = A_{i,aq^{-2}}\frac{\Si_{i,a}}{\Si_{i,aq^{-4}}},\qquad
\Theta_i(\Si_{i,a}) = 1 - \Si_{i,a} = -A_{i,a}^{-1}\Si_{i,aq^{-2}}, \qquad (i\in I,\ a \in \C^*), 
\end{equation}
in which $A_{i,a} \in \YY$ is identified with its diagonal embedding $(A_{i,a})_{w\in W} \in \Pi$.

\subsection{The ring $\Pi'$ and its automorphisms} \label{subsection-Pi-prime}

In this paper, we will need to extend the construction of \cite{FH2} to a slightly larger ring. 
First we extend the ring $\YY$ by adding the new variables $\Psi_{i,a}$ and $[\la]$, and define:
\[
\YY' := \Z[\Psi_{i,a}^{\pm1}, [\la];\ i\in I,\ a\in\C^*,\ \la\in P]\ \supset\ \YY. 
\]
Let $\M'\subset \YY'$ denote the multiplicative group of Laurent monomials in the variables $\Psi_{i,a}$ and the variables $[\la]$. We denote
by $\omega'$ the group homomorphism from $\M'$ to $P$ defined by 
\[\omega'(\Psi_{i,a}) := 0,\qquad
 \omega'([\la]) := \la, \qquad (i\in I,\ a\in \C^*,\ \la\in P).
\]
The restriction of $\omega'$ to the subgroup $\M$ of $\M'$ coincides with the homomorphism $\omega$ defined above.

For each $w\in W$, we define a completion $\widetilde{\YY_w}'$ of $\YY'$ in a similar way as above, namely the elements 
of $\widetilde{\YY_w}'$ are formal power series of the form
\[
 \sum_{m\in S} a_m\,m,
\]
where $S$ is any subset of $\M'$ such that $w(\omega(S))$ is contained in a finite union of cones of $P$ of the form $C_\la$ ($\la\in P$) 
and for any $\la \in P$, the number of $a_m\not = 0$ with $\omega(m) = \la$ is finite.

\begin{Def}
We put 
\[
 \Pi' := \bigoplus_{w\in W} \widetilde{\YY_w}',
\]
and we denote by $E'_w : \Pi' \to \widetilde{\YY_w}'$ the projection morphisms.
\end{Def}
By construction, each component $\widetilde{\YY_w}$ is a subalgebra of $\widetilde{\YY_w}'$, and so $\Pi$ is a subalgebra of $\Pi'$.

For every $i\in I$ we introduce a ring endomorphism $\tT_i = (\tT_i^w)_{w\in W}$ of $\Pi'$. Each component $\tT_i^w$ is in fact a ring homomorphism from 
$\widetilde{\YY_{w}}'$ to $\widetilde{\YY_{ws_i}}'$ defined by :
\[
\tT_i^w([\la]) = [s_i(\la)],
\qquad
 \tT_i^w(\Psi_{j,a}) = 
\left\{ 
\begin{array}{lc}
\Psi_{j,a} & \mbox{ if }j \not = i, \\[2mm]
(1-[-\a_i])\widetilde{\Psi}_{i,aq^{-2}}\Si_{i,aq^{-2}}^{ws_i} & \mbox{ if }j = i,
\end{array}
\right.
\]
where the Laurent monomial $\widetilde{\Psi}_{i,aq^{-2}}$ is defined in Eq.~(\ref{def-tilde-Psi}). 
As before, these formulas combine into a well-defined continuous ring endomorphism $\tT_i$ of $\Pi'$.

\begin{Prop}\label{Prop-restriction}
The restriction of $\tT_i$ to $\Pi$ is equal to $\Theta_i$. 
\end{Prop}
\begin{proof}
It is enough to check that $\tT_i(Y_{j,a}) = \Theta_i(Y_{j,a})$ for every $j\in I$ and $a\in\C^*$. 
We have
\[
 Y_{j,a} = [\varpi_j] \frac{\Psi_{j,aq^{-1}}}{\Psi_{j,aq}},
\]
hence if $j\not = i$, since $\tT_i([\varpi_j]) = [s_i(\varpi_j)] = [\varpi_j]$ and $\tTheta_i(\Psi_{j,b}) = \Psi_{j,b}$ for every $b\in\C^*$, we get that $\tT_i(Y_{j,a}) = Y_{j,a}$.
Otherwise, if $j=i$ we have for every $w\in W$, using Lemma~\ref{Lem31},
\[
\tT_i^w(Y_{i,a}) = [s_i(\varpi_i)]\frac{\tT_i^w(\Psi_{i,aq^{-1}})}{\tT_i^w(\Psi_{i,aq})}
= [\varpi_i - \a_i] \frac{\widetilde{\Psi}_{i,aq^{-3}}}{\widetilde{\Psi}_{i,aq^{-1}}}
\frac{\Si_{i,aq^{-3}}^{ws_i}}{\Si_{i,aq^{-1}}^{ws_i}}
= Y_{i,a} A_{i,aq^{-1}}^{-1}\frac{\Si_{i,aq^{-3}}^{ws_i}}{\Si_{i,aq^{-1}}^{ws_i}}
= \Theta_i^w(Y_{i,a}),
\]
hence $\tT_i(Y_{i,a}) = \Theta_i(Y_{i,a})$. \cqfd
\end{proof}

The endomorphisms $\tTheta_i$ are also involutions of $\Pi'$, as shown by the next lemma.

\begin{Lem}\label{Lem-38}
For every $i\in I$ and $a\in\C^*$, we have 
\[
(\tTheta_i)^2(\Psi_{i,a}) = \Psi_{i,a}.  
\]
\end{Lem}
\begin{proof}
Since $\Si_{i,aq^{-2}} \in \Pi$, it follows from Proposition~\ref{Prop-restriction} and from Equation (\ref{form-7}) that
\[
 \tTheta_i(\Si_{i,aq^{-2}}) = \Theta_i(\Si_{i,aq^{-2}}) = -A_{i,aq^{-2}}^{-1}\Si_{i,aq^{-4}}.
\]
On the other hand, for $j\not = i$ we have $\tTheta_i(\Psi_{j,b}) = \Psi_{j,b}$ for every $b\in\C^*$, therefore
\[
 \tTheta_i(\widetilde{\Psi}_{i,aq^{-2}}) = \tTheta_i(\Psi_{i,aq^{-2}}^{-1})
 \Psi_{i,aq^{-2}}\widetilde{\Psi}_{i,aq^{-2}}
 = (1-[-\a_i])^{-1}\widetilde{\Psi}_{i,aq^{-4}}^{-1}\Si_{i,aq^{-4}}^{-1}\Psi_{i,aq^{-2}}\widetilde{\Psi}_{i,aq^{-2}}. 
\]
Hence, using Lemma~\ref{Lem31}, we get 
\begin{eqnarray*}
 \tTheta_i(\widetilde{\Psi}_{i,aq^{-2}}) &=& (1-[-\a_i])^{-1}[-\a_i+\varpi_i]\Si_{i,aq^{-4}}^{-1}\Psi_{i,aq^{-2}} 
 Y_{i,aq^{-1}}^{-1}A_{i,aq^{-2}} \\
 &=& (1-[-\a_i])^{-1}[-\a_i]\Si_{i,aq^{-4}}^{-1}\Psi_{i,a} A_{i,aq^{-2}},
\end{eqnarray*}
and finally, 
\begin{eqnarray*}
(\tTheta_i)^2(\Psi_{i,a}) &=& \tTheta_i(1-[-\a_i])\tTheta_i(\widetilde{\Psi}_{i,aq^{-2}}) \tTheta_i(\Si_{i,aq^{-2}})
\\
&=& (1-[\alpha_i])(1-[-\a_i])^{-1}(-[-\a_i])\,\Psi_{i,a}
\\
&=& \Psi_{i,a}.
\end{eqnarray*}
\cqfd 
\end{proof}

Like the automorphisms $\Theta_i$, the automorphisms $\tTheta_i$ are compatible with the action of the ordinary simple reflexions $s_i$. 
Recall the subgroup $\overline{\mathcal{M}}$ of invertible elements of $\Pi$ introduced in \cite[Section 6.4]{FH2}.
We consider the subgroup $\overline{\mathcal{M}}'$ of invertible elements in $\Pi'$ generated by $\overline{\mathcal{M}}$, by the subgroup $\M'$ defined at the beginning of \S\ref{subsection-Pi-prime}, and by the elements $(1  - [\alpha])$ with $\alpha \in \Delta$.

By \cite[\S6]{FH2}, $\Theta_i$ defines an automorphism of  $\overline{\mathcal{M}}$.
But, from the defining formulas of $\tTheta_i$, we have that $\tTheta_i(\Psi_{j,b})\in \overline{\mathcal{M}}'$
for any $j\in I$ and $b\in\C^*$. Hence $\tTheta_i$ induces an automorphism of $\overline{\mathcal{M}'}$.

Recall the map $\varpi : \overline{\mathcal{M}}\rightarrow \mathcal{R}$ of \cite[Lemma 6.5]{FH2}\footnote{In \cite{FH2}, the generators of $\mathcal{R}$ are denoted by $y_i$, $a_{\a_i}$, etc. They correspond to our notation $[\varpi_i]$, $[\a_i]$, etc.}, which satisfies 
\[
\varpi(Y_{j,b}) = [\varpi_j],\quad \varpi(\Si_{j,b}) = (1-[-\a_j])^{-1}, \qquad  (j\in I,\ b\in\C^*).
\] 
Here $\mathcal{R} = \mathbb{Z}[[\pm \varpi_i],(1-[\alpha])^{\pm 1}]_{i\in I,\alpha\in \Delta}$ has a natural Weyl group action.
The morphism $\varpi$ can be extended to $\varpi' : \overline{\mathcal{M}}'\rightarrow \mathcal{R}$ by setting
\[
\varpi'(\Psi_{j,b}) = 1,\quad \varpi'([\la]) = [\la], \qquad \varpi'(1  - [\alpha]) = (1 - [\alpha])\qquad 
(j\in I,\ b\in\C^*,\ \la\in P,\ \alpha\in \Delta). 
\]
We still have the following.

\begin{Lem}\label{intsi} For $i\in I$, we have on the group $\overline{\mathcal{M}}'$ the relation
\begin{equation}\label{Eq-commut}
\varpi'\circ \tTheta_i = s_i\circ \varpi'.
\end{equation}
\end{Lem}

\begin{proof} This can be checked on all generators of $\overline{\mathcal{M}}'$. For instance:
\[
 \varpi'(\tTheta_i(\Psi_{i,a})) = \varpi'(1-[-\a_i])\varpi'(\widetilde{\Psi}_{i,aq^{-2}})\varpi(\Si_{i,aq^{-2}}) = 1  = s_i(\varpi'(\Psi_{i,a})).
\]
\end{proof}
\cqfd

Moreover the $\tTheta_i$ also satisfy the braid relations, as shown by the next proposition.
\begin{Prop}\label{Prop-braid-rel}
The endomorphisms $\tTheta_i$ of $\Pi'$ satisfy the braid relations.
Hence the assignment $s_i \mapsto \tTheta_i$ defines an action of the Weyl group $W$ on the ring $\Pi'$ by ring automorphisms.  
\end{Prop}

\begin{proof}
The braid relations satisfied by the $s_i$ are of the form
\[
(s_is_j)^{m_{ij}} = e, \qquad (i\not = j),
\]
where $m_{ij}=2$ (\resp $3$) if 
$c_{ij} = 0$ (\resp $-1$). Let us write
\[
 \mathcal{R}_{ij} := \left(\tTheta_i\tTheta_j\right)^{m_{ij}}.
\]
Clearly, we have $\mathcal{R}_{ij}([\la]) = [\la]$ for every $\la \in P$. It remains to show that $\mathcal{R}_{ij}(\Psi_{k,a}) = \Psi_{k,a}$ for every $k\in I$ and $a\in \C^*$. 

Now, by Proposition~\ref{Prop-restriction} and Theorem~\ref{Thm-Weyl-relations}, we know that
$\mathcal{R}_{ij}(Y_{k,a}) = Y_{k,a}$. It follows that 
\[
 \mathcal{R}_{ij}(\Psi_{k,aq})\Psi_{k,aq}^{-1} =
 \mathcal{R}_{ij}(\Psi_{k,aq^{-1}})\Psi_{k,aq^{-1}}^{-1},
 \qquad (k\in I,\ a\in\C^*).
\]

An element in $\overline{\mathcal{M}}'$ is a product $\Psi \times g$ with
$\Psi\in \mathcal{M}'$ and $g\in \overline{\mathcal{M}}$. This factorization is not unique, but
it becomes unique with the condition $\Lambda(g) = 1$ where $\Lambda:\overline{\mathcal{M}}\rightarrow
\mathcal{M}$ is defined in \cite[\S6]{FH2}. 
This implies that there is a unique factorization of the form
\[
\mathcal{R}_{ij}(\Psi_{k,aq})\Psi_{k,aq}^{-1} = M_aS_a
\]
where $M_a\in\mathcal{M}'$ and $S_a\in\overline{\mathcal{M}}$ with leading term $\Lambda(S_a) = 1$.
Because of this leading term, the above equality $M_aS_a = M_{aq^{-2}}S_{aq^{-2}}$
implies that $M_a = M_{aq^{-2}}$ does not depend on the spectral parameter $a$.
Therefore $S_a = S_{aq^{-2}}$ is also independent of $a$, and
$\mathcal{R}_{ij}(\Psi_{k,aq})\Psi_{k,aq}^{-1} = \chi \in \mathcal{R}$ is a constant.
It remains to show that $\chi = 1$. But, using Equation~(\ref{Eq-commut}) we have:
\[
\chi = \varpi'(\chi) = \varpi'\left(\mathcal{R}_{ij}(\Psi_{k,aq})\right)\varpi'\left(\Psi_{k,aq}^{-1}\right)
= (s_is_j)^{m_{ij}}\varpi'(\Psi_{k,aq}) = 1. 
\]
\cqfd
\end{proof}

\begin{remark}\label{sproof} 
{\rm
We can give another argument for the proof of the last Proposition, based on a restriction to the rank $2$ case as for the operators $\Theta_i$ in \cite{FH2}. 
The result is clear in type $A_1\times A_1$ as for $i\neq j$, $\tTheta_i\tTheta_j(\Psi_{i,a}) = \tTheta_i(\Psi_{i,a}) = \tTheta_j\tTheta_i(\Psi_{i,a})$.

In type $A_2$, we have for $i\neq j$
\begin{equation}\label{eq9}
(\tTheta_j\tTheta_i)(\Psi_{i,a}) =  (1 - [-\alpha_i-\alpha_j])(1 - [-\alpha_j]) \Psi_{j,aq^{-3}}^{-1}\Sigma_{ji,aq^{-2}},
\end{equation}
where $\Sigma_{ji,a}$ is defined by $\Theta_j(\Sigma_{i,a}) = \Sigma_{ji,a}\Sigma_{j,aq^{-1}}^{-1}$.
Clearly, we have
\[
\tTheta_i\left(\Psib_{j,aq^{-3}}^{-1}\right) = \Psib_{j,aq^{-3}}^{-1},\quad
\tTheta_i\left((1 - [-\alpha_i-\alpha_j])(1 - [-\alpha_j])\right)  
=
(1 - [-\alpha_j])(1 - [-\a_i-\alpha_j]).
\]
Since, by \cite[4.35]{FH2}, we have $\Theta_i(\Sigma_{ji,a}) = \Sigma_{ji,a}$, the right hand-side of 
Eq.~(\ref{eq9}) is invariant by $\tTheta_i$.
As $\Psi_{i,a}$ is invariant by $\tTheta_j$, it follows that 
\[
\tTheta_i\tTheta_j\tTheta_i(\Psi_{i,a}) = \tTheta_j\tTheta_i\tTheta_j(\Psi_{i,a}).
\]
\cqfd}
\end{remark}

Using Proposition~\ref{Prop-braid-rel} we can define for every $w\in W$ a ring automorphism $\tTheta_w$ of $\Pi'$ by 
\begin{equation}\label{eq-9}
 \tTheta_w := \tTheta_{i_1}\ldots \tTheta_{i_r},
\end{equation}
where $w=s_{i_1}\ldots s_{i_r}$ denotes an arbitrary factorization of $w$.

\begin{remark}\label{Rem3-12}
{\rm 
There are two possible variations of this Weyl group action, both extending the operators $\Theta_i$ on $\Pi$.

The first one is a braid group action. As in \cite[Remark 3.11]{FH3}, define an operator $\Tp_i$ on $\Pi'$ by: 
\[
\Tp_i([\la]) = [s_i(\la)],
\qquad
\Tp_i(\Psi_{j,a}) = 
\left\{ 
\begin{array}{lc}
\Psi_{j,a} & \mbox{ if }j \not = i, \\[2mm]
\widetilde{\Psi}_{i,aq^{-2}}\Si_{i,aq^{-2}} & \mbox{ if }j = i.
\end{array}
\right.
\]
Then the analogue of Proposition~\ref{Prop-restriction} is true with a similar proof, but the analogue of Lemma~\ref{Lem-38} is false (the order of $\Tp_i$ is not finite), 
as well as the analogue of Lemma~\ref{intsi}. However, the proof of Proposition \ref{Prop-braid-rel} explained in Remark \ref{sproof} works the same, and so the operators $\Tp_i$ satisfy the braid relations. For example, if $c_{ij} = -1$ we have
\[
(\Tp_j\Tp_i)(\Psi_{i,a}) =  \Psi_{j,aq^{-3}}^{-1}\Sigma_{ji,aq^{-2}}
\]
which is invariant by $\Tp_i$.

The second variation is an action of a finite covering of the Weyl group $W$, which will be discussed in \S\ref{subsec-finite-cov} below. \cqfd
}
\end{remark}

\section{$QQ$-systems} \label{sec-QQ}

Using the Weyl group action, we introduce distinguished elements $\QQ_{w(\varpi_i),\,a}$ 
of the ring $\Pi'$ satisfying a system of functional relations called the $QQ$-system.

\subsection{Normalization factors $\chi_{w(\varpi_i)}$}

We first introduce some normalization factors $\chi_{w(\varpi_i)}\ (i\in I,\ w\in W)$ indexed by the extremal weights $w(\varpi_i)$ of the fundamental $\g$-modules. The fact that these normalization factors are well-defined relies on \cite[Lemma 5.3]{FH3}.

\begin{Def}
There is a unique family $\chi_{w(\varpi_i)}$ of elements of $\Pi'$ defined by induction by the following conditions:
\begin{enumerate}
 \item[(i)] for every $i\in I$, $\chi_{\varpi_i} = 1$;
 \item[(ii)] for every $i\in I$ and $w\in W$ such that $ws_i > w$,
 \[
   \chi_{ws_i(\varpi_i)}\chi_{w(\varpi_i)} = \frac{\prod_{j\not = i} \chi_{w(\varpi_j)}^{|c_{ij}|}}{1 - [-w(\alpha_i)]}.
 \]
\end{enumerate}
\end{Def}

\begin{example}
{\rm
In type $A_2$ we have:
\[
 \chi_{\varpi_1} = 1, \quad \chi_{s_1(\varpi_1)} = \frac{1}{1-[-\alpha_1]},\quad
 \chi_{s_2s_1(\varpi_1)} = \frac{1}{(1-[-\alpha_1-\alpha_2])(1-[-\alpha_2])},
\]
\[
 \chi_{\varpi_2} = 1, \quad \chi_{s_2(\varpi_2)} = \frac{1}{1-[-\alpha_2]},\quad
 \chi_{s_1s_2(\varpi_2)} = \frac{1}{(1-[-\alpha_1-\alpha_2])(1-[-\alpha_1])}.
\]
\cqfd
}
\end{example}

\subsection{$\QQ$-variables}

Using the automorphisms $\tTheta_w$ of $\Pi'$ introduced in Eq.~(\ref{eq-9}), we can now define:
\begin{Def}
For $i\in I$, $w\in W$ and $a\in \C^*$, let
\[
 \QQ_{w(\varpi_i),\,a} := \chi_{w(\varpi_i)}\tTheta_w(\Psi_{i,a}).
\]
\end{Def}

It follows from the definition of $\tTheta_i$ and from Proposition~\ref{Prop-braid-rel} that 
$\tTheta_w(\Psi_{i,a})$ depends only on the weight $w(\varpi_i)$, hence the notation $\QQ_{w(\varpi_i),\,a}$.

\begin{example}\label{Ex4-4}
{\rm
In type $A_2$ we have:
\begin{eqnarray*}
\QQ_{\varpi_1,\,a} &=& \Psi_{1,a}, \\
\qquad \QQ_{s_1(\varpi_1),\,a} &=& \chi_{s_1(\varpi_1)}\tTheta_1(\Psi_{1,a}) =
\widetilde{\Psi}_{1,aq^{-2}}\Si_{1,aq^{-2}}
= \Psi_{1,aq^{-2}}^{-1}\Psi_{2,aq^{-1}} \Si_{1,aq^{-2}},\\
\QQ_{s_2s_1(\varpi_1),\,a} &=& \chi_{s_2s_1(\varpi_1)}\tTheta_2\tTheta_1(\Psi_{1,a}) 
= \Psi_{2,aq^{-3}}^{-1}\Si_{2,aq^{-3}}\Theta_2(\Si_{1,aq^{-2}})
= \Psi_{2,aq^{-3}}^{-1}\Si_{21,aq^{-2}},
\end{eqnarray*}
where $\Si_{21,aq^{-2}} \in \Pi$ is described in \cite[\S4.4]{FH2}. Its projection on $\widetilde{\YY}_e$ is given by 
\[
E_e(\Si_{21,aq^{-2}}) = \sum_{0\le \ell \le k} \left(\prod_{i=0}^{k-1} A_{2,aq^{-3-2i}}^{-1} \prod_{j=0}^{\ell-1} A_{1,aq^{-2-2j}}^{-1}\right) 
= 1 + A_{2,aq^{-3}}^{-1} + A_{2,aq^{-3}}^{-1}A_{1,aq^{-2}}^{-1} + \cdots.
\]
Moreover,
\[
 \QQ_{s_2(\varpi_1),\,a} = \QQ_{\varpi_1,\,a}, \quad \QQ_{s_1s_2(\varpi_1),\,a} = \QQ_{s_1(\varpi_1),\,a}, \quad
 \QQ_{s_2s_1s_2(\varpi_1),\,a} = \QQ_{s_1s_2s_1(\varpi_1),\,a} = \QQ_{s_2s_1(\varpi_1),\,a}.
\]
Finally, because of the symmetry $1 \leftrightarrow 2$ of the root system of type $A_2$, the expressions for $\QQ_{w(\varpi_2),\,a}$ are obtained by switching $1$ and $2$ in the above formulas.
\cqfd
}
\end{example}

Recall the operators $\Tp_i$ introduced in Remark~\ref{Rem3-12}. Since they satisfy the braid relations,  using an arbitrary reduced decomposition of $w\in W$ we can define a ring automorphism $\Tp_w$ of $\Pi'$. In fact we have

\begin{Lem}\label{QQTp} For $w\in W$, $i\in I$, $a\in\C^*$ we have
\[
\QQ_{w(\varpi_i),\,a} = \Tp_w(\Psi_{i,a}).
\]
\end{Lem}
 
\begin{proof}
We have 
$$\Theta_w(Y_{i,a}) 
= [w(\omega_i)]\Theta_{w}'(\Psib_{i,aq^{-1}})/\Theta_{w}'(\Psib_{i,aq}) 
= [w(\omega_i)]\QQ_{w(\varpi_i),aq^{-1}}/\QQ_{w(\varpi_i),aq}.$$
Hence, by uniqueness, as in \cite[Theorem 4.5]{FH3}, we have 
$$\Theta_{w}'(\Psib_{i,a})/ \varpi'(\Theta_{w}'(\Psib_{i,a}))
= \QQ_{w(\varpi_i),a}/ \varpi'(\QQ_{w(\varpi_i),a}).$$
By construction  
$$ \varpi'(\QQ_{w(\varpi_i),a}) = \chi_{w(\varpi)}.$$ 
In addition, it is established in the proof of \cite[Lemma 5.3]{FH3} that
$$ \varpi'(\Theta_w'(\Psib_{i,a})) = \chi_{w(\varpi_i)}.$$
This implies the result. \cqfd
\end{proof}

\subsection{$QQ$-systems}

It is proved in \cite{FH3} that the elements $\QQ_{w(\varpi_i),\,a}$ satisfy the following functional relations.

\begin{Thm}[\cite{FH3}]\label{Thm-QtQ-system}
For $i\in I$ and $w\in W$ such that $ws_i > w$, 
we have:
\[
\QQ_{ws_i(\varpi_i),\,aq} \QQ_{w(\varpi_i),\,aq^{-1}} - [-w(\alpha_i)] \QQ_{ws_i(\varpi_i),\,aq^{-1}} \QQ_{w(\varpi_i),\,aq}
= \prod_{j:\ c_{ij}=-1} \QQ_{w(\varpi_j),\,a}. 
\]
\end{Thm}

\begin{example}
{\rm
In type $A_1$, the only relation is for $i=1$ and $w=e$, and we have: 
\[
\QQ_{\varpi_1,\,a} = \Psi_{1,a},\qquad 
\QQ_{s_1(\varpi_1),\,a} = \Psi_{1,aq^{-2}}^{-1} \Si_{1,aq^{-2}}. 
\]
Then, using the relation $A_{1,aq^{-1}}^{-1} = [-\alpha_1]\Psi_{1,aq^{-3}}^{-1}\Psi_{1,aq}$, we get:
\begin{eqnarray*}
{\QQ}_{s_1(\varpi_1),\,aq} \QQ_{\varpi_1,\,aq^{-1}} - [-\alpha_i] {\QQ}_{s_1(\varpi_1),\,aq^{-1}} \QQ_{\varpi_1,\,aq}
&=& \Psi_{1,aq^{-1}}^{-1} \Si_{1,aq^{-1}} \Psi_{1,aq^{-1}} 
- [-\alpha_i]\Psi_{1,aq^{-3}}^{-1} \Si_{1,aq^{-3}} \Psi_{1,aq}\\
&=& \Si_{1,aq^{-1}} - A_{1,aq^{-1}}^{-1}\Si_{1,aq^{-3}}\\
&=& \Si_{1,aq^{-1}} - (\Si_{1,aq^{-1}} -1) \\
&=& 1.
\end{eqnarray*}
\cqfd
}
\end{example}

\section{Cluster algebra structures on formal power series rings} \label{sect-cluster-FPS}

\subsection{The ring $K_\Z$} \label{subsection-KZ}

Let $K:= E'_e(\Pi')$ denote the projection of $\Pi'$ on its first factor. In other words,
\[
 K = \widetilde{\YY_e}'.
\]
The ring $K$ contains in particular the elements:
\[
Q_{w(\varpi_i),\,a} := E'_e\left(\QQ_{w(\varpi_i),\,a}\right),\qquad (i\in I,\ w\in W,\ a\in \C^*).
\]
Each of these elements has a unique factorization of the form
\[
Q_{w(\varpi_i),\,a} = \Psi_{w(\varpi_i),\,a}\, \Si_{w(\varpi_i),\,a}, 
\]
where $\Psi_{w(\varpi_i),\,a}$ is a Laurent monomial in the variables $\Psi_{j,b}$, and $\Si_{w(\varpi_i),\,a}$ is a formal power series in the variables $A_{j,b}^{-1}$ with constant term equal to 1. We call $\Psi_{w(\varpi_i),\,a}$ the highest weight monomial of $Q_{w(\varpi_i),\,a}$.

It is explained in \cite{FH2} that $\widetilde{\YY_e}$ is a complete topological ring. For the same reason, $K$ has the structure of a complete topological ring. 
Moreover, $K$ is an algebra over the group ring $[P]$ generated by the elements $[\lambda]\ (\lambda \in P)$.

In the sequel, we will only consider elements $Q_{w(\varpi_i),\,a}$ in which the spectral parameter $a$ belongs to a discrete subset of $\C^*$ (depending on $i\in I$). Also, for reasons which will appear in \S\ref{ssec-renorm}, we will have to use weights in $\uP := \frac{1}{2}\Z\otimes_\Z P$. More precisely,   
recall the vertex set $V \subset I \times \Z$ of the basic quiver $\Gamma_e$. We define $K_\Z$ as the 
 $[\uP]$-subalgebra of $[\uP]\otimes_{[P]} K$ topologically generated by the elements:
\[
Q_{w(\varpi_i),\,q^r},\quad ((i,r)\in V,\ w\in W). 
\]

\begin{Prop} \label{prop-invertible}
For every $(i,r)\in V$ and $w\in W$, the element $Q_{w(\varpi_i),\,q^r}$ 
is invertible in $K_{\Z}$.
\end{Prop}

\begin{proof} Let $K_{\Z}'$ be the topological $[\uP]$-subalgebra of $[\uP]\otimes_{[P]} K$ generated by the $\Psi_{i,q^r}^{\pm 1}$ with $(i,r)\in V$.
Let us prove that $K_{\Z}$ equals $K_{\Z}'$.

First, we establish that the $Q_{w(\varpi_i),\,q^r}$ are in $K_{\Z}'$ if $(i,r)\in V$ and $w\in W$.
The point here is the restriction on spectral parameters. We need to show that all terms of $Q_{w(\varpi_i),\,q^r}$
depend only on the $\Psib_{i,q^r}$ with $(i,r)\in V$.
Recall the factorization
\[
Q_{w(\varpi_i),\,q^r} = \Psi_{w(\varpi_i),\,q^r}\, \Si_{w(\varpi_i),\,q^r}.
\]
By \cite{FH3}, the Laurent monomial $\Psi_{w(\varpi_i),\,q^r}$ can be calculated from $\Psi_{\varpi_i,q^r}$ using Chari's braid group action, hence it is a Laurent monomial in the
$\Psi_{j,q^s}^{\pm 1}$ with $(j,s)\in V$. 
Therefore it suffices to prove that the factor $\Si_{w(\varpi_i),q^r}$
belongs to $K_{\Z}'$.
For $w = e$ or $w = s_j$ with $j\neq i$, $\Si_{w(\varpi_i),q^r} = 1$
and for $w = s_i$,
\[
\Si_{w(\varpi_i),q^r} = E'_e(\Sigma_{i,q^{r-2}}) = 1 + A_{i,q^{r-2}}^{-1}(1 + A_{i,q^{r-4}}^{-1}(\cdots))\in K_{\Z}'.
\]
Consider the two following properties depending on $l = 0, 1, \ldots , \ell(w_0)$:
\begin{itemize}
 \item[($A_l$)]   For any $w\in W$ of length $\leq l$ and any $(i,r)\in V$, 
 we have $E'_e(\tT_w(\Psi_{i,r}^{\pm 1}))\in K_{\Z}'$.
 \item[($B_l$)]   For any $w\in W$ of length $\leq l$ and any $(i,r)\in V$, 
 we have $E'_e(\Theta_w(\Sigma_{i,r}^{\pm 1}))\in K_{\Z}'$.
\end{itemize}
We have already checked $(A_1)$ and $(B_0)$. 
Now the defining relations of the actions $\Theta$ and $\tT$ give that $(B_l, A_l)\Rightarrow (A_{l+1})$.
Also, it follows from the $q$-difference equation
\[
\Theta_w(\Sigma_{i,a}) = 1 + \Theta_w(A_{i,a}^{-1})\Theta_w(\Sigma_{i,aq^{-2}})
\]
that $(A_l)\Rightarrow (B_l)$. Thus we obtain by induction that $A_l$ and $B_l$ hold for any $l\le \ell(w_0)$, and we have established that $K_{\Z}\subset K_{\Z}'$.

Conversely, we can show using the braid group action that, for every $(i,r)$ in $V$,  $\Psi_{w_0(\varpi_i),\,q^r} = \Psi_{\nu(i),\,q^{r - h}}^{-1}$, where $\nu$ is the Nakayama involution defined in Proposition~\ref{prop:slices}, and $h$ is the Coxeter number. 
It follows that for any Laurent monomial $M$ in
the variables $\Psi_{i,q^r}$, $(i,r)\in V$, there exists a product of elements of $K_{\Z}$ of the form 
$Q_{\varpi_j,\,q^s}$ $((j,s)\in V)$ and $Q_{w_0(\varpi_k),\,q^t}$ $((k,t)\in V)$
 whose leading term is $M$. Hence the generators of $K_{\Z}$ generate also $K_{\Z}'$ as a topological $[P]$-module. Thus we have proved that $K_{\Z} = K_{\Z}'$.

Now, recall that $\Si_{w(\varpi_i),\,q^r}$ is invertible, and $\Si_{w(\varpi_i),\,q^r}^{-1}$ belongs to $K_{\Z}'$. Clearly, we also have $\Psi_{w(\varpi_i),\,q^r}^{-1}\in K_{\Z}'$. It follows that 
$Q_{w(\varpi_i),\,q^r}$ is invertible in $K_{\Z}' = K_\Z$ for every $(i,r)\in V$. \cqfd
\end{proof}

\subsection{Renormalized $Q$-variables}\label{ssec-renorm}
\quad
Let $\Omega$ be the group homomorphism from the multiplicative group of Laurent monomials in the variables 
\[ 
\Psi_{i,q^r},\ [\la],\qquad ((i,r)\in V,\ \la\in P), 
\]
to the additive group $\uP$ defined by 
\[
 \Omega(\Psi_{i,q^r}) := \frac{r}{2}\, \varpi_i,\quad \Omega([\la]) := \la, \qquad ((i,r)\in V,\ \la\in P).
\]
Note that 
\begin{equation}\label{Eq-Omega-1}
\Omega(Y_{i,q^{r+1}}) = \Omega\left([\varpi_i]\frac{\Psi_{i,q^{r}}}{\Psi_{i,q^{r+2}}}\right) = 0, \qquad ((i,r)\in V). 
\end{equation}
We set: 
\[
 \bQ_{w(\varpi_i),\,q^r} := \left[-\Omega(\Psi_{w(\varpi_i),\,q^r})\right]\,Q_{w(\varpi_i),\,q^r},
 \qquad ((i,r)\in V,\ w\in W). 
\]
In terms of these renormalized elements $\bQ_{w(\varpi_i),\,q^r}$, the $QQ$-system simplifies slightly and becomes:
\begin{Prop}\label{Prop-renorm-QQ}
For $(i,r)\in V$ and $w\in W$ such that $ws_i > w$, 
we have:
\[
{\bQ}_{ws_i(\varpi_i),\,q^{r}} \bQ_{w(\varpi_i),\,q^{r-2}} -\ {\bQ}_{ws_i(\varpi_i),\,q^{r-2}} \bQ_{w(\varpi_i),\,q^{r}}
=\prod_{j:\ c_{ij}=-1} \bQ_{w(\varpi_j),\,q^{r-1}}. 
\]
\end{Prop}

\begin{proof}
We have to check that, using this new normalization, the factor $[-w(\alpha_i)]$ in the second term of the left hand-side of Theorem~\ref{Thm-QtQ-system} cancels out.

Let us first check it for $w = e$. The highest weight monomial of ${\bQ}_{s_i(\varpi_i),\,q^{r}} \bQ_{\varpi_i,\,q^{r-2}}$ is equal to 
\[
\widetilde{\Psi}_{i,q^{r-2}}\Psi_{i,q^{r-2}} =
\prod_{j:\ c_{ij}=-1} \Psi_{\varpi_j,\,q^{r-1}} 
\]
which is nothing else than the highest weight-monomial of the right-hand side. Therefore  
applying $\Omega$ to the highest weight-monomial of the right-hand side, we get exactly 
$\Omega(\widetilde{\Psi}_{i,q^{r-2}}\Psi_{i,q^{r-2}})$.
Now the quotient of the highest weight monomial of ${\bQ}_{s_i(\varpi_i),\,q^{r-2}} \bQ_{\varpi_i,\,q^{r}}$ by the highest weight monomial of ${\bQ}_{s_i(\varpi_i),\,q^{r}} \bQ_{\varpi_i,\,q^{r-2}}$ is equal to
\[
\frac{\widetilde{\Psi}_{i,q^{r-4}}\Psi_{i,q^{r}}}{\widetilde{\Psi}_{i,q^{r-2}}\Psi_{i,q^{r-2}}} 
=
[\alpha_i-\varpi_i]Y_{i,q^{r-1}} A_{i,q^{r-2}}^{-1} [\varpi_i]Y_{i,q^{r-1}}^{-1}
=
[\alpha_i] A_{i,q^{r-2}}^{-1},
\]
by Lemma~\ref{Lem31}. Hence, since by Eq.~(\ref{Eq-Omega-1}) we have $\left[\Omega\left(A_{i,q^{r-2}}^{-1}\right)\right] = 1$, we get:  
\[
 \left[-\Omega\left(\widetilde{\Psi}_{i,q^{r-4}}\Psi_{i,q^{r}}\right)\right] =
   \left[-\Omega\left(\widetilde{\Psi}_{i,q^{r-2}}\Psi_{i,q^{r-2}}\right)\right]
   [-\alpha_i] 
\]
and the Proposition is proved in the case $w=e$.

The case of an arbitrary $w$ can be checked in a similar way using results from \cite{FH3}. First, one  can see that for every $w$, the highest weight monomial of the right-hand side coincides with the highest weight monomial of the first term of the left-hand side. Therefore we only have to compare the result of applying $\Omega$ to the highest weight monomials of the two terms of the left hand-side. Now, the highest weight monomial $\Lambda(\Theta_w(Y_{i,a}))$ of $\Theta_w(Y_{i,a})$, where $\Lambda$ is defined in \cite[§6.4]{FH2}, is known to be equal to
\[
\Lambda(\Theta_w(Y_{i,a})) = [w(\varpi_i)] \Psi_{w(\varpi_i),aq^{-1}}\Psi_{w(\varpi_i),aq}^{-1},
\quad (w\in W,\ a\in\C^*).
\]
Since $\Lambda(\Theta_w(Y_{i,a}))$ is a Laurent monomial in the variables $Y_{j,b}\ (j\in I,\ b\in\C^*)$, it follows from Eq.~(\ref{Eq-Omega-1}), that:
\[
\left[\Omega\left(\Lambda(\Theta_w(Y_{i,q^{r-1}}))\right)\right] = [w(\varpi_i)]\, \left[\Omega\left(\Psi_{w(\varpi_i),q^{r-2}}\Psi_{w(\varpi_i),q^r}^{-1}\right)\right] = 1,
\]
and similarly
\[
\left[\Omega\left(\Lambda(\Theta_{ws_i}(Y_{i,q^{r-3}}))\right)\right] = [ws_i(\varpi_i)]\, \left[\Omega\left(\Psi_{ws_i(\varpi_i),q^{r-4}}\Psi_{ws_i(\varpi_i),q^{r-2}}^{-1}\right)\right] = 1.
\]
Therefore,
\[
\left[\Omega\left(\frac{\Psi_{ws_i(\varpi_i),q^{r-4}}\Psi_{w(\varpi_i),q^{r}}}{\Psi_{ws_i(\varpi_i),q^{r-2}}\Psi_{w(\varpi_i),q^{r-2}}} 
\right)\right] 
= [w(\varpi_i) - ws_i(\varpi_i)] = [w(\alpha_i)], 
\]
hence
\[
 \left[-\Omega\left(\Psi_{ws_i(\varpi_i),q^{r-4}}\Psi_{w(\varpi_i),q^{r}}\right)\right] =
   \left[-\Omega\left(\Psi_{ws_i(\varpi_i),q^{r-2}}\Psi_{w(\varpi_i),q^{r-2}}\right)\right]
   [-w(\alpha_i)]. 
\]
\cqfd
\end{proof}

\subsection{Action of a finite covering of $W$}
\label{subsec-finite-cov}

Consider the subring $\Pi'_\Z\subset  [\uP]\otimes_{[P]} \Pi'$ constructed as $\Pi'$ but using only variables $\Psi_{i,q^r}$ with $(i,r)\in V$ and $[\la]\ (\lambda \in \uP)$ (the completion procedure is the same as in \S\ref{sect3-2}). 

We define an operator $\overline{s}_i$ on $\Pi'_\Z$ by $\overline{s}_i([\la]) := [s_i(\la)] \ (\la \in \uP)$
and : 
\[
\overline{s}_i([-\Omega(\Psi_{j,q^r})]\,\Psi_{j,q^r}) = 
\left\{ 
\begin{array}{lc}
[-\Omega(\Psi_{j,q^r})]\,\Psi_{j,q^r} & \mbox{ if }j \not = i, \\[2mm]
\left[-\Omega(\widetilde{\Psi}_{i,q^{r-2}})\right]\,\widetilde{\Psi}_{i,q^{r-2}}\Si_{i,q^{r-2}} & \mbox{ if }j = i.
\end{array}
\right.
\]
Then the analogue of Proposition \ref{Prop-restriction} is true with a similar proof, because 
\[
\overline{s}_i(Y_{i,q^r}) = \bar{s}_i([-\Omega(\Psi_{i,q^{r-1}})]\Psi_{i,q^{r-1}} [\Omega(\Psi_{i,q^{r+1}})]\Psi_{i,q^{r+1}}^{-1})
=\Theta_i(Y_{i,q^r})
\]
and $\overline{s}_i(Y_{j,q^r}) = Y_{j,q^r}$ for $i\neq j$.

Although $\overline{s}_i$ is not an involution, it is an involution up to sign (and so it is an automorphism). 
Indeed, by a calculation similar to the proof of Lemma~\ref{Lem-38}, we have:
\[
\begin{array}{l}
\overline{s}_i^2([-\Omega(\Psi_{i,q^r})]\Psi_{i,q^r})
\\[2mm]
\qquad = \ \bar{s}_i\left([\Omega(\Psi_{i,q^{r-2}})]\Psi_{i,q^{r-2}}^{-1}\right)
\left(\widetilde{\Psi}_{i,q^{r-2}}\Psi_{i,q^{r-2}}\right)
\left[-\Omega(\widetilde{\Psi}_{i,q^{r-2}}\Psi_{i,q^{r-2}})\right] (-A_{i,q^{r - 2}}^{-1})\Sigma_{i,q^{r - 4}}
\\[2mm]
\qquad= \ -[-\Omega(\Psi_{i,q^r})]\Psi_{i,q^r},
\end{array}
\]
and $\overline{s}_i^2(\Psi_{j,q^r}) = \Psi_{j,q^r}$ for $i\neq j$. 
Hence we have: 
\[
\overline{s}_i^4 = \id,\qquad (i\in I).
\]

Finally, the proof of Proposition \ref{Prop-braid-rel} explained in Remark \ref{sproof} works also in this case, 
and we obtain that the $\overline{s}_i$ satisfy the braid relations. 
For example, in type $A_2$ with $i\not = j$,
$$(\overline{s}_j\overline{s}_i)(\Psi_{i,q^{r}}[-\Omega(\Psi_{i,q^{r}})]) = [\Omega(\Psi_{j,q^{r-3}})] \Psi_{j,q^{r-3}}^{-1}\Sigma_{ji,q^{r-2}},$$
which is invariant by $\overline{s}_i$.

\begin{remark}\label{rem-7-3}
{\rm
Since the $\overline{s}_i$ satisfy the braid relations,  using an arbitrary reduced decomposition of $w\in W$ we can define a ring automorphism $\overline{s}_w$ of $\Pi'_\Z$. Let 
$\bQQ_{w(\varpi_i),q^r} := \left[-\Omega(\Psi_{w(\varpi_i),\,q^r})\right]\QQ_{w(\varpi_i),q^r}$.
Then we have 
$$\overline{s}_w(\bQQ_{\varpi_i,q^r}) = \bQQ_{w(\varpi_i),\,q^r}.$$
The proof is analog to the proof of Lemma \ref{QQTp}. We have
$$\overline{s}_w(Y_{i,q^{r-1}}) = \overline{s}_w(\bQQ_{\varpi_i,q^{r-2}})/\overline{s}_w(\bQQ_{\varpi_i,q^r})
= \bQQ_{w(\varpi_i),q^{r-2}}/\bQQ_{w(\varpi_i),q^r}.$$
Hence, by uniqueness as in \cite[Theorem 4.5]{FH3}, we have that $\overline{s}_w(\bQQ_{\varpi_i,q^r})$ 
and $\bQQ_{w(\varpi_i),q^r}$ are equal up to their image by $\varpi'$. 
(Here we keep denoting by $\varpi'$ the map extended to $\Pi'_\Z$ by setting $\varpi'([\lambda]) = [\lambda]$ for $\lambda\in\uP$.)
We know that  
$$\left[\Omega(\Psi_{w(\varpi_i),\,q^r})\right] \varpi'(\bQQ_{w(\varpi_i),q^r}) = \chi_{w(\varpi_i)}$$
which is the solution of $QQ$-system without spectral parameters. 
As above, one can check that 
$\left[\Omega(\Psi_{w(\varpi_i),\,q^r})\right] \varpi'(\overline{s}_w(\bQQ_{\varpi_i,q^r}))$ 
satisfies the same equation, and conclude in the same way.

}
\end{remark}

\subsection{Cluster algebra structure on $K_\Z$}\label{ssubsec-main-th}

In this section we will use freely the notation and results of \S\ref{subsec4-2} and \S\ref{subsec4-3}.
In particular, we fix a Coxeter element $c$ and we denote by $\Gamma$ the quiver of the corresponding initial seed of the cluster algebra $\AA_{w_0}$. The cluster variables $z_{(i,a)} \ ((i,a) \in V)$ of this initial seed are parametrized by their stabilized $g$-vectors, which by \S\ref{subsec4-3}, Equation~(\ref{form-theta-g-vect}), can be written in the form 
\[
 \bg^{(\infty)}_{(i,a)} = \theta_{i_1}\cdots\theta_{i_{t}}(\be_{(i,m_{i})})[s],
\]
for some well-defined $s\in \Z$ and $t\in\{0,1,\ldots, N\}$ such that $i_t = i$. 
Let $L$ denote the Laurent polynomial ring in these initial cluster variables. By the Laurent phenomenon,
the cluster algebra $\AA_{w_0}$ embeds into $L$.

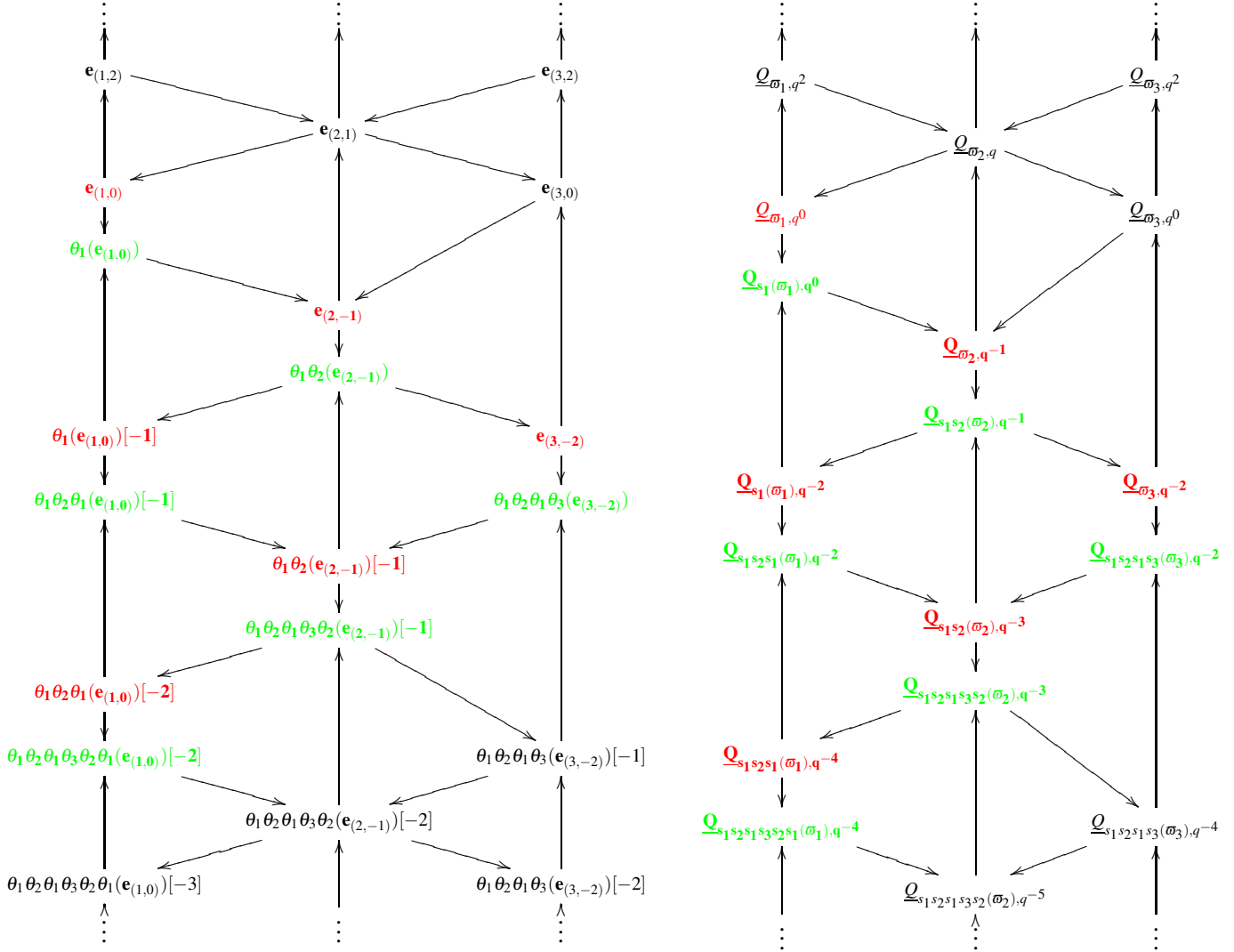
\begin{figure}[t]
\[
\def\objectstyle{\scriptstyle}
\def\lablestyle{\scriptstyle}
\xymatrix@-1.0pc{
{}\save[]+<0cm,1ex>*{\vdots}\restore&{}\save[]+<0cm,1ex>*{\vdots}\restore  
&{}\save[]+<0cm,1ex>*{\vdots}\restore
\\
{\be_{(1,2)}}\ar[rd]\ar[u]&
&\ar[ld] \be_{(3,2)} \ar[u]
\\
&\ar[ld] \be_{(2,1)} \ar[rd]\ar[uu]&&
\\
{\ar[uu]}{\red\be_{(1,0)}}\ar[d]&&\ar[ldd] \be_{(3,0)} \ar[uu]
\\
{\green\mathbf{\theta_1(\be_{(1,0)})}}\ar[rd]
\\
&\ar[uuu] \mathbf{\red \be_{(2,-1)}}\ar[d]&&
\\
&\ar[ld] \mathbf{\green \theta_1\theta_2(\be_{(2,-1)})} \ar[rd]&&
\\
\ar[uuu]\mathbf{\red \theta_1(\be_{(1,0)})[-1]} \ar[d] && \ar[d] \mathbf{\red e_{(3,-2)}}\ar[uuuu]
\\
\mathbf{\green \theta_1\theta_2\theta_1(\be_{(1,0)})[-1]} \ar[rd] &&\ar[ld] 
\mathbf{\green \theta_1\theta_2\theta_1\theta_3(\be_{(3,-2)})} 
\\
& \ar[uuu]\mathbf{\red \theta_1\theta_2(\be_{(2,-1)})[-1]} \ar[d] &&
\\
&\ar[ld]\mathbf{\green \theta_1\theta_2\theta_1\theta_3\theta_2(\be_{(2,-1)})[-1]} \ar[rdd]&&
\\
\ar[uuu]\mathbf{\red \theta_1\theta_2\theta_1(\be_{(1,0)})[-2]} \ar[d]  
\\
\mathbf{\green \theta_1\theta_2\theta_1\theta_3\theta_2\theta_1(\be_{(1,0)})[-2]} \ar[rd]&&
\ar[ld] {\theta_1\theta_2\theta_1\theta_3(\be_{(3,-2)})[-1]}\ar[uuuu]
\\
& \ar[uuu]{\theta_1\theta_2\theta_1\theta_3\theta_2(\be_{(2,-1)})[-2]} \ar[ld]\ar[rd]&&
\\
\ar[uu] {\theta_1\theta_2\theta_1\theta_3\theta_2\theta_1(\be_{(1,0)})[-3]} &&
\ar[uu] {\theta_1\theta_2\theta_1\theta_3(\be_{(3,-2)})[-2]} 
\\
{}\save[]+<0cm,1ex>*{\vdots}\ar[u]\restore&{}\save[]+<0cm,1ex>*{\vdots}\ar[uu]\restore  
&{}\save[]+<0cm,1ex>*{\vdots}\ar[u]\restore
\\
}
\xymatrix@-1.0pc{
{}\save[]+<0cm,1ex>*{\vdots}\restore&{}\save[]+<0cm,1ex>*{\vdots}\restore  
&{}\save[]+<0cm,1ex>*{\vdots}\restore
\\
{\bQ_{\varpi_1,q^2}}\ar[rd]\ar[u]&
&\ar[ld] \bQ_{\varpi_3,q^2} \ar[u]
\\
&\ar[ld] \bQ_{\varpi_2,q} \ar[rd]\ar[uu]&&
\\
{\ar[uu]}{\red\bQ_{\varpi_1,q^0}}\ar[d]&&\ar[ldd] \bQ_{\varpi_3,q^0} \ar[uu]
\\
{\green\mathbf{\bQ_{s_1(\varpi_1),q^0}}}\ar[rd]
\\
&\ar[uuu] \mathbf{\red \bQ_{\varpi_2,q^{-1}}}\ar[d]&&
\\
&\ar[ld] \mathbf{\green \bQ_{s_1s_2(\varpi_2),q^{-1}}} \ar[rd]&&
\\
\ar[uuu]\mathbf{\red \bQ_{s_1(\varpi_1),q^{-2}}} \ar[d] && \ar[d] \mathbf{\red \bQ_{\varpi_3,q^{-2}}}\ar[uuuu]
\\
\mathbf{\green \bQ_{s_1s_2s_1(\varpi_1),q^{-2}}} \ar[rd] &&\ar[ld] 
\mathbf{\green \bQ_{s_1s_2s_1s_3(\varpi_3),q^{-2}}} 
\\
& \ar[uuu]\mathbf{\red \bQ_{s_1s_2(\varpi_2),q^{-3}}} \ar[d] &&
\\
&\ar[ld]\mathbf{\green \bQ_{s_1s_2s_1s_3s_2(\varpi_2),q^{-3}}} \ar[rdd]&&
\\
\ar[uuu]\mathbf{\red \bQ_{s_1s_2s_1(\varpi_1),q^{-4}}} \ar[d]  
\\
\mathbf{\green \bQ_{s_1s_2s_1s_3s_2s_1(\varpi_1),q^{-4}}} \ar[rd]&&
\ar[ld] {\bQ_{s_1s_2s_1s_3(\varpi_3),q^{-4}}}\ar[uuuu]
\\
& \ar[uuu]{\bQ_{s_1s_2s_1s_3s_2(\varpi_2),q^{-5}}} 
\\
{}\save[]+<0cm,1ex>*{\vdots}\ar[uu]\restore&{}\save[]+<0cm,1ex>*{\vdots}\ar[u]\restore  
&{}\save[]+<0cm,1ex>*{\vdots}\ar[uu]\restore
\\
}
\]
\caption{\label{Fig9} {\it An initial seed of $\AA_{w_0}$ and its image by $F$ in type $A_3$.}}
\end{figure}

\begin{Thm}\label{main-Thm} 
For $(i,a)\in V$, let $z_{(i,a)}$ denote the initial cluster variable with stabilized $g$-vector 
\[
 \bg^{(\infty)}_{(i,a)} = \theta_{i_1}\cdots\theta_{i_{t}}(\be_{(i,m_{i})})[s].
\]
The assignment 
\[
z_{(i,a)} \mapsto 
\bQ_{s_{i_1}\cdots s_{i_t}(\varpi_i),\,q^{ m_i+2s}},\qquad ((i,a)\in V),
\]
extends to an injective ring homomorphism $F : L \to K_{\Z}$. 
Moreover the topological closure of $ [\uP]\otimes_\Z F(\AA_{w_0})$ is equal to $K_{\Z}$.
\end{Thm}

\begin{example}\label{exam-5.3}
{\rm
Figure~\ref{Fig9} illustrates Theorem~\ref{main-Thm} in type $A_3$. 
The left quiver is the quiver of an initial seed of $\AA_{w_0}$ associated with the Coxeter 
element $c=s_1s_2s_3$. At each vertex we have written the stabilized $g$-vector of the corresponding cluster variable. The right quiver shows the $Q$-variables images under $F$ of these cluster variables.

Mutating the initial cluster variable $\bQ_{\varpi_1,q^0}$ we get the new cluster variable $\bQ_{s_1(\varpi_1),q^2}$, and the corresponding exchange relation is
\[
\bQ_{\varpi_1,q^0} \bQ_{s_1(\varpi_1),q^2} = \bQ_{\varpi_1,q^2} \bQ_{s_1(\varpi_1),q^0}\, 
+ \, \bQ_{\varpi_2,q}, 
\]
an example of relation of the $QQ$-system.

Similarly, mutating $\bQ_{s_1s_2(\varpi_2),q^{-3}}$ we get the new  variable $\bQ_{s_1s_2s_1s_3s_2(\varpi_2),q^{-1}}$, and the corresponding exchange relation is
\[
\bQ_{s_1s_2(\varpi_2),q^{-3}} \bQ_{s_1s_2s_3s_1s_2(\varpi_2),q^{-1}} = 
\bQ_{s_1s_2(\varpi_2),q^{-1}} \bQ_{s_1s_2s_3s_1s_2(\varpi_2),q^{-3}}
+ \,\bQ_{s_1s_2s_1(\varpi_1),q^{-2}} \bQ_{s_1s_2s_1s_3(\varpi_3),q^{-2}}, 
\]
which is also an example of relation of the $QQ$-system (for $i=2$, $w=s_1s_2s_3s_1$, 
and $r=-1$).

On the other hand, mutating $\bQ_{\varpi_2,q}$, we get a new cluster variable, say $x$, given by the exchange relation
\[
x\, \bQ_{\varpi_2,q} = \bQ_{\varpi_1,q^{0}} \bQ_{\varpi_2,q^3}\bQ_{\varpi_3,q^0} +\, 
\bQ_{\varpi_1,q^{2}} \bQ_{\varpi_2,q^{-1}}\bQ_{\varpi_3,q^2}
\]
which is \emph{not} a relation of the $QQ$-system. Explicitly,
\[
 x = \left[-\varpi_1 + \varpi_2-\varpi_3\right]
 \Psi_{2,q}^{-1}\Psi_{2,q^{-1}}\Psi_{1,q^2}\Psi_{3,q^2}\left(1 + A_{2,q}^{-1}\right).
\]
This other type of relation has been studied in \cite{HL2} and \cite{FH}, where it was given the name of $QQ^*$-system.
\cqfd
}
\end{example}

\begin{figure}[t]
\[
\def\objectstyle{\scriptscriptstyle}
\def\lablestyle{\scriptscriptstyle}
\xymatrix@-1.0pc{
{}\save[]+<0cm,1ex>*{\vdots}\restore&{}\save[]+<0cm,1ex>*{\vdots}\restore  
\\
{\bQ_{\varpi_1,q^2}}\ar[rd]\ar[u]&
\\
&\ar[ld] \bQ_{\varpi_2,q} \ar[uu]
\\
\ar[uu]{\mathbf{\red \bQ_{\varpi_1,q^0}}}\ar[d]&
\\
\fbox{$\scriptstyle\mathbf{\green \bQ_{s_1(\varpi_1),q^0}}$}\ar[rd]
\\
&\ar[uuu] \mathbf{\red\bQ_{s_1(\varpi_2),q^{-1}}}\ar[d] 
\\
& \ar[ld]\mathbf{\green\bQ_{s_1s_2(\varpi_2),q^{-1}}}
\\
\ar[uuu]\mathbf{\red\bQ_{s_1s_2(\varpi_1),q^{-2}}} \ar[d] 
\\
{\mathbf{\green\bQ_{s_1s_2s_1(\varpi_1),q^{-2}}}}\ar[rd]
\\
& \ar[uuu]\bQ_{s_1s_2s_1(\varpi_2),q^{-3}}\ar[ld] 
\\
\ar[uu]\bQ_{s_1s_2s_1(\varpi_1),q^{-4}} \ar[rd] 
\\
{}\save[]+<0cm,0ex>*{\vdots}\ar[u]\restore&{}\save[]+<0cm,0ex>*{\vdots}\ar[uu]\restore  
\\
}
\xymatrix@-1.0pc{
{}\save[]+<0cm,1.5ex>*{\vdots}\restore&{}\save[]+<0cm,1.5ex>*{\vdots}\restore  
\\
{\bQ_{\varpi_1,q^2}}\ar[rd]\ar[u]&
\\
&\ar[ld] \bQ_{\varpi_2,q} \ar[uu]
\\
\ar[uu]{\mathbf{\bQ_{\varpi_1,q^0}}}\ar[rdd]&
\\
\mathbf{\bQ_{\varpi_1,q^{-2}}}\ar[u]\ar[ddd]
\\
&\ar[uuu]\ar[lu] \mathbf{\bQ_{s_1(\varpi_2),q^{-1}}}\ar[d] 
\\
& \ar[ld]\mathbf{\bQ_{s_1s_2(\varpi_2),q^{-1}}}
\\
\fbox{$\scriptstyle\mathbf{\bQ_{s_1s_2(\varpi_1),q^{-2}}}$} \ar[d]\ar[ruu] 
\\
{\mathbf{\bQ_{s_1s_2s_1(\varpi_1),q^{-2}}}}\ar[rd]
\\
& \ar[uuu]\bQ_{s_1s_2s_1(\varpi_2),q^{-3}}\ar[ld] 
\\
\ar[uu]\bQ_{s_1s_2s_1(\varpi_1),q^{-4}}  \ar[rd] 
\\
{}\save[]+<0cm,0ex>*{\vdots}\ar[u]\restore&{}\save[]+<0cm,0ex>*{\vdots}\ar[uu]\restore  
\\
}
\xymatrix@-1.0pc{
&{}\save[]+<0cm,1.5ex>*{\vdots}\restore&{}\save[]+<0cm,1.5ex>*{\vdots}\restore  
\\
&{\bQ_{\varpi_1,q^2}}\ar[rd]\ar[u]&
\\
&&\ar[ld] \bQ_{\varpi_2,q} \ar[uu]
\\
&\ar[uu]{\mathbf{\bQ_{\varpi_1,q^0}}}\ar[rdd]&
\\
&\mathbf{\bQ_{\varpi_1,q^{-2}}}\ar[u]\ar@/_{2pc}/[dddd]
\\
&&\ar[uuu]\ar[ldd] \mathbf{\bQ_{s_1(\varpi_2),q^{-1}}} 
\\
&& \fbox{$\scriptstyle\mathbf{\bQ_{s_1s_2(\varpi_2),q^{-1}}}$}\ar[ldd]
\\
&\mathbf{\bQ_{s_2(\varpi_2),q^{-1}}} \ar[uuu]\ar[ru] 
\\
&{\mathbf{\bQ_{s_1s_2s_1(\varpi_1),q^{-2}}}}\ar[rd]\ar[u]
\\
&& \ar[uuu]\bQ_{s_1s_2s_1(\varpi_2),q^{-3}}\ar[ld] 
\\
&\ar[uu]\bQ_{s_1s_2s_1(\varpi_1),q^{-4}} \ar[rd] 
\\
&{}\save[]+<0cm,0ex>*{\vdots}\ar[u]\restore&{}\save[]+<0cm,0ex>*{\vdots}\ar[uu]\restore  
\\
}
\xymatrix@-1.0pc{
{}\save[]+<0cm,1.5ex>*{\vdots}\restore&{}\save[]+<0cm,1.5ex>*{\vdots}\restore  
\\
{\bQ_{\varpi_1,q^2}}\ar[rd]\ar[u]&
\\
&\ar[ld] \bQ_{\varpi_2,q} \ar[uu]
\\
\ar[uu]\bQ_{\varpi_1,q^0}\ar[rd]&
\\
&\ar[uu] \mathbf{\red\bQ_{\varpi_2,q^{-1}}}\ar[d] 
\\
& \ar[ld]\mathbf{\green\bQ_{s_2(\varpi_2),q^{-1}}} 
\\
 \mathbf{\red\bQ_{s_2(\varpi_1),q^{-2}}} \ar[uuu]\ar[d]
\\
\mathbf{\green\bQ_{s_2s_1(\varpi_1),q^{-2}}} \ar[dr]
\\
& \ar[uuu]\ar[d]\mathbf{\red\bQ_{s_2s_1(\varpi_2),q^{-3}}}
\\
& \mathbf{\green\bQ_{s_2s_1s_2(\varpi_2),q^{-3}}}\ar[ld] 
\\
\ar[uuu]\bQ_{s_2s_1s_2(\varpi_1),q^{-4}} \ar[rd] 
\\
{}\save[]+<0cm,0ex>*{\vdots}\ar[u]\restore&{}\save[]+<0cm,0ex>*{\vdots}\ar[uu]\restore  
\\
}
\]
\caption{\label{Fig10} {\it A sequence of mutations relating two seeds of $K_\Z$ in type $A_2$.}}
\end{figure}
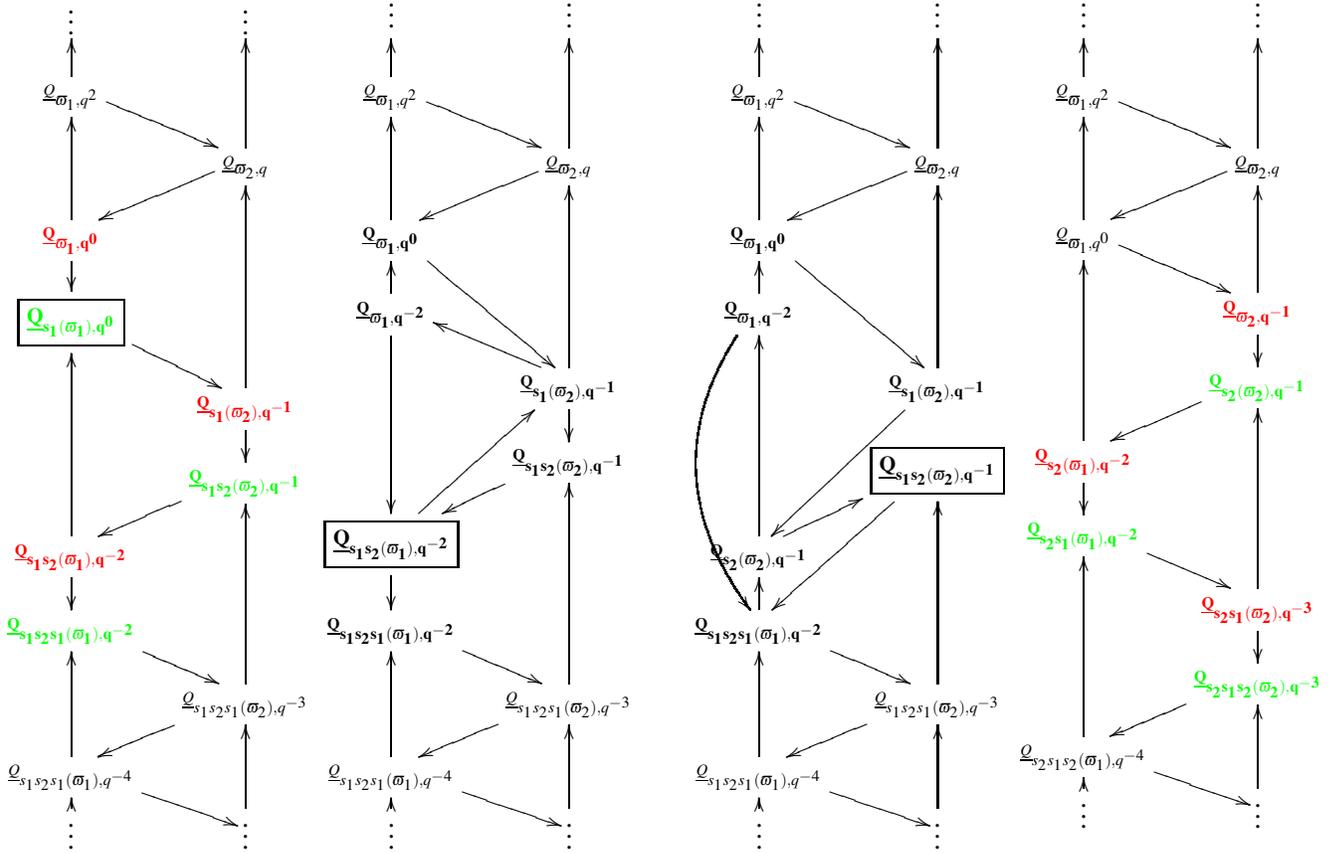

\begin{remark}\label{Rem54}
{\rm
The homomorphism $F$ of Theorem~\ref{main-Thm} is defined by means of an initial seed of the cluster algebra $\AA_{w_0}$, corresponding to the datum of a Coxeter element $c$. 
If we use a different Coxeter element, we will define a priori another homomorphism $F'$, giving rise to another cluster structure on $K_\Z$, which is possibly different but is isomorphic. 

In fact, we could start from a more general initial seed defined by the datum of a reduced decomposition $\bi = (i_1,\ldots, i_N)$ of $w_0$ and a sequence of integers $\br = (r_1\ldots, r_n)$, as in Definition~\ref{Def3-6}. Then we could define as in Theorem~\ref{main-Thm} a homomorphism $F_{\bi,\br}$ from $\AA_{w_0}$ to $K_\Z$ corresponding to this initial seed.   

It turns out that all these homomorphisms coincide.
To check it, one needs to show that if $\mathcal{S}_{\bi,\br}$ and $\mathcal{S}_{\bj,\bs}$ are the two initial seeds of
$K_\Z$ coming from $F_{\bi,\br}$ and $F_{\bj,\bs}$, then $\mathcal{S}_{\bi,\br}$ and $\mathcal{S}_{\bj,\bs}$ are connected by a finite sequence of mutations. 

If $\mathbf{i} = \mathbf{j}$, and only $\mathbf{r}$ and $\bs$ differ, then by induction it is enough to consider the elementary mutations described in Lemma~\ref{Lem22}. But the exchange relations corresponding to these mutations of the seeds $\mathcal{S}_{\bi,\br}$ and $\mathcal{S}_{\bj,\bs}$ are instances of $QQ$-system relations.
This is illustrated by the first two mutations of Example~\ref{exam-5.3}. So in this case we are done: if the two seeds $\mathcal{S}_{\bi,\br}$ and $\mathcal{S}_{\bj,\bs}$ differ only by their $\ell$-tuples $\mathbf{r}$ and $\mathbf{s}$, they are connected by a finite sequence of mutations, all coming from the $QQ$-system.     

Now if $\mathbf{i} \not = \mathbf{j}$, we can reduce to the case when these two reduced words differ 
only by a 2-move or a 3-move. The case of a 2-move is similar to the previous case, and can be dealt with using only $QQ$-system relations.  
For a 3-move, Figure~\ref{Fig10} displays this verification in type $A_2$, using the mutation sequence of Proposition~\ref{Prop25}. 
Note that in order to match the third and fourth seeds, we have to use obvious equalities like
\[
\bQ_{s_1(\varpi_2),q^{-1}} = \bQ_{\varpi_2,q^{-1}},
\quad
\bQ_{s_2(\varpi_1),q^{-2}} = \bQ_{\varpi_1,q^{-2}},
\quad
\bQ_{s_1s_2s_1(\varpi_1),q^{-2}} = \bQ_{s_2s_1(\varpi_1),q^{-2}}.
\]
Note also that the exchange relation corresponding to the second mutation is
\[
\bQ_{s_1(\varpi_1),\,q^{-2}} \bQ_{s_2(\varpi_2),\,q^{-1}}
=
\bQ_{\varpi_1,\,q^{-2}} \bQ_{s_1s_2(\varpi_2),\,q^{-1}}
+
\bQ_{s_2s_1(\varpi_1),\,q^{-2}} \bQ_{\varpi_2,\,q^{-1}},
\]
which is \emph{not} a relation coming from the $QQ$-system.

Proposition~\ref{prop-5.5} below proves a more general relation of this type, valid for all $A$, $D$, $E$ types. Using it, one can check that all homomorphisms $F_{\bi,\br}$ coincide.
\cqfd
}
\end{remark}

\begin{Prop}\label{prop-5.5}
Suppose that $c_{ij} = -1$, and let $w \in W$ be such that
$\ell(ws_is_js_i) = \ell(w) + 3.$
For every $r\in\Z$ we have:
\[
 \bQ_{ws_i(\varpi_i),\,q^r} \bQ_{ws_j(\varpi_j),\,q^{r+1}}
=
\ \bQ_{w(\varpi_i),\,q^{r}} \bQ_{ws_is_j(\varpi_j),\,q^{r+1}}
+\ 
\bQ_{ws_js_i(\varpi_i),\,q^{r}} \bQ_{w(\varpi_j),\,q^{r+1}}.
\]
\end{Prop}

\begin{proof}
Following \cite{FH2}, for $a\in\mathbb{C}^*$, we define $\Sigma_{ij,a}$ and $\Sigma_{ji,a}$ by:
\[
\Theta_i(\Sigma_{j,a}) = \Sigma_{ij,a}\Sigma_{i,aq^{-1}}^{-1},\qquad
\Theta_j(\Sigma_{i,a}) = \Sigma_{ji,a}\Sigma_{j,aq^{-1}}^{-1}.
\]
Then the following relation holds in $\Pi$:
$$\Sigma_{i,a}\Sigma_{j,aq} = \Sigma_{ij,aq} +
A_{j,aq}^{-1}\Sigma_{ji,a}.$$
In type $A_2$, this is Eq.~(4.36) in \cite{FH2}. The same proof
as in  \cite[Section 4.5]{FH2} gives this
relation for general simply-laced type. Now this relation can be
rewritten in the following form:
\[
\begin{array}{lll}
\tT_i(\Psi_{i,aq^2})\widetilde{\Psi}_{i,a}^{-1}(1 - [-\alpha_i])^{-1}
\tTheta_j(\Psi_{j,aq^3})\widetilde{\Psi}_{j,aq}^{-1}(1 - [-\alpha_j])^{-1}
\\[2mm]
\qquad=\quad
\tTheta_i(\tTheta_j(\Psi_{j,aq^3})\widetilde{\Psi}_{j,aq}^{-1}(1 -
[-\alpha_j])^{-1})\Sigma_{i,a}
\ +\ 
A_{j,aq}^{-1}\tTheta_j(\tTheta_i(\Psi_{i,aq^2})\widetilde{\Psi}_{i,a}^{-1}(1
- [-\alpha_i])^{-1})\Sigma_{j,aq^{-1}}
\\[2mm]
\qquad=\quad
\tTheta_{s_is_j}(\Psi_{j,aq^3}) (1 - [-\alpha_j-\alpha_i])^{-1})(1 -
[-\alpha_i])^{-1})
\widetilde{\Psi}_{j,aq}^{-1}\Psi_{i,aq^2}\widetilde{\Psi}_{i,a}^{-1}
\\[1mm]
\qquad\qquad
+\quad
A_{j,aq}^{-1}
\tTheta_{s_js_i}(\Psi_{i,aq^2}) (1 - [-\alpha_j-\alpha_i])^{-1})(1 -
[-\alpha_j])^{-1})
\widetilde{\Psi}_{i,a}^{-1}\Psi_{j,aq}\widetilde{\Psi}_{j,aq^{-1}}^{-1}.
\end{array}
\]
This implies:
\[
\begin{array}{lll}
\tTheta_i(\Psi_{i,a})(1 - [-\alpha_i])^{-1}
\tTheta_j(\Psi_{j,aq})(1 - [-\alpha_j])^{-1}
\\[2mm]
\qquad\qquad=\quad
\tTheta_{s_is_j}(\Psi_{j,aq}) (1 - [-\alpha_j-\alpha_i])^{-1})(1 -
[-\alpha_i])^{-1})\Psi_{i,a}
\\[1mm]
\qquad\qquad\qquad+\quad
A_{j,aq^{-1}}^{-1}
\tTheta_{s_js_i}(\Psi_{i,a}) (1 - [-\alpha_j-\alpha_i])^{-1})(1 -
[-\alpha_j])^{-1})
\widetilde{\Psi}_{j,aq^{-1}}\widetilde{\Psi}_{j,aq^{-3}}^{-1}\Psi_{j,aq^{-1}}
\end{array}
\]
that is, after applying the projection $E_e'$,
\[
Q_{s_i(\varpi_i),a}Q_{s_j(\varpi_j),aq}
= Q_{\varpi_i,a}Q_{s_is_j(\varpi_j),aq} +
[-\alpha_j]Q_{\varpi_j,aq}Q_{s_js_i(\varpi_i),a}.
\]
Now for $w\in W$, applying $\tTheta_w$ to this relation gives:
\[
\begin{array}{lll}
(1 - [-w(\alpha_i)])^{-1}(1 -
[-w(\alpha_j)])^{-1}\tTheta_{ws_i}(\Psi_{i,a})\tTheta_{ws_j}(\Psi_{j,aq})
\\[2mm]
\qquad\qquad = \quad(1 - [-w(\alpha_i+\alpha_j)])^{-1}(1 -
[-w(\alpha_i)])^{-1}\tTheta_{w}(\Psi_{i,a})\tTheta_{ws_is_j}(\Psi_{j,aq})
\\[1mm]
\qquad\qquad\qquad
+ \quad [-w(\alpha_j)]
(1 - [-w(\alpha_i+\alpha_j)])^{-1}(1 -
[-w(\alpha_j)])^{-1}\tTheta_{w}(\Psi_{j,aq})\tTheta_{ws_js_i}(\Psi_{i,a})
.
\end{array}
\]
Since we assume that $l(ws_is_js_i) = l(w) + 3$, using the defining formulas for the elements
$\chi_{w(\varpi_i)}$ we have:

\[
\begin{array}{lll}
(1 - [-w(\alpha_i)])(1 -
[-w(\alpha_j)])\chi_{ws_i(\varpi_i)}\chi_{ws_j(\varpi_j)}
\\[2mm]
\qquad\qquad=\qquad (1 - [-w(\alpha_i+\alpha_j)])(1 -
[-w(\alpha_i)])\chi_{w(\varpi_i)}\chi_{ws_is_j(\varpi_j)}
\\[2mm]
\qquad\qquad=\qquad (1 - [-w(\alpha_i+\alpha_j)])(1 -
[-w(\alpha_j)])\chi_{w(\varpi_j)}\chi_{ws_js_i(\varpi_i)}
\\[2mm]
\qquad\qquad=\qquad \displaystyle\prod_{k\notin\{i, j\},\, c_{ik} + c_{jk} < 0}\chi_{w(\varpi_k)}.
\end{array}
\]
Hence, we obtain:
\begin{equation}\label{eq11}
Q_{ws_i(\varpi_i),a}Q_{ws_j(\varpi_j),aq}
= Q_{w(\varpi_i),a}Q_{ws_is_j(\varpi_j),aq} +
[-w(\alpha_j)]Q_{w(\varpi_j),aq}Q_{ws_js_i(\varpi_i),a}.
\end{equation}

Finally, we handle the renormalization by $\Omega$. Comparing the highest
$\ell$-weights in the expansions of both sides of Eq.~(\ref{eq11}) in $\widetilde{\mathcal{Y}}_e$, we obtain
$$\Psi_{ws_i(\varpi_i),a}\Psi_{ws_j(\varpi_j),aq} =
\Psi_{w(\varpi_i),a}\Psi_{ws_is_j(\varpi_j),aq}.$$
Inverting the role of $i$ and $j$, we have also :
$$\Psi_{ws_i(\varpi_i),a}\Psi_{ws_j(\varpi_j),aq^{-1}} =
\Psi_{w(\varpi_j),aq^{-1}}\Psi_{ws_js_i(\varpi_i),a}.$$
Assume that $a = q^r$. Recall that
$[w(\varpi_i)]\Psi_{w(\varpi_i),q^{r -
1}}\Psi_{w(\varpi_i),q^{r+1}}^{-1}$ is a Laurent monomial in the
variables $Y_{j,b}$.
Hence
\begin{equation}\label{shiftq}\left[\Omega(\Psi_{w(\varpi_i),q^{r+1}})\right] =
[w(\omega_i)]\left[\Omega(\Psi_{w(\varpi_i),q^{r-1}})\right].\end{equation}
Now
$$\Omega(\Psi_{ws_i(\varpi_i),q^r}\Psi_{ws_j(\varpi_j),q^{r-1}}) =
\Omega(\Psi_{w(\varpi_j),q^{r-1}}\Psi_{ws_js_i(\varpi_i),q^r})$$
and applying (\ref{shiftq}) to $\Psi_{ws_j(\varpi_j),q^{r-1}}$ and $\Psi_{w(\varpi_j),q^{r-1}}$, this implies that
$$\left[\Omega(\Psi_{ws_i(\varpi_i),q^r}\Psi_{ws_j(\varpi_j),q^{r+1}})\right] =
[ w(\varpi_j) - w(s_j(\varpi_j))]\left[\Omega(\Psi_{w(\varpi_j),q^{r+1}}\Psi_{ws_js_i(\varpi_i),q^r})\right].$$
Since $w(\varpi_j) - w(s_j(\varpi_j)) = w(\alpha_j)$, we obtain that:
$$
\left[\Omega(\Psi_{w(\varpi_i),q^r}\Psi_{ws_is_j(\varpi_j),q^{r+1}})\right]
= \left[\Omega(\Psi_{ws_i(\varpi_i),q^r}\Psi_{ws_j(\varpi_j),q^{r+1}})\right]
=
[w(\alpha_j)]\left[\Omega(\Psi_{w(\varpi_j),q^{r+1}}\Psi_{ws_js_i(\varpi_i),q^r})\right].$$
So we obtain:
$$\bQ_{ws_i(\omega_i),q^r}\bQ_{ws_j(\omega_j),q^{r+1}}
= \ \bQ_{w(\omega_i),q^r}\bQ_{ws_is_j(\omega_j),q^{r+1}} + \
\bQ_{w(\omega_j),q^{r+1}}\bQ_{ws_js_i(\omega_i),q^r}.$$
\cqfd

\end{proof}

\section{Proof of Theorem~\ref{main-Thm}}

Recall the notation of \S\ref{ssubsec-main-th}. 
Let $\mathcal{Z} := \{z_{(i,r)} \mid (i,r)\in V \}$
denote the set of cluster variables of the initial seed of $\AA_{w_0}$. Let $F: \AA_{w_0} \to K_\Z$ be the map defined in Theorem~\ref{main-Thm}.  

\begin{Prop}\label{prop-alg-ind}
The elements of $F(\mathcal{Z})$ are algebraically independent in $K_\Z$. 
\end{Prop}

\begin{proof}
Recall the notation $\Psi_{w(\varpi_i),a}$ for the highest weight monomial of $Q_{w(\varpi_i),a}$.
Suppose that $w = s_{i_1}\cdots s_{i_k}$ is a reduced expression, and that $a = q^r$, where $(i,r)\in V$. Then it follows from \cite[Theorem 4.5]{FH3} that $\Psi_{w(\varpi_i),q^r}$ can be calculated using a multiplicative analogue of the braid group action introduced in \S\ref{subsec434}. More precisely, we have 
\[
 \Psi_{w(\varpi_i),q^r} = \prod_{(j,s)\in V} \Psi_{j,q^s}^{m_{j,s}},
\]
where the exponents $m_{j,s}$ are given by
\[
 \theta_{i_1}\cdots\theta_{i_k}(\be_{(i,r)}) = \sum_{(j,s)\in V} m_{j,s}\be_{(j,s)}.
\]
Therefore, by Proposition~\ref{prop-braid}, the exponents of the highest weight monomial of $F(z_{(i,r)})$ coincide with the coordinates of the stabilized $g$-vector $\bg_{(i,r)}^{(\infty)}$. 
Since the stabilized $g$-vectors form a $\Z$-basis of $\Z^V$, the 
elements of $F(\mathcal{Z})$ are algebraically independent in $K_\Z$. \cqfd
\end{proof}

\begin{Prop}\label{prop-polynomial}
For every $(i,r)\in V$ and $w\in W$, the element $\bQ_{w(\varpi_i),\,q^r}$ is the image by $F$ of a cluster variable of $\AA_{w_0}$.
\end{Prop}

\begin{proof}
By Remark~\ref{Rem54} and Proposition~\ref{prop-5.5}, for every datum $(\bi,\br)$ the images of all the cluster variables of the seed $\G_{(\bi,\br)}$ of $\AA_{w_0}$ by $F$ are of the form $\bQ_{w(\varpi_i),q^r}$. By varying $\br$ and $\bi$, we can obtain all possible $w(\varpi_i)$ and $r$ such that $(i,r)\in V$.
Therefore all elements $\bQ_{w(\varpi_i),\,q^r}$ are images under $F$ of a cluster variable of $\AA_{w_0}$.
\cqfd
\end{proof}

\begin{example}
{\rm
In type $A_3$, choosing $\bi = (1,2,1,3,2,1)$, as in Example~\ref{exam-5.3}, we will get all variables
$\bQ_{w(\varpi_i),\,q^r}$ for the following weights $w(\varpi_i)$:
\[
 \varpi_1,\ s_1(\varpi_1),\ s_2s_1(\varpi_1),\ s_3s_2s_1(\varpi_1),\ \varpi_2,\ s_1s_2(\varpi_2),\
 s_2s_1s_3s_2(\varpi_2),\ \varpi_3,\ s_1s_2s_3(\varpi_3). 
\]
With $\bi = (2,1,3,2,1,3)$ we will get in addition the weights:
\[
 s_2(\varpi_2),\ s_2s_3(\varpi_3).
\]
Finally, with $\bi = (3,1,2,3,1,2)$ we will get the weight $s_1s_3s_2(\varpi_2)$, and with $\bi = (3,2,3,1,2,3)$ we will get the last missing weight $s_3s_2(\varpi_2)$.
\cqfd
}
\end{example}

We can now complete the proof of Theorem~\ref{main-Thm}. First, the homomorphism $F$ is well-defined because by Proposition~\ref{prop-invertible} every element $\bQ_{w_0(\varpi_k),\,q^t}$  is invertible in $K_\Z$. The injectivity of $F$ follows from Proposition~\ref{prop-alg-ind}. The remaining statement then follows from the definition of $K_\Z$ and Proposition~\ref{prop-polynomial}. \cqfd


\section{Shifted quantum affine algebras}\label{sect-shift-quantum}
 
In this section, we give a rapid overview of shifted quantum affine algebras and their
category~$\O^{\rm sh}$. The reader is referred to \cite{FT} and \cite{H} for more details and references.
The main motivations for the study of these algebras and their representations come from 
quantized $K$-theoretic Coulomb branches of certain quiver gauge theories, 
and from representation theory of Borel subalgebras of quantum affine algebras.

Our main point here is to explain that the ring $K_\Z$ coincides with the $q$-character ring of a subcategory $\O_\Z$ of $\O^{\rm sh}$, specified by certain integrality conditions on the highest loop-weights. 
This allows us to translate Theorem~\ref{main-Thm} into Theorem~\ref{isomgv}, and give an interpretation of the cluster 
algebra $\AA_{w_0}$ in terms of the category $\O_\Z$.

\subsection{Representations of shifted quantum affine algebras}\label{sub-sec-shift-quantum}
 
As above, we assume that $\g$ is of simply-laced type.
Because of this, we can identify coweight lattice and weight lattice, and regard the shifting parameters $\mu$ arising in the definition below as integral weights rather than coweights.
We fix a non-zero complex number $q$, not a root of unity. 
 
\begin{Def}[\cite{FT}] Let $\mu\in P$.
The shifted quantum affine algebra $U_q^\mu(\hat{\mathfrak{g}})$
  is the $\C$-algebra with the same Drinfeld generators $\phi_{i,r}^\pm$, $x_{i,r}^\pm$ ($i\in I$, $r\in\mathbb{Z}$) as the
  ordinary quantum affine algebra $U_q(\hat{\mathfrak{g}})$ and the same relations, except that
  the definition of the series $\phi_i^-(z)$ is
  modified as follows:
\begin{equation}\label{phrelshm}\phi_i^-(z) =
  \sum_{m\in\mathbb{Z}}\phi_{i, m}^- z^{ m} =
  \phi_{i,\alpha_i(\mu)}^-z^{\alpha_i(\mu)} \exp\left(- (q -
  q^{-1})\sum_{r > 0} h_{i,-r } z^{- r} \right)\end{equation}
(the definition of the series $\phi_i^+(z)$ stays the same). In
addition, we impose the condition that $\phi_{i,0}^+$ and
$\phi_{i,\alpha_i(\mu)}^-$ are invertible (but they are not
necessarily inverse to each other).
\end{Def}
The complete list of relations of $U_q^\mu(\hat{\mathfrak{g}})$ can be found in \cite[Section 5]{FT}, see also \cite[Section 3.1]{H}.

Recall the basis $(h_i)_{i\in I}$ of the Cartan subalgebra $\mathfrak{h}$. For a $U_q^\mu(\hat{\mathfrak{g}})$-module $V$ and $\lambda\in P$, we define
\[
V_{\lambda}:=\left\{v\in V \mid \phi_{i,0}^+\, v = q^{\lambda(h_i)} v,\ \mbox{ for every } i\in I\right\},
\]
and call it the weight space of $V$ of weight $\lambda$.
We say that $V$ is Cartan-diagonalizable
if
\[
V = \bigoplus_{\lambda\in P}V_{\lambda}.
\]

For $\lambda\in P$, recall the cone $C_\lambda = \lambda - Q_+\subset P$ defined in Section \ref{sect3-2}.

\begin{Def}[\cite{H}] A $U_q^\mu(\hat{\mathfrak{g}})$-module $V$ is said to be in category $\mathcal{O}^\mu$ if
\begin{itemize} 
\item[i)] $V$ is Cartan-diagonalizable;
\item[ii)] for all $\lambda\in P$ we have 
$\dim (V_{\lambda})<\infty$;
\item[iii)] there exist finitely many weights 
$\lambda_1,\ldots,\lambda_s\in P$ 
such that all weights of $V$ are contained in 
${\bigcup}_{j=1,\ldots, s}\,C_{\lambda_j}$.
\end{itemize}
\end{Def}

\begin{remark}{\rm
This category is slightly different from the category introduced in \cite{H}, in that here we only allow integral weights $\lambda$.
}
\end{remark}

\begin{Def} A collection $\Psib=(\Psi_{i, m})_{i\in I, m\geq 0}$
of complex numbers $\Psi_{i,m}$ such that $\Psi_{i,0} \in q^\Z$ for all $i\in I$ is called an $\ell$-weight. We denote by $\tb^\times_\ell$ the set of $\ell$-weights.
\end{Def}
 
We will identify the collection $(\Psi_{i, m})_{m\geq 0}$ with its
generating series and use this to represent every $\ell$-weight as an
$I$-tuple of formal power series in $z$:
\[
\Psib = (\Psi_i(z))_{i\in I},
\qquad
\Psi_i(z) := \sum_{m\geq 0} \Psi_{i,m} z^m.
\]
Since each $\Psi_i(z)$ is an invertible formal power series,
$\tb^\times_\ell$ has a natural group structure by pointwise multiplication. 

We have a surjective group homomorphism $\varpi :
\tb^\times_\ell\rightarrow P$ given by
\begin{equation}
\label{spep}
 \varpi(\Psib) := \sum_i a_i\varpi_i, \qquad (\Psib\in \tb^\times_\ell),
\end{equation}
where for $i\in I$ we have written $\Psib_i(0) = q^{a_i}$.
This is well-defined because $q$ is not a root of unity.

\begin{Def}
We say that an $\ell$-weight $\Psib\in \tb^\times_\ell$ is rational if 
$\Psi_i(z)$ is the power series expansion of a rational function in $z$ for every $i\in I$.
We denote by $\mfr$ the subgroup of $\tb^\times_\ell$ consisting of all
rational $\ell$-weights.
\end{Def}

\begin{Thm}[\cite{H}] \label{Th-classification}
The simple representations in $\mathcal{O}^\mu$ are parameterized by the subset $\mathfrak{r}_\mu$ of rational $\ell$-weights of degree $\mu$, that is, the rational $\ell$-weights $\Psib = (\Psi_i(z))_{i\in I}$  such that: 
$$ \sum_{i\in I} \deg(\Psi_i(z)) \varpi_i = \mu.$$ 
\end{Thm}
The simple representation associated with $\Psib\in\mathfrak{r}_\mu$ will be denoted by $L(\Psib)$. The $\ell$-weight $\Psib$ encodes the eigenvalues of the generators $\phi_{i,m}^+$ ($i\in I$, $m\geq 0$) on a highest weight vector $v\in L(\Psib)$ :
$$\phi_{i,m}^+\cdot v=\Psi_{i, m}v,\quad (i\in I,\ m\ge 0).$$

The fundamental examples of simple representations come from the simplest rational $\ell$-weights of degree $\pm\varpi_i$ and are defined as follows \cite{HJ, H}. For $i\in I$ and $a\in\C^\times$, let 
\begin{equation}
L_{i,a}^\pm := L\left(\Psib_{i,a}^{\pm 1}\right)
\quad \mbox{where}\quad 
(\Psib_{i,a})_j(z) := 
\left\{ 
\begin{array}{cc}
1 - za & (j=i)\,,\\
1 & (j\neq i)\,.\\
\end{array}
\right.
\label{fund-rep}
\end{equation}
We call $L_{i,a}^+$ (resp. $L_{i,a}^-$) a {\em positive (resp. negative) prefundamental representation}.

For $\lambda \in P$, we also have a one-dimensional
invertible representation $L(\Psib_\lambda)$ in $\mathcal{O}^0$, where
\begin{equation}\label{clw}
\Psib_\lambda = (q^{\lambda(h_i)})_{i\in I}\in \tb_\ell^\times
\end{equation}
is the unique constant $\ell$-weight such that
$\varpi(\Psib_\lambda) = \lambda$.

Note that $\mathfrak{r}$ is the free (multiplicative) abelian group generated by the $\ell$-weights $\Psi_{i,a}\,((i,a)\in I\times\C^*)$ 
and $\Psi_{\la}\,(\la\in P)$.

\begin{Thm}[\cite{H}] \label{thmH}
Let $\mu\in P$.
\begin{itemize}
\item The algebra $U_q^\mu(\hat{\mathfrak{g}})$ has a non-zero finite-dimensional representation if and only if $\mu$ is a dominant weight.
\item If $\mu$ is antidominant, then there is a natural embedding of the ordinary quantum affine Borel algebra $U_q(\widehat{\mathfrak{b}})$ into $U_q^\mu(\hat{\mathfrak{g}})$.
\end{itemize}
\end{Thm}
To illustrate, the prefundamental modules have the following properties :
\begin{itemize}
\item the $U_q^{\varpi_i}(\hat{\mathfrak{g}})$-module $L_{i,a}^+$ is one-dimensional.

\item the $U_q^{-\varpi_i}(\hat{\mathfrak{g}})$-module $L_{i,a}^-$ is infinite-dimensional. It remains simple when restricted to the Borel subalgebra $U_q(\widehat{\mathfrak{b}})$ naturally embedded in $U_q^{-\varpi_i}(\hat{\mathfrak{g}})$. The corresponding $U_q(\widehat{\mathfrak{b}})$-module can be obtained as a limit of Kirillov-Reshetikhin modules \cite{HJ}.
\end{itemize}

Let $\mathcal{E}$
be the additive group of maps $c : P \rightarrow \mathbb{Z}$
whose support
is contained in
a finite union of cones $C_\mu$.
For $\lambda\in P$, we define $[\lambda] = \delta_{\lambda,-}\in\mathcal{E}$. We will regard elements of $\mathcal{E}$
as formal sums
\[
c = \sum_{\omega\in \mathrm{Supp}(c)} c(\omega)[\omega].
\]
 
\begin{Def}
For $V$ in the category $\mathcal{O}^\mu$ we define the character of
$V$ as:
$$\chi(V) := \sum_{\lambda\in P}
\dim(V_\lambda) [\lambda] \in \mathcal{E}.$$
\end{Def} 
 
\begin{remark} 
{\rm
We have a ring structure on $\mathcal{E}$ so that $[\lambda + \lambda'] = [\lambda][\lambda']$ for any $\lambda,\lambda'$. This multiplicative notation is compatible with the notation $[\lambda]$ introduced in \S\ref{subsection-Pi-prime}.
}
\end{remark}
 
Let $K_0(\mathcal{O}^\mu)$ be the Grothendieck group of the category $\mathcal{O}^\mu$. 
Let $\mathcal{O}^{\,\mathrm{sh}}$ denote the direct sum of the abelian categories $\mathcal{O}^\mu\ (\mu\in P)$.
It is shown in \cite{H} that the direct sum of Grothendieck groups
$$K_0(\mathcal{O}^{\,\mathrm{sh}}) := \bigoplus_{\mu\in P}K_0(\mathcal{O}^\mu)$$
has a ring structure, coming from a fusion product constructed from the deformed Drinfeld coproduct, and we have 
$$K_0(\mathcal{O}^\mu)\cdot K_0(\mathcal{O}^{\mu'})\subset K_0(\mathcal{O}^{\mu + \mu'}).$$
In particular, $K_0(\mathcal{O}^0)$ is a subring of $K_0(\mathcal{O}^{\,\mathrm{sh}})$. 

We naturally identify $\mathcal{E}$ with the Grothendieck ring of the category of representations in
$\mathcal{O}^0$ with constant $\ell$-weights,
the simple objects of which are the $[L(\Psib_\lambda)]$, $\lambda\in P$, that are identified with their character $[\lambda]$. The character map $\chi$ induces a ring homomorphism $\chi: K_0(\mathcal{O}^{\,\mathrm{sh}})\rightarrow \mathcal{E}$, which is not injective.

Let $\mathcal{E}_\ell$ be the additive group of maps $\gamma : \mathfrak{r} \rightarrow \Z$
so that $\varpi(\mathrm{supp}(\gamma))$ is contained in
a finite union of sets of the form $C_\mu$, where
$$\mathrm{supp}(\gamma) := \{\Psib\in \mathfrak{r}\mid \gamma(\Psib) \neq 0\}.$$
Then $\mathcal{E}_\ell$ has a natural ring structure.
For $\Psib\in \tb^\times_\ell$, we define $[\Psib] = \delta_{\Psib,-}\in\mathcal{E}_\ell$.

\begin{remark}\label{rc}
{\rm 
\begin{itemize}
\item[(i)] 
The group homomorphism $\varpi :
\tb^\times_\ell\rightarrow P$ given by Equation~(\ref{spep})
is naturally extended to a surjective homomorphism
$$\varpi : \mathcal{E}_\ell\rightarrow \mathcal{E}.$$

\item[(ii)] Recall the ring $\YY'$ of \S\ref{subsection-Pi-prime}, generated by variables $\Psi_{i,a}^{\pm 1}$, $[\la]$ with $i\in I$, $a\in\mathbb{C}^*$ and $\la\in P$. 
We have a natural identification of this ring $\YY'$ with the group algebra of $\mfr$, by identifying the variables $\Psi_{i,a}$ with the $\ell$-weights $\Psib_{i,a}\in \mfr$ defined in Equation (\ref{fund-rep}), and the variables  
$[\la]$ with the $\ell$-weights $\Psib_\lambda\in\mfr$ defined in Equation (\ref{clw}).

\item[(iii)] It follows that the completion $K =  \widetilde{\YY_e}'$ of \S\ref{subsection-KZ} identifies with $\mathcal{E}_\ell$. 
\end{itemize}
}
\end{remark}

\begin{Thm}[\cite{H}] 
The $\ell$-weights of a representation in $\mathcal{O}^\mu$ belong to $\mathfrak{r}_\mu$.
\end{Thm}

It follows that every representation $V$ in the category $\mathcal{O}^\mu$ has a $q$-character
$$\chi_q(V) = \sum_{\Psib\in
\mathfrak{r_\mu}}\dim(V_{\Psib})[\Psib]\in\mathcal{E}_\ell,$$
where $V_{\Psib}$ is the $\ell$-weight space of $\ell$-weight $\Psib$ (common generalized eigenspace).

Since the $q$-character map is additive on short exact sequences, it gives rise to an injective ring homomorphism \cite[Corollary 5.1]{H}:
$$\chi_q : K_0(\mathcal{O}^{\,\mathrm{sh}})\rightarrow \mathcal{E}_\ell.$$
In fact, $\chi_q$ is a ring isomorphism 
(the argument is the same as in \cite[Theorem 4.19]{W}).

Since $\mathcal{E}_\ell$ has the structure of a topological ring, $K_0(\mathcal{O}^{\,\mathrm{sh}})$ itself inherits the structure of a topological ring. The classes $[L(\Psib)]\ (\Psib \in \mathfrak{r})$ of simple modules form a topological basis of $K_0(\mathcal{O}^\mathrm{sh})$. 

The construction in \cite{H} implies also immediately the following:
\begin{Prop}\label{fact} For $\Psib, \Psib'\in \tb^\times_\ell$, $[L(\Psib\Psib')]$ occurs in $[L(\Psib)][L(\Psib')]$ with multiplicity $1$ and
$$[L(\Psib)][L(\Psib')] = [L(\Psib\Psib')] + \sum_{\Psib''\in\tb^\times_\ell,\varpi(\Psib'')< \varpi(\Psib') + \varpi(\Psib)} m_{\Psib''}[L(\Psib'')]$$
for some coefficients $m_{\Psib''}\in \Z_{\ge 0}$.
\end{Prop}

\begin{Def}
Let $\O_\Z$ denote the full subcategory of $\mathcal{O}^{\,\mathrm{sh}}$ consisting of the objects $M$ which satisfy the following property:   
\begin{quote}
every simple subquotient of $M$ is of the form $L(\Psib)$ where $\Psib$ is a Laurent monomial in the variables $[\la]\ (\la\in P)$ and $\Psi_{i,q^r}\ (i,r)\in V$.
\end{quote}
\end{Def}

The Grothendieck group $K_0(\mathcal{O}_{\mathbb{Z}})$ is the $[P]$-topological subspace of $K_0(\mathcal{O}^{\,\mathrm{sh}})$
generated by the simple classes $[L(\Psib)]$ such that
$\Psib\in \mathbb{Z}[\Psib_{i,q^r}^{\pm 1}]_{(i,r)\in V}$.

\begin{Prop} $K_0(\mathcal{O}_{\mathbb{Z}})$ is a subring of $K_0(\mathcal{O}^{\,\mathrm{sh}})$. The homomorphism $\chi_q$ restricts to an isomorphism
$$\chi_q : K_0(\mathcal{O}_{\mathbb{Z}}) \rightarrow \mathcal{E}_{\ell, \mathbb{Z}}$$
where $\mathcal{E}_{\ell,\mathbb{Z}}$ is the $[P]$-topological subring of $\mathcal{E}_\ell$ generated by the $[\Psib_{i,q^r}^{\pm 1}]$, $(i,r)\in V$.
\end{Prop}

\begin{proof} 
The first point follows from the second point.

For the second point, let us prove that for $(i,r)\in V$, the
$q$-character of the negative prefundamental representation $L_{i,q^{r}}^-$ has a $q$-character
which belongs to $\mathcal{E}_{\ell,\mathbb{Z}}$.
By \cite{HJ, H},
this $q$-character can be obtained as a limit of
$q$-characters of finite-dimensional Kirillov-Reshetikhin modules. The $q$-character of a KR-module 
can be obtained by an algorithm of Frenkel-Mukhin, and this implies the result (this also follows from the main result of \cite{HL1}).

As in addition $[\Psib_{i,q^r}] = \chi_q(L(\Psib_{i,q^r}))$ for any $(i,r)\in V$, we obtain from Proposition \ref{fact}
$\chi_q(L(\Psib))\in \mathcal{E}_{\ell,\mathbb{Z}}$ for any simple class $[L(\Psib)]$ among the generators of $K_0(\mathcal{O}_{\mathbb{Z}})$.

Now $\mathcal{E}_{\ell,\mathbb{Z}}$ is generated, as a $[P]$-topological ring, by the $[\Psib_{i,q^r}] = \chi_q(L(\Psib_{i,q^r}))$ and by the $\chi_q(L(\Psib_{i,q^r}^{-1}))$ for $(i,r)\in V$. This concludes the proof.
\cqfd
\end{proof}

Recall the subring $K_\mathbb{Z}\subset [\uP]\otimes_{[P]} K$ introduced in \S\ref{subsection-KZ}.
As each generator of $K_\mathbb{Z}$ belongs to $\mathcal{E}_{\ell,\mathbb{Z}}$, we can identify $K_\mathbb{Z}$ with $ [\uP]\otimes_{[P]}\mathcal{E}_{\ell,\mathbb{Z}}$.

Combining the previous results, Theorem~\ref{main-Thm} now implies immediately the following:

\begin{Thm}\label{isomgv} 
We have an injective ring homomorphism
$$I : \mathcal{A}_{w_0}\rightarrow [\uP]\otimes_{[P]}K_0(\mathcal{O}_{\mathbb{Z}})$$
and the topological closure of $ [\uP]\otimes_{[P]} I(\mathcal{A}_{w_0})$ is the entire topological ring $ [\uP]\otimes_{[P]} K_0(\mathcal{O}_{\mathbb{Z}})$.
\end{Thm}

We say that a simple class $[L(\Psib)]$ in $K_0(\mathcal{O}^\mathrm{sh})$ is real if $[L(\Psib)]^2 = [L(\Psib^2)]$ is a simple class.

\begin{Conj}\label{main-conj} 
The image of a cluster monomial of $\mathcal{A}_{w_0}$ is a real simple class in $ [\uP]\otimes_{[P]}K_0(\mathcal{O}_{\mathbb{Z}})$.
\end{Conj}

In the rest of this section we will provide some evidence supporting Conjecture~\ref{main-conj}.

\subsection{Finite-dimensional representations}

Let $\mathcal{C}_{\mathbb{Z}}$ denote the full subcategory of
$\mathcal{O}_{\mathbb{Z}}$ whose objects are the
finite-dimensional modules.

\begin{Thm}[\cite{H}] The simple representations in $\mathcal{C}_{\mathbb{Z}}$ are the
simple modules $L(\Psib)$ in $\mathcal{O}_{\mathbb{Z}}$ such that
$\Psib$ is a monomial in the $\ell$-weights:
\[
\Psib_{\lambda},\quad \Psib_{i,q^r},\quad Y_{i,q^{r+1}} = [\omega_i]\frac{\Psib_{i,q^r}}{\Psib_{i,q^{r+2}}},
\qquad (\lambda\in P,\ (i,r)\in V).
\]
\end{Thm}

For example, the positive prefundamental representations $L_{i,q^r}^+ = L(\Psi_{i,q^r})\ ((i,r)\in V)$ are contained in 
$\mathcal{C}_{\mathbb{Z}}$, but the negative prefundamental representations $L_{i,q^r}^- = L(\Psib_{i,q^r}^{-1})$ are not. The category $\mathcal{C}_{\mathbb{Z}}$ also contains the fundamental representations $V_{i,q^{r+1}} = L(Y_{i,q^{r+1}})$ of the ordinary quantum affine algebra $U_q(\hg)$.

Consider the corresponding subring $K_0(\mathcal{C}_{\mathbb{Z}})$ of $K_0(\mathcal{O}_{\mathbb{Z}})$. Note that
the simple classes in $K_0(\mathcal{C}_{\mathbb{Z}})$
form a basis of $K_0(\mathcal{C}_{\mathbb{Z}})$
(not only a topological basis) and that each product
$[L(\Psib)][L(\Psib')]$ in $K_0(\mathcal{C}_{\mathbb{Z}})$
is a finite sum of simple classes.

Recall the cluster algebra $\AA_\G = \mathcal{A}_e$ associated with the quiver $\G=\Gamma_e$. 
Building on results of \cite{HJ} and \cite{KKOP}, the following theorem was established in \cite{H}.

\begin{Thm}\label{fdrep} There is an isomorphism of algebras 
\[
[\uP]\otimes_{[P]} \mathcal{A}_e \to  [\uP]\otimes_{[P]} K_0(\mathcal{C}_{\mathbb{Z}}) 
\]
mapping the initial cluster variable of $\mathcal{A}_e$ attached to vertex $(i,r)$ of $\G_e$ to the renormalized positive prefundamental representation $ [-\Omega(\Psi_{i,q^r})]L(\Psi_{i,q^r})$.
Moreover, all cluster monomials in $\mathcal{A}_e$ are mapped to simple real classes in 
$[\uP]\otimes_{[P]}K_0(\mathcal{C}_{\mathbb{Z}})$.
\end{Thm}

We have a natural embedding
$\mathcal{A}_e\rightarrow \mathcal{A}_{w_0}$ fitting into the commutative diagram:
\[\xymatrix{\mathcal{A}_{e} \ar[r] \ar[d]& [\uP]\otimes_{[P]}K_0(\mathcal{C}_{\mathbb{Z}}) \ar[d]
\\\mathcal{A}_{w_0} \ar[r]&  [\uP]\otimes_{[P]}K_0(\mathcal{O}_{\mathbb{Z}})}
\]
This embedding comes from Remark~\ref{Remark3-7}, which shows that the initial seed of $\mathcal{A}_e$ with quiver $\Gamma = \Gamma_e$ can be regarded as a limit of a sequence of mutations applied to an initial seed of $\mathcal{A}_{w_0}$. In particular, since by construction any cluster variable of $\mathcal{A}_e$ is obtained from its initial seed via a finite sequence of mutations, it can also be obtained by performing the same sequence of mutations starting from one of the initial seeds of $\mathcal{A}_{w_0}$. Therefore every cluster variable (\resp cluster monomial) of $\AA_e$ is a cluster variable (\resp cluster monomial) of $\mathcal{A}_{w_0}$. Hence, Theorem \ref{fdrep} implies Conjecture \ref{main-conj} for all cluster monomials of $\AA_{w_0}$
contained in $\mathcal{A}_e$.
However, there are infinitely many isoclasses in $\mathcal{O}_{\Z}$ which do not belong to $\mathcal{C}_{\mathbb{Z}}$, 
and infinitely many cluster monomials of $\mathcal{A}_{w_0}$ which do not belong to $\mathcal{A}_e$.

\subsection{$Q$-variables}

Recall the $Q$-variables introduced in Section 7.1. For $i\in I$, $w\in W$, and $a\in \C^*$, we have the $Q$-variable 
\[
Q_{w(\varpi_i),\,a} \in K,
\]
whose leading term $\Psi_{w(\varpi_i),\,a}$ is a Laurent monomial in the variables $\Psi_{j,b}$, explicitly computed in \cite{FH3} in terms of Chari's braid group action.

By Remark \ref{rc}(iii), $K$ identifies with $\mathcal{E}_\ell$. Hence it makes sense to ask whether $Q$-variables are $q$-characters of simple representations. The following conjecture was formulated in \cite{FH3}.

\begin{Conj}\label{cjq} We have the $q$-character formula :
$$\chi_q(L(\Psi_{w(\varpi_i),a})) = Q_{w(\varpi_i),a}.$$
\end{Conj}

By our construction, for $(i,r)\in V$, the $Q$-variable $\bQ_{w(\varpi_i),q^r}$ (which is a rescaling of $Q_{w(\varpi_i),q^r}$ by a constant $\ell$-weight) is the image in $K_{\mathbb{Z}}$ of a cluster variable of $\mathcal{A}_{w_0}$. So Conjecture \ref{main-conj} for these cluster variables is compatible with Conjecture \ref{cjq}. In addition, Conjecture \ref{main-conj} states that the corresponding simple module should be real, and Conjecture \ref{cjq} gives the precise highest $\ell$-weight of the corresponding simple class. If Conjecture \ref{main-conj} is true for these cluster variables, the $QQ$-systems, which correspond to exchange relations in $\mathcal{A}_{w_0}$, are the conjectural extended $QQ$-systems between simple classes conjectured in \cite{FH3}.

\subsubsection{Simple reflections}

\begin{Prop} \label{prop-s_i}
Conjectures \ref{main-conj} and \ref{cjq} hold for the $Q$-variables associated with $w=e$ and $w=s_i\ (i\in I)$.
\end{Prop}

\begin{proof} Conjecture \ref{cjq} is proved in \cite{FH3} for
$w = e$ or $w = s_i$. We reproduce the short proof here for the convenience of the reader.

For $w = e$, we have
$$\chi_q(L(\Psi_{i,a})) = Q_{\varpi_i,a} = \Psib_{i,a}.$$
Moreover, $L(\Psib_{i,a})$ is real as it is one-dimensional.

For $w = s_i$ and $j\not = i$, we have $\Psib_{w(\varpi_j),\,q^r} = \Psib_{j,q^r}$, and the result follows as above. 
For $j = i$, we have $\Psib_{s_i(\omega_i),\,q^r} = \widetilde{\Psib}_{i,q^{r-2}}$, and
$$Q_{w(\varpi_i),q^r} = \widetilde{\Psib}_{i,q^{r-2}}
(1 + A_{i,q^{r-2}}^{-1}(1 + A_{i,q^{r-4}}^{-1}(1 + \cdots))\cdots)= \chi_q(L(\widetilde{\Psib}_{i,q^{r-2}})), $$
where the second equality comes from \cite[Example 5.2]{H}.
This proves Conjecture \ref{cjq} in this case. 
To finish the proof of Conjecture \ref{main-conj}, it suffices to prove that $L(\widetilde{\Psib}_{i,q^{r-2}})$ is real.

Recall the second point of Theorem~\ref{thmH}. 
For $\g = \Sl_2$, we have $i = 1$, and $\widetilde{\Psib}_{1,q^{r-2}} = \Psi_{1,q^{r-2}}^{-1}$. Hence  
$L(\widetilde{\Psib}_{1,q^{r-2}})$ is a negative prefundamental module, and by \cite[Theorem 5]{H} we know that 
negative prefundamental modules are real. In particular, the character of $L(\widetilde{\Psib}_{1,q^{r-2}})$ is 
$(1 - [-\alpha_1])^{-2}$.

For a general $\g$, consider the subalgebra of $U_q^{-2\varpi_i}(\widehat{\mathfrak{g}})$ generated by the Drinfeld generators of index $i$, which is isomorphic to $U_q^{-2\varpi_i}(\widehat{\mathfrak{sl}}_2)$. For this subalgebra, the submodule generated by the highest weight vector of $L(\widetilde{\Psib}_{i,q^{r-2}}^2)$ is isomorphic to $L(\Psib_{1,q^{r-2}}^{-2})$ studied above. By the above discussion,
this implies that the weight multiplicities of $\chi(L(\widetilde{\Psib}^{2}_{i,q^{r-2}}))$ are at least as large as those of $(1 - [-\alpha_i])^{-2}$. But this is precisely the character associated with $[L(\widetilde{\Psib}_{i,q^{r-2}})]^2$. Hence
$[L(\widetilde{\Psib}_{i,q^{r-2}})]^2 = [L(\widetilde{\Psib}_{i,q^{r-2}}^2)]$ and $L(\widetilde{\Psib}_{i,q^{r-2}})$ is real. \cqfd
\end{proof}

\subsubsection{Longest Weyl group element}

\begin{Prop} \label{prop-w_0}
Conjectures \ref{main-conj} and \ref{cjq} hold for the $Q$-variables associated with $w=w_0$.
\end{Prop}

\begin{proof}
It is noted in \cite[Remark 4.10]{FH0} that a Baxter $TQ$-relation can be written not only in 
terms of positive prefundamental representations, but also, using a dual construction, in terms of 
negative prefundamental representations. This is a particular case of the extended 
$TQ$-relations of \cite[Conjecture 3.8]{FH3}, corresponding to the longest Weyl group element $w_0$ 
(see \cite[Remark 3.3]{FH3}). The substitution goes as follows. Variable $Y_{i,a}$ in the 
$q$-character of a representation $V$ is replaced by 
\[
[w_0(\varpi_i)]\,\frac{[L_{\nu(i),aq^{-h-1}}^-]}{[L_{\nu(i),aq^{-h+1}}^-]},
\]
where $h$ is the Coxeter number.
Then we obtain a valid relation in the Grothendieck ring, that is, a solution to the Baxter $TQ$-relation with $w = w_0$.  
In general, a solution is a substitution 
\[
Y_{i,a}\mapsto [w_0(\varpi_i)]\, \frac{X_{i,aq^{-1}}}{X_{i,aq}}
\]
which 
fixes $q$-characters of finite-dimensional representations. Our solution satisfies
$$X_{i,a}\in \Psib_{\nu(i),aq^{-h}}^{-1}\,\mathbb{Z}[[A_{j,b}^{-1}]]_{j\in I,b\in\mathbb{C}^*}\in K,$$
with highest weight term $\Psib_{\nu(i),aq^{-h}}^{-1}$. Such a solution is unique.

Indeed, consider another similar substitution $X$. Then $X(Y_{i,a}) \in Y_{\nu(i),aq^{-h}}^{-1}\mathbb{Z}[[A_{j,b}^{-1}]]_{j\in I,b\in\mathbb{C}^*}$
has highest weight term $Y_{\nu(i),aq^{-h}}^{-1}$. Then the Baxter $TQ$-relations for the various fundamental representations $V_{j,b}$ 
determine the other terms of $X(Y_{i,a})$ by induction on the weight. So $X(Y_{i,a})$ is unique. Then it determines again the terms of 
$X_{i,a} \Psib_{\nu(i),aq^{-h}} = 1 + \cdots$ by weight induction.

Now $E'_e(\Theta_{w_0}'(\Psib_{i,a}))$ is also a solution. Hence we obtain 
$$\chi_q\left(L_{\nu(i),aq^{-h}}^-\right) = E'_e\left(\Theta_{w_0}'(\Psib_{i,a})\right).$$
It follows from Lemma \ref{QQTp} that 
$$\Theta_{w_0}'(\Psib_{i,a}) = \QQ_{w_0(\varpi_i),a}.$$
Hence, applying $E'_e$, we obtain  
$$Q_{w_0(\varpi_i),a} = \chi_q\left(L_{\nu(i),aq^{-h}}^-\right).$$
This proves Conjecture~\ref{cjq} in this case. Moreover, by \cite[Theorem 5]{H} we know that $L_{i,a}^-$ is real, so Conjecture~\ref{main-conj} also holds.
\cqfd
\end{proof}

\begin{remark} 
{\rm 
 As a consequence, $\chi_{-\varpi_{\nu(i)}} = \chi\left(L_{\nu(i),aq^{-h}}^-\right)$. 
}
\end{remark}

\begin{example}
{\rm
Putting together Proposition~\ref{prop-s_i} and Proposition~\ref{prop-w_0}, we get a proof of 
Conjectures \ref{main-conj} and \ref{cjq} for all $Q$-variables in type $A_2$.
In this case, explicit expansions of the $Q$-variables are given in Example~\ref{Ex4-4}, and we have
\[
Q_{\omega_1,q^r} = \chi_q\left(L^+_{1,q^r}\right), \quad 
Q_{\omega_2 - \omega_1,q^r} = \chi_q\left(L(\Psib_{1,q^{r-2}}^{-1}\Psib_{2,q^{r-1}})\right),\quad 
Q_{-\omega_2,q^r} = \chi_q\left(L^-_{2,q^{r-3}}\right),
\]
\[
Q_{\omega_2,q^r} = \chi_q\left(L^+_{2,q^r}\right), \quad 
Q_{\omega_1 - \omega_2,q^r} = \chi_q\left(L(\Psib_{2,q^{r-2}}^{-1}\Psib_{1,q^{r-1}})\right),\quad 
Q_{-\omega_1,q^r} = \chi_q\left(L^-_{1,q^{r-3}}\right). 
\]
}
\end{example}

\subsection{Type $A_1$}

In this subsection we assume that $\mathfrak{g} = \mathfrak{sl}_2$ and we establish Conjecture \ref{main-conj} in this case.
Here $V = \{(1,2r)\mid r\in\mathbb{Z}\}.$
For readability, we will drop the unique index $i = 1$ of the Dynkin diagram in this section, and 
we will identify $V$ with $2\Z$.

The $Q$-variable
$Q_{\varpi,q^{2r}}$ (\resp $Q_{-\varpi,q^{2r}}$) coincides with the $q$-character of the positive (\resp negative) prefundamental representation $L^+_{q^{2r}}$ (\resp $L^-_{q^{2r-2}}$). Indeed, we have explicit $q$-character formulas:
\begin{eqnarray*}
&&\chi_q(L^+_{q^{2r}})=\chi_q(L(\Psib_{q^{2r}})) = \Psib_{q^{2r}} = Q_{\varpi,q^{2r}},\\[2mm]
&&\chi_q(L^-_{q^{2r-2}})=\chi_q(L(\Psib_{q^{2r-2}}^{-1})) = \Psib_{q^{2r-2}}^{-1}
(1 + A_{q^{2r-2}}^{-1}(1 + A_{q^{2r-4}}^{-1}(1 + \cdots))\cdots) = Q_{-\varpi,q^{2r}},
\end{eqnarray*}
(recall that, in type $A_1$, we have
$A_{q^{2r}} = Y_{q^{2r+1}}Y_{q^{2r - 1}}$,  where $Y_{q^s} = [\varpi]\Psib_{q^{s - 1}}\Psib_{q^{s + 1}}^{-1}$.)

By \cite[Corollary 5.6]{H}, any simple module in $\mathcal{O}_{\mathbb{Z}}$ is a quotient of a fusion product of positive and negative prefundamental representations and of an invertible representation $[\Psib_{\lambda}]$, $\lambda\in P$.
The category $\mathcal{O}_\mathbb{Z}$ contains also the finite-dimensional representations
$L(\Psib_{q^{2r}}\Psib_{q^{2s}}^{-1})$ for $r\leq s\in \Z$. Their $q$-character is equal to:
$$\chi_q(L(\Psib_{q^{2r}}\Psib_{q^{2s}}^{-1})) =
\Psib_{q^{2r}}\Psib_{q^{2s}}^{-1}
(1 + A_{q^{2s}}^{-1}(1 + A_{q^{2s-2}}^{-1}(1 + \cdots + A_{q^{2r+4}}^{-1}(1+ A_{q^{2r+2}}^{-1}))\cdots)$$
$$= \sum_{t = s, s-1 , \cdots, r} [(t-s)\alpha]\,\Psib_{q^{2(s+1)}}\Psib_{q^{2(1 + t)}}^{-1}\Psib_{q^{2t}}^{-1}\Psib_{q^{2r}}.$$
Indeed,
$[(s - r)\omega_1]\otimes L(\Psib_{q^{2r}}\Psib_{q^{2s}}^{-1})\simeq L(Y_{q^{2r+1}}Y_{q^{2r+3}}\cdots Y_{q^{2s - 1}})$
is a representation of the ordinary quantum affine algebra $U_q(\hatsl_2)$. Its $q$-character is known as it is obtained as an evaluation of a simple representation of $U_q(\mathfrak{sl}_2)$.

Let us recall that a simple representation $V$ is said to be prime if it can not be factorized as a
tensor product of two non-invertible representations.

\begin{Prop}\label{prep} The simple representations $L^+_{1,q^{2r}}$, $L_{1,q^{2r}}^-$, $L(\Psib_{q^{2r}}\Psib_{q^{2s}}^{-1})\ (r\le s \in \Z)$ are prime objects of $\O_\Z$. 
\end{Prop}

\begin{proof} The representation $L^+_{1,q^{2r}}$ is prime. Indeed, it is simple since it is one-dimensional. Also, it is not invertible because the only simple representation $V$ such that the highest $\ell$-weight of $L_{1,q^{2r}}^+ * V$ is $1$ is $V = L_{1,q^{2r}}^-$ which is infinite-dimensional. Now, assume that $L_{1,q^{2r}}^+\simeq V * V'$. Then $V$ and $V'$ are one-dimensional and simple : $V = L(\Psib)$ and $V'  = L(\Psib')$. If we assume in addition that $V$ and $V'$ are not invertible, then their highest $\ell$-weights are not constant. Now if a simple representation $L(\Psib)$ is one-dimensional, then $\Psib = \Psib(z)$ is a polynomial. Indeed, it follows from the relations of the shifted quantum affine algebras that  the action of $\phi^+(z) - \phi^-(z)$ is zero on the representation, and so the coefficient of $z^s$ in $\Psib(z)$ is zero for $s$ large enough. Then $\Psib\Psib'$ is a product of non-constant polynomials which is of degree $1$, and we get a contradiction.

Consider now a simple representation in $\mathcal{O}_{\mathbb{Z}}$. It is a highest weight module with a highest weight $\lambda\in P$. If this representation is not one-dimensional, then the dimension of its weight space of weight $\lambda - \alpha$ is at least $1$. This implies that if the dimension of this weight space is actually equal to $1$, then this representation is prime. Consequently, from the explicit $q$-character formulas above, the simple representations $L_{1,q^{2r}}^-$, $L(\Psib_{q^{2r}}\Psib_{q^{2s}}^{-1})$ are also prime. 
\cqfd
\end{proof}

Mimicking the notations for prime finite-dimensional representations of the ordinary quantum affine algebra $U_q(\hatsl_2)$, we will denote these prime representations by segments:
\[
[r ,\, s-1] = L(\Psib_{q^{2r}}\Psib_{q^{2s}}^{-1}),\quad 
[-\infty,\,s-1] = L(\Psib_{q^{2s}}^{-1}),\quad 
[r,\,+\infty] = L(\Psib_{q^{2r}}).
\]
We also set $[-\infty,+\infty] = 1$. Then $[r,s]$ is defined for any $r\leq s \in \Z\cup \{\pm\infty\}$.

The following identities are easy consequences of the above $q$-character formulas.
(Here we identify a simple representation with its class in $K_0(\mathcal{O}_{\mathbb{Z}})$.)

\begin{Prop}\label{prop-rel-K}
Let $r\leq s$ and $r'\leq s'$. Suppose $r < r'$ and $s < s'$ and $r'\leq s + 1$ (in particular, $r'$ and $s$ are finite). 
In $K_0(\mathcal{O}_{\mathbb{Z}})$ we have:
\begin{equation}\label{dco}[r,\,s] [r',\,s'] = [r,\,s'] [r',\,s] + [2(r'-s-2)\varpi] [r,\,r'-2][s+2,s'],
\end{equation}
where if $a>b$, we understand $[a,b]= 1$.
\cqfd
\end{Prop}

Proposition~\ref{prop-rel-K} recovers several classical families of relations :
\begin{itemize}
\item $T$-system relations : for $r' - r = s' - s = 1$,

\item Baxter $TQ$-relations : for $r = s = r' - 1 $, $s' = +\infty$,

\item $QQ$-relations : for $r = - \infty$, $s' = +\infty$, $r' = s +1$.
\end{itemize}
The next proposition generalizes a classical result of Chari and Pressley \cite{CP} for $U_q(\hatsl_2)$.

\begin{Prop}\label{simpcond} Let $r\leq s$ and $r'\leq s'$. The product $[r,\,s]\cdot [r',\,s']$ is a simple class in $K_0(\mathcal{O}_{\mathbb{Z}})$
 if and only if the union of the two intervals $[r,\,s] \cup [r',\,s']$ is not an interval containing properly both intervals $[r,\,s]$ and $[r',\,s']$.
\end{Prop}

\begin{proof} It follows from Equation (\ref{dco}) that if the conditions on $r,s,r',s'$ are not satisfied, then the product is not a simple class.

Now let us suppose that the conditions are satisfied. We will prove that the product $[r,\,s]\cdot[r',\,s']$ is a simple class.

For finite $s,s'$ the representations $[r,s]$ and $[r',s']$ are
representations of a shifted quantum affine algebra $U_q^\mu(\hatsl_2)$ with $\mu\le 0$.
Moreover, by \cite{H}, these representations are simple when restricted to the Borel algebra $U_q(\widehat{\mathfrak{b}})$. The
tensor product of these representations seen as $U_q(\widehat{\mathfrak{b}})$-module is simple 
(it follows from \cite{CP} for the finite-dimensional ones, and from \cite{HL2} in general).
At the level of representations of shifted quantum affine algebras, it implies that the simple quotient of the fusion product of $[r,s]$ by $[r',s']$ has weight spaces at least as large as in the fusion product. So this fusion product is simple and the product of the classes in $K_0(\mathcal{O}_{\mathbb{Z}})$ corresponds to a simple class.

For $s = s' = +\infty$, the representations are one-dimensional, as well as their fusion product, so the product is a simple class.

It remains to study the product $[r,\,s]\cdot[r',\,+\infty]$
with $s < + \infty$ and ($r'\leq r$ or $s \leq r' - 2$). Since $[r',+\infty]$ is one-dimensional, one can construct directly a structure of representation on the tensor product space $[r,s]\otimes [r',+\infty]$ by using the
Drinfeld coproduct (without using forms, specializations and fusion product). This is explained in \cite[Remark 5.8]{H}. For example, the action of
$$x^+(z) = \sum_{m\in\mathbb{Z}} x_m^+ z^m$$
on the tensor product coincides with the action of $x^+(z)\otimes \mathrm{Id}$.
So the tensor product is cocyclic (every non-zero submodule contains the highest weight vector). The action of
$$x^-(z) = \sum_{m\in\mathbb{Z}} x_m^- z^m$$
on the tensor product coincides with the action of
$$(1-z q^{2r'})x^-(z)\otimes \mathrm{Id}.$$
As explained above, $[r,\,s]$ can be obtained from an evaluation module of the ordinary quantum affine algebra and so the action of $x^-(z)$ is of the form 
$$\left(\sum_{m\in\mathbb{Z}} \alpha^m z^m\right) x_0$$
for an operator $x_0$ and $\alpha = q^{2A}$ with $r < A \leq s+1$. In particular $\alpha \neq q^{2r'}$. Then
$$(1-z q^{2r'})\cdot\left(\sum_{m\in\mathbb{Z}} \alpha^m z^m\right) = (1 - \alpha^{-1}q^{2r'})\left(\sum_{m\in\mathbb{Z}} \alpha^m z^m\right)\neq 0$$
is non zero. As moreover $[r,s]$ is generated by a highest weight vector
under the action of $x_0$, we obtain that the tensor product is also cyclic (generated by a highest weight vector). A cocyclic representation which is cyclic is simple, hence the result.
\cqfd 
\end{proof}

\begin{Prop} \label{prop-facto}
Consider a product $\prod_{1\le i\le N}[r_i,\,s_i]$ in $K_0(\mathcal{O}_{\mathbb{Z}})$ such that for every $i< j$, the classes $[r_i,\,s_i]$ and $[r_j,\,s_j]$ satisfy the condition of Proposition \ref{simpcond}. Then this product is a simple class.
\end{Prop}

\begin{proof} If $s_i < +\infty$ for all $i$, then we are reduced to consider representations of the Borel algebra $U_q(\hat{\mathfrak{b}})$ as in the previous proof. For such representations the result was established in \cite{HL2}. Otherwise, one can assume $s_i < +\infty$ for $1\leq i\leq M$ and
$s_i = +\infty$ for $M+1\leq i\leq N$. Then the product is of the form
$P_1P_2$ where $P_1$ is the simple product of the $M$ first factors, and $P_2$ is one dimensional, product of the
classes of $N-M$ positive prefundamental representations. Then an argument analog to the last argument in the proof of \ref{simpcond}
allows to conclude. \cqfd
\end{proof}

\begin{Cor} \label{cor-simpl-fact}
The prime representations of Proposition \ref{prep} are all the prime representations of the category $\mathcal{O}_{\mathbb{Z}}$
(up to invertible representations $L(\Psib_\lambda)$). A simple class in $K_0(\mathcal{O}_{\mathbb{Z}})$ admits a unique factorization
as a product of prime representations (up to ordering and to invertible representations). 
\end{Cor}

\begin{proof}
The highest $\ell$-weight of a simple representation of $\mathcal{O}_{\mathbb{Z}}$ can be factorized in a unique way (up to ordering and a constant $\ell$-weight factor)  as a product of highest $\ell$-weights of representations $[r,s]$ such that the conditions of Proposition~\ref{prop-facto} are satisfied. This implies that the class of this simple representation admits a factorization into a (unique) product of classes of prime representations listed in Proposition \ref{prep}. 
\cqfd
\end{proof}

\begin{remark}\label{Rem-sh-Borel} 
{\rm
(i)\  Corollary~\ref{cor-simpl-fact} implies that all simple representations in $\mathcal{O}_{\mathbb{Z}}$ are real.
Note that this is specific to $\mathfrak{g} = \mathfrak{sl}_2$.

\smallskip
(ii)\  In the Grothendieck ring of the category $\mathcal{O}$ 
of $U_q(\widehat{\mathfrak{b}})$, factorization into primes is not unique in general.
Following an example of \cite{MY}, the following factorization is discussed in \cite[Remark 3.10]{HL2} (we use the notation $L^{\mathfrak{b}}$ for representations of $U_q(\widehat{\mathfrak{b}})$):
$$L^{\mathfrak{b}}(\Psib_{q^4}\Psib_{q^{-6}}^{-1})\otimes L^{\mathfrak{b}}(\Psib_{q^8}\Psib_{q^{-10}}^{-1})
\simeq L^{\mathfrak{b}}(\Psib_{q^4}\Psib_{q^{-10}}^{-1}) \otimes L^{\mathfrak{b}}(\Psib_{q^8}\Psib_{q^{-6}}^{-1})
\simeq L^{\mathfrak{b}}(\Psib_{q^4}\Psib_{q^8}\Psib_{q^{-6}}^{-1}\Psib_{q^{-10}}^{-1}),$$
where the $4$ representations occurring in the $2$ factorizations are prime as representations of $U_q(\widehat{\mathfrak{b}})$.
But, in the Grothendieck ring of the category $\mathcal{O}^{\mathrm{sh}}$, 
for $r > s \in \Z$, the class
$$[L(\Psib_{q^{2r}}\Psib_{q^{2s}}^{-1})] = [L(\Psib_{q^{2r}})][ L(\Psib_{q^{2s}}^{- 1})]$$
is not prime. Hence 
the corresponding representations in the factorization above are not prime. The
unique factorization into prime representations is :
$$[L(\Psib_{q^4}\Psib_{q^8}\Psib_{q^{-6}}^{-1}\Psib_{q^{-10}}^{-1})] = [L(\Psib_{q^4})][ L(\Psib_{q^8})][ L(\Psib_{q^{-6}}^{-1})][L(\Psib_{q^{-10}}^{-1})].$$\cqfd
}
\end{remark}

The main result of this subsection is the following.

\begin{Thm} \label{main-Thm-sl2}
Conjecture \ref{main-conj} is true for $\mathfrak{g} = \mathfrak{sl}_2$.
\end{Thm}

Before proving Theorem~\ref{main-Thm-sl2}, we need to give a more detailed description of 
the cluster algebra $\AA_{w_0}$ in type $A_1$.
The quivers $\Gamma_{w_0,2r}$ $(r\in\mathbb{Z})$ of the standard initial seeds of $\AA_{w_0}$ 
were discussed in detail in Section~\ref{sect-3.1}. These seeds can be regarded as limits when $N \to \infty$ of certain acyclic seeds of cluster algebras $\AA_N$ of finite type $A_N$. It is well known that the clusters of $\AA_N$ are in natural bijection with triangulations of a convex $(N+3)$-gon. This leads us to a convenient geometric model 
for clusters and cluster variables of $\AA_{w_0}$ in terms of triangulations of an $\infty$-gon, which we shall now explain. 

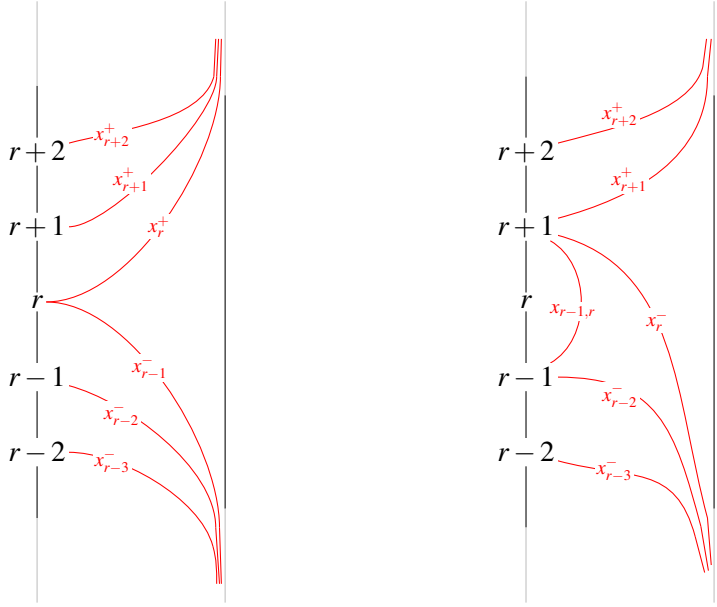
\begin{figure}
\begin{center}
\begin{tikzpicture}
	\begin{pgfonlayer}{nodelayer}
		\node [style=smat] (0) at (4, 4) {};
		\node [style=smat] (1) at (4, 2) {};
		\node [style=smat] (2) at (4, 0) {};
		\node [style=smat] (3) at (4, -2) {};
		\node [style=smat] (4) at (4, -4) {};
		\node [style=none] (5) at (4, 6) {};
		\node [style=none] (6) at (4, -6) {};
		\node [style=bcoord2] (9) at (4, 4) {$r+2$};
		\node [style=none] (10) at (4, 2) {};
		\node [style=bcoord2] (11) at (4, 2) {$r+1$};
		\node [style=bcoord2] (12) at (4, 0) {$r$};
		\node [style=bcoord2] (13) at (4, -2) {$r-1$};
		\node [style=bcoord2] (14) at (4, -4) {$r-2$};
		\node [style=none] (15) at (9, 5.5) {};
		\node [style=none] (16) at (9, -5.5) {};
		\node [style=rcoord] (17) at (6.5, 5) {$x_{r+2}^+$};
		\node [style=rcoord] (18) at (6.75, 3.25) {$x_{r+1}^+$};
		\node [style=rcoord] (19) at (7.5, -0.5) {$x_{r}^-$};
		\node [style=rcoord] (20) at (5.25, -0.25) {$x_{r-1,r}$};
		\node [style=rcoord] (21) at (6.5, -2.5) {$x_{r-2}^-$};
		\node [style=rcoord] (22) at (6.35, -4.45) {$x_{r-3}^-$};
		\node [style=none] (23) at (-9, 5.75) {};
		\node [style=bcoord2] (24) at (-9, 4) {$r+2$};
		\node [style=bcoord2] (25) at (-9, 2) {$r+1$};
		\node [style=bcoord2] (26) at (-9, 0) {$r$};
		\node [style=bcoord2] (27) at (-9, -2) {$r-1$};
		\node [style=bcoord2] (28) at (-9, -4) {$r-2$};
		\node [style=none] (29) at (-9, -5.75) {};
		\node [style=none] (30) at (-4, -5.5) {};
		\node [style=none] (31) at (-4, 5.5) {};
		\node [style=none] (32) at (-4.15, -6) {};
		\node [style=none] (33) at (-4.25, -6) {};
		\node [style=none] (34) at (-4.325, -6.5) {};
		\node [style=none] (35) at (-4.125, 6) {};
		\node [style=none] (36) at (-4.225, 6) {};
		\node [style=none] (37) at (-4.3, 6) {};
		\node [style=rcoord] (38) at (-7, 4.425) {$x_{r+2}^+$};
		\node [style=rcoord] (39) at (-6.5, 3.25) {$x_{r+1}^+$};
		\node [style=rcoord] (40) at (-5.75, 2) {$x_r^+$};
		\node [style=rcoord] (41) at (-6, -1.75) {$x_{r-1}^-$};
		\node [style=rcoord] (42) at (-6.75, -3) {$x_{r-2}^-$};
		\node [style=rcoord] (43) at (-7, -4.25) {$x_{r-3}^-$};
		\node [style=none] (44) at (-9, 8) {};
		\node [style=none] (45) at (-4, 8) {};
		\node [style=none] (46) at (-9, -8) {};
		\node [style=none] (47) at (-4, -8) {};
		\node [style=none] (48) at (-4.1, 7) {};
		\node [style=none] (49) at (-4.25, 7) {};
		\node [style=none] (50) at (-4.175, 7) {};
		\node [style=none] (51) at (-4.225, -7.5) {};
		\node [style=none] (52) at (-4.15, -7.5) {};
		\node [style=none] (53) at (-4.1, -7.5) {};
		\node [style=none] (54) at (4, -8) {};
		\node [style=none] (55) at (9, -8) {};
		\node [style=none] (56) at (4, 8) {};
		\node [style=none] (57) at (9, 8) {};
		\node [style=none] (58) at (8.825, 6) {};
		\node [style=none] (59) at (8.7, 6.25) {};
		\node [style=none] (60) at (8.925, 7) {};
		\node [style=none] (61) at (8.8, 7) {};
		\node [style=none] (62) at (8.825, -5.75) {};
		\node [style=none] (63) at (8.65, -5.975) {};
		\node [style=none] (64) at (8.5, -6.25) {};
		\node [style=none] (65) at (8.925, -7) {};
		\node [style=none] (66) at (8.85, -7.15) {};
		\node [style=none] (67) at (8.75, -7.2) {};
	\end{pgfonlayer}
	\begin{pgfonlayer}{edgelayer}
		\draw (5.center) to (9);
		\draw (9) to (11);
		\draw (11) to (12);
		\draw (12) to (13);
		\draw (13) to (14);
		\draw (14) to (6.center);
		\draw [style=sred, in=30, out=-30] (11) to (13);
		\draw (15.center) to (16.center);
		\draw (23.center) to (24);
		\draw (24) to (25);
		\draw (25) to (26);
		\draw (26) to (27);
		\draw (27) to (28);
		\draw (28) to (29.center);
		\draw (31.center) to (30.center);
		\draw [style=sred, in=90, out=0, looseness=0.75] (26) to (32.center);
		\draw [style=sred, in=90, out=-15, looseness=0.75] (27) to (33.center);
		\draw [style=sred, in=105, out=0, looseness=0.75] (28) to (34.center);
		\draw [style=sred, in=-90, out=0, looseness=0.75] (26) to (35.center);
		\draw [style=sred, in=270, out=0, looseness=0.50] (25) to (36.center);
		\draw [style=sred, in=255, out=15, looseness=0.75] (24) to (37.center);
		\draw [style=dotted] (30.center) to (47.center);
		\draw [style=dotted] (29.center) to (46.center);
		\draw [style=dotted] (23.center) to (44.center);
		\draw [style=dotted] (31.center) to (45.center);
		\draw [style=sred] (48.center) to (35.center);
		\draw [style=sred] (37.center) to (49.center);
		\draw [style=sred] (50.center) to (36.center);
		\draw [style=sred, in=90, out=-75] (34.center) to (51.center);
		\draw [style=sred] (33.center) to (52.center);
		\draw [style=sred] (32.center) to (53.center);
		\draw [style=dotted] (6.center) to (54.center);
		\draw [style=dotted] (16.center) to (55.center);
		\draw [style=dotted] (15.center) to (57.center);
		\draw [style=dotted] (5.center) to (56.center);
		\draw [style=sred, in=270, out=15] (11) to (58.center);
		\draw [style=sred, in=255, out=15, looseness=0.75] (9) to (59.center);
		\draw [style=sred] (58.center) to (60.center);
		\draw [style=sred] (59.center) to (61.center);
		\draw [style=sred, in=105, out=-15] (11) to (62.center);
		\draw [style=sred, in=105, out=0, looseness=1.25] (13) to (63.center);
		\draw [style=sred, in=105, out=-15, looseness=1.25] (14) to (64.center);
		\draw [style=sred] (62.center) to (65.center);
		\draw [style=sred] (63.center) to (66.center);
		\draw [style=sred] (64.center) to (67.center);
	\end{pgfonlayer}
\end{tikzpicture}
\end{center}
\caption{\label{Fig-infty-gon} {\it Two triangulations of the $\infty$-gon.}}
\end{figure}

Let $\P$ denote an $\infty$-gon, with vertex set $\VV$ and edge set $\E$ given by:
\[
\VV := \Z\cup\{-\infty, +\infty\},\qquad
\E := \{(r,r+1) \mid r\in \Z\}\cup \{(-\infty,+\infty)\}.
\]
Figure~\ref{Fig-infty-gon} represents two triangulations of $\P$, with edges in black and inner diagonals in red. In the left one, all finite vertices $v\ge r$
(\resp $v\le r$) are connected by a diagonal to vertex $+\infty$ (\resp $-\infty$). 
This triangulation corresponds to the standard initial seed of $\AA_{w_0}$ with quiver $\G_{w_0,2r}$.  
More precisely, the diagonal $(v,+\infty)$ encodes the cluster variable with stabilized $g$-vector $\be_{2v}$,
and the diagonal $(-\infty, v)$ encodes the cluster variable with stabilized $g$-vector $-\be_{2v-2}$.
There is an arrow between two vertices of $\G_{w_0,2r}$ if and only if the two corresponding diagonals form two sides of one triangle of the triangulation, and the direction of the arrow indicates if one passes from one diagonal to the other by a clockwise or an anti-clockwise rotation around their common vertex. Thus, choosing $r=0$ in the triangulation of Figure~\ref{Fig-infty-gon}, one gets the standard initial seed of $\AA_{w_0}$ whose quiver is the leftmost one in Figure~\ref{Fig7e}.

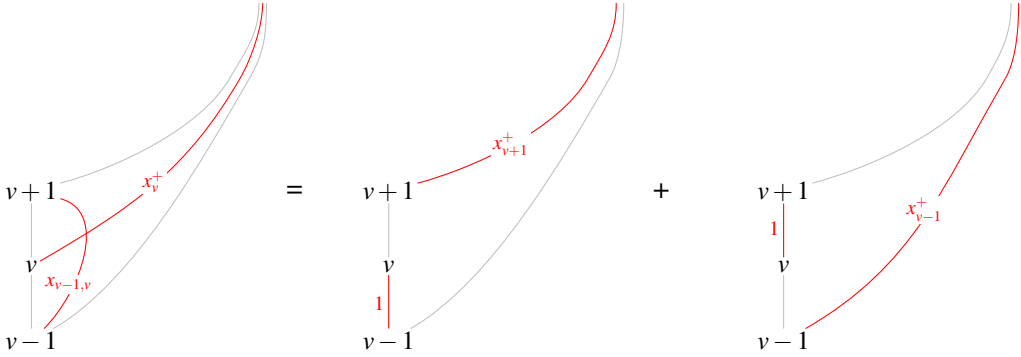
\begin{figure}
\begin{center}
\begin{tikzpicture}
	\begin{pgfonlayer}{nodelayer}
		\node [style=bco2] (0) at (-9, 0) {$v$};
		\node [style=bco2] (1) at (-9, 2) {$v+1$};
		\node [style=bco2] (2) at (-9, -2) {$v-1$};
		\node [style=none] (3) at (-3.725, 5) {};
		\node [style=none] (4) at (-3.45, 5) {};
		\node [style=none] (5) at (-3.2, 5) {};
		\node [style=rcoord] (6) at (-5.75, 2.25) {$x_v^+$};
		\node [style=rcoord] (7) at (-8, -0.5) {$x_{v-1,v}$};
		\node [style=none] (8) at (-2.95, 7) {};
		\node [style=none] (9) at (-2.85, 7) {};
		\node [style=none] (10) at (-2.75, 7) {};
		\node [style=bco2] (11) at (0.5, 0) {$v$};
		\node [style=bco2] (12) at (0.5, 2) {$v+1$};
		\node [style=bco2] (13) at (0.5, -2) {$v-1$};
		\node [style=none] (14) at (5.8, 5) {};
		\node [style=none] (15) at (6, 5) {};
		\node [style=none] (16) at (6.375, 5) {};
		\node [style=none] (18) at (6.55, 7) {};
		\node [style=none] (19) at (6.65, 7) {};
		\node [style=none] (20) at (6.75, 7) {};
		\node [style=bco2] (21) at (11, 0) {$v$};
		\node [style=bco2] (22) at (11, 2) {$v+1$};
		\node [style=bco2] (23) at (11, -2) {$v-1$};
		\node [style=none] (24) at (16.4, 5) {};
		\node [style=none] (25) at (16.5, 5) {};
		\node [style=none] (26) at (16.9, 5) {};
		\node [style=none] (28) at (17.05, 7) {};
		\node [style=none] (29) at (17.15, 7) {};
		\node [style=none] (30) at (17.25, 7) {};
		\node [style=rcoord] (31) at (3.75, 3.25) {$x_{v+1}^+$};
		\node [style=rcoord] (32) at (14.75, 1.5) {$x_{v-1}^+$};
		\node [style=none] (33) at (-2, 2) {=};
		\node [style=none] (34) at (7.75, 2) {+};
		\node [style=rcoord] (35) at (0.25, -1) {$1$};
		\node [style=rcoord] (36) at (10.75, 1) {$1$};
	\end{pgfonlayer}
	\begin{pgfonlayer}{edgelayer}
		\draw [style=dotted, in=240, out=15, looseness=0.75] (1) to (3.center);
		\draw [style=dotted] (1) to (0);
		\draw [style=dotted] (0) to (2);
		\draw [style=dotted, in=240, out=30, looseness=0.75] (2) to (5.center);
		\draw [style=sred, in=240, out=30] (0) to (4.center);
		\draw [style=sred, in=45, out=-15] (1) to (2);
		\draw [style=dotted, in=270, out=60] (3.center) to (8.center);
		\draw [style=dotted, in=270, out=60] (5.center) to (10.center);
		\draw [style=sred, in=270, out=60, looseness=0.75] (4.center) to (9.center);
		\draw [style=sred, in=240, out=15, looseness=0.75] (12) to (14.center);
		\draw [style=dotted] (12) to (11);
		\draw [style=sred] (11) to (13);
		\draw [style=dotted, in=240, out=30, looseness=0.75] (13) to (16.center);
		\draw [style=sred, in=270, out=60] (14.center) to (18.center);
		\draw [style=dotted, in=270, out=60, looseness=0.75] (16.center) to (20.center);
		\draw [style=dotted, in=240, out=15, looseness=0.75] (22) to (24.center);
		\draw [style=sred] (22) to (21);
		\draw [style=dotted] (21) to (23);
		\draw [style=sred, in=240, out=30] (23) to (26.center);
		\draw [style=dotted, in=270, out=60] (24.center) to (28.center);
		\draw [style=sred, in=270, out=60, looseness=0.75] (26.center) to (30.center);
	\end{pgfonlayer}
\end{tikzpicture}
\end{center}
\caption{\label{Fig-Ptolemy} {\it An example of Ptolemy relation.}}
\end{figure}

Let us denote by $x^+_v\ (v\in \Z)$ (\resp $x^-_{v-1}\ (v\in\Z)$) the cluster variable corresponding to the diagonal $(v,+\infty)$ (\resp $(-\infty,v)$). The remaining cluster variables $x_{r,s}$ of $\AA_{w_0}$ correspond to diagonals $(r,s+1)$ joining two finite vertices $r,s+1\in\Z$ such that $r<s$. For example, 
the right triangulation of Figure~\ref{Fig-infty-gon} contains the diagonal $(r-1,r+1)$ corresponding to the cluster variable $x_{r-1,r}$.

Using this notation, the initial exchange relations of the seed with quiver $\G_{w_0,2r}$ are:
\begin{eqnarray}
x^+_v x_{v-1, v} & = & x^+_{v+1} + x^+_{v-1},\quad (v>r), \label{eq-28}\\
x^+_r x^-_{r} & = & x^+_{r+1} + x^-_{r-1}, \label{eq-29}\\
x^-_{r-1} x^+_{r-1} & = & x^+_{r} + x^-_{r-2}, \label{eq-30}\\
x^-_v x_{v,v+1} & = & x^-_{v+1} + x^-_{v-1},\quad (v<r-1). \label{eq-31}
\end{eqnarray}
They can be visualized as Ptolemy relations associated with a flip of diagonals in a suitable quadrilateral. For instance Figure~\ref{Fig-Ptolemy} represents the exchange relation of Equation~(\ref{eq-28}). Note that the cluster algebra $\AA_{w_0}$ has no frozen variables, so the edges $(v-1,v)$ and $(v,v+1)$ of the $\infty$-gon count for $1$.

The clusters of $\AA_{w_0}$ are represented by those triangulations of $\P$ which are reachable from the triangulation corresponding to a standard initial seed by means of a \emph{finite} number of flips of diagonals. Two cluster variables are compatible (that is, may belong to the same cluster) if and only if the corresponding diagonals do not intersect in their interior. If two cluster variables are not compatible, then 
the corresponding arcs form the two diagonals of a quadrilateral inscribed in $\P$, and 
their product is the left-hand side of an exchange relation represented by the Ptolemy relation in this quadrilateral. Let $r<r'< s+2<s'+2\in \VV$, where $r$ (\resp $s'$) is allowed to be $-\infty$ (\resp $+\infty$). 
Then replacing cluster variables by the corresponding diagonals, the Ptolemy relation in the quadrilateral $(r,r',s+2,s'+2)$ reads:
\begin{equation}\label{eq-Pt}
(r,s+2)(r',s'+2) = (r,s'+2)(r',s+2) + (r,r')(s+2, s'+2). 
\end{equation}

\begin{remark}\label{structclus} 
{\rm
The idea of using triangulations of an $\infty$-gon as a combinatorial model for cluster algebras of type $A$ in the infinite rank limit is very natural, and it has already been studied by many authors \cite{IT,HoJo,GG,BG,PY,CKP}. These
papers develop generalizations of the classical cluster categories of type $A_N$ and study the correspondence between their cluster-tilting subcategories and triangulations of an $\infty$-gon. Since in our case every triangulation contains infinitely many diagonals joining a finite vertex to $+\infty$ or $-\infty$, the papers \cite{BG,PY,CKP} which also deal with this type of configurations seem to be the most relevant for us. It would be interesting to establish connections between the abelian category $\O_\Z$ and the (ex-)triangulated categories studied in these papers.  \cqfd   
}
\end{remark}

We can now explain the proof of Theorem~\ref{main-Thm-sl2}.

\begin{proof}
Recall the injective ring morphism $F : \AA_{w_0}\rightarrow K_{\mathbb{Z}}$ of Theorem \ref{main-Thm}.
We want to prove that every cluster monomial of $\AA_{w_0}$ is mapped by $F$ to the class of a real simple
module in $K_\Z$. By Proposition~\ref{prop-polynomial}, we already know that the renormalized $Q$-variables 
\[
\underline{Q}_{\varpi,q^{2r}}
= [-r\varpi] Q_{\varpi,q^{2r}} =
\chi_q\left([-r\varpi]L_{q^{2r}}^+\right),
\quad
\underline{Q}_{-\varpi,q^{2r}}
= [(r-1)\varpi] Q_{-\varpi,q^{2r}} =
\chi_q\left([(r-1)\varpi]L_{q^{2r-2}}^-\right)
\]
are images by $F$ of cluster variables. More precisely 
$\underline{Q}_{\varpi,q^{2r}} = F(x^+_r)$ and 
$\underline{Q}_{-\varpi,q^{2r}} = F(x^-_{r-1})$.

Now, a special case of Equation~(\ref{eq-Pt}) gives, for $r<s\in\Z$ :
\[
 (-\infty, s+1)(r,+\infty) = (r, s+1) + (-\infty, r) (s+1, +\infty),
\]
where we have used that $(-\infty,+\infty) = 1$. In other words, we have: 
\[
 x_{r,s} = x^-_sx^+_r - x^-_{r-1}x^+_{s+1}.
\]
Comparing with Equation~(\ref{dco}), we obtain:
\begin{eqnarray*}
F(x_{r,s}) &=& \chi_q\left([s\varpi]L^-_{q^{2s}}\right)\chi_q\left([-r\varpi]L^+_{q^{2r}}\right)
- \chi_q\left([(r-1)\varpi]L^+_{q^{2r-2}}\right)\chi_q\left([-(s+1)\varpi]L^+_{q^{2s+2}}\right)
\\
&=&\chi_q\left([(s-r)\varpi]\,L\left(\Psi_{q^{2r}}\Psi^{-1}_{q^{2s}}\right)\right).
\end{eqnarray*}
Hence we see that the image by $F$ of the set of cluster variables of $\AA_{w_0}$ is precisely
the set of classes of prime simple modules of $\O_\Z$ (up to invertible elements of $[\uP]$) described
in Proposition~\ref{prep}.

In more geometric terms, we see that $F$ maps the cluster variable associated with the diagonal $(r,s+2)$ 
of $\P$ to the prime simple module labelled by the segment $[r,s]$ (up to invertible elements). 
Finally it is easy to check that the condition of Proposition~\ref{prop-facto} on two segments $[r,s]$ and $[r',s']$ is equivalent to the condition that the two diagonals $(r,s+2)$ and $(r',s'+2)$ do not cross in their interior, that is, the compatibility condition of the corresponding cluster variables. Therefore,
by Corollary~\ref{cor-simpl-fact},
$F$ maps the set of cluster monomials of $\AA_{w_0}$ to the set of classes of simple modules in $\O_\Z$.
\cqfd
\end{proof}

\begin{remark}
{\rm
Let $\mathcal{O}_{\mathbb{Z}}^f$ be the abelian subcategory of $\mathcal{O}_\mathbb{Z}$ of objects of finite length. Its Grothendieck group $K_0^f\subset K_0$ is the subgroup of $K_0$ generated by (finite) linear combinations of simple classes. It follows from our results that
$ [\uP]\otimes_{[P]} K_0^f \simeq [\uP]\otimes \mathcal{A}_{w_0}$.
This implies that $K_0^f$ is stable under multiplication. The representation-theoretical interpretation of this property is that $\mathcal{O}_{\mathbb{Z}}^f$ is stable by fusion product. Indeed, it is established in \cite{H} that a truncated shifted quantum affine algebra has a finite number of classes of simple representations.
Moreover, in the case of $\mathfrak{g} = \mathfrak{sl}_2$, a prefundamental representation, after a twist by an invertible representation, descends to a truncation. Hence, one obtains as in \cite{HZ} that a fusion product of various prefundamental representations has finite length. As a simple representation is a subquotient of
a fusion product of prefundamental representations, we obtain that the finite length property is stable by fusion product.
We expect this holds for a general Lie algebra $\mathfrak{g}$.
}
\end{remark}


\section{Double Bruhat cells and quantum Wronskians} \label{sect-Bruhat-Wronskian}

\subsection{The open double Bruhat cell}\label{sec-9-1}

Let $G$ be a simple simply-connected complex algebraic group with Lie algebra $\g$.
Let $B$ and $B_{-}$ be a pair of opposite Borel subgroups of $G$. The double Bruhat cells of $G$, introduced by Fomin and Zelevinsky in \cite{FZ-Bruhat}, are the intersections of the strata of the two Bruhat decompositions of $G$ with respect to $B$ and $B_{-}$: 
\[
 G^{v,w} :=  BvB\ \cap\ B_{-}wB_{-},\qquad (v,w\in W).
\]
The complex dimension of the stratum $G^{v,w}$ is equal to $\ell(v)+\ell(w)+n$. In particular, the double Bruhat cell $G^{w_0,w_0}$ is the unique stratum of maximal dimension $2\ell(w_0) + n = \dim G$. It can be described explicitly as the open dense subset of $G$ defined by the non-vanishing of certain generalized minors:
\[
 G^{w_0,w_0} = \{g\in G\mid \Delta_{w_0(\varpi_i),\varpi_i}(g) \not = 0,\ 
 \Delta_{\varpi_i,w_0(\varpi_i)}(g) \not = 0,\ \mbox{for all } i\in I\}.
\]
It was shown in \cite{BFZ} that the coordinate ring $\C[G^{v,w}]$ of $G^{v,w}$ has the structure of a cluster algebra, with explicit initial seeds consisting of certain generalized minors.

\begin{example}\label{example-9-1}
{\rm
In type $A_2$, we have $G = SL(3,\C)$, and we can represent elements of $G$ by $3\times 3$ complex
matrices of determinant 1. For $A, B$, two subsets of $\{1,2,3\}$ of the same cardinality, we denote by $\Delta_{A,B}(g)$ the minor of $g$ with row set $A$ and column set $B$.
We have
\[
 G^{w_0,w_0} =
\left\{
g \in G \mid \Delta_{1,3}(g) \not = 0,\ \Delta_{12,23}(g) \not = 0,\ 
\Delta_{3,1}(g) \not = 0,\ \Delta_{23,12}(g) \not = 0 
\right\}.
\]
An initial seed of the cluster structure of $\C[G^{w_0,w_0}]$ is displayed in Figure~\ref{fig-16}. 
The four minors painted in blue are frozen variables. These are the minors which do not vanish on 
$G^{w_0,w_0}$. 
In this small rank case, the cluster structure has finite type $D_4$. There are 20 cluster variables (including the 4 frozen ones) and 50 clusters (see \cite[Example 2.11 and Example 2.18]{BFZ}).

For instance, mutating at vertex $\De_{12,12}$, we get the new cluster variable $\De_{23,23}$, as shown
by the exchange relation:
\[
\De_{12,12} \De_{23,23} = \De_{23,12} \De_{12,23} + \De_{2,2}. 
\]
(This minor identity follows from the classical Lewis Caroll identity: 
\[
\De_{12,12} \De_{23,23} - \De_{23,12} \De_{12,23} = \De_{2,2}\De_{123,123}, 
\]
if we take into account that the determinant function $\De_{123,123}$ is equal to 1 on $G$.)
\cqfd

\begin{figure}[t]
\[
\def\lablestyle{\scriptstyle}
\quad
\xymatrix@-1.0pc{
&\blue{\De_{3,1}}&
\\
&&\ar[ld] \blue{\Delta_{23,12}} &&
\\
&\ar[uu] \Delta_{2,1}\ar[d]&
\\
& \Delta_{2,2}\ar[rd]
\\
&&\ar[uuu] \Delta_{12,12}\ar[d] &&
\\
&& \ar[ld]\blue{\Delta_{12,23}}
\\
&\ar[uuu]\Delta_{1,2} \ar[d] 
\\
&\blue{\Delta_{1,3}}
}
\]
\caption{\label{fig-16} {\it An initial seed for $\C[G^{w_0,w_0}]$ in type $A_2$}}
\end{figure}
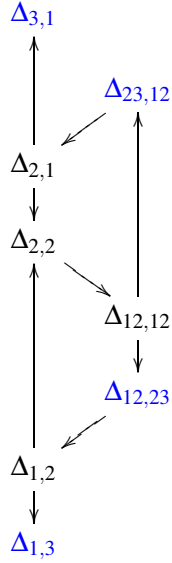

}
\end{example}

In \cite{BFZ} several initial seeds of the cluster algebra structure of $\C[G^{v,w}]$ were described, associated with pairs of reduced expressions of the Weyl group elements $v$ and $w$. In this paper we will only consider a subclass of these initial seeds parametrized by the Coxeter elements $c$, whose precise description we will now recall. 

Let $c$ be a fixed Coxeter element of $W$. For $i\in I$ we denote by $m_i$ the smallest integer $k$ such that
$c^k(\varpi_i) = w_0(\varpi_i) = -\varpi_{\nu(i)}$. In \S\ref{subsec434}, we have attached to $c$ a reduced expression $(i_1,i_2,\ldots, i_N)$ of $w_0$. The initial seed $\SS_c$ of \cite[\S2.2]{BFZ} corresponding to the reduced expression $(i_1,-i_1,i_2,-i_2,\ldots, i_N,-i_N)$ of $(w_0,w_0)\in W\times W$ can be described as follows. 

Recall the infinite quiver $\G_c$ defined in \S\ref{subsec-Coxeter}. The quiver of $\SS_c$ is the full 
subquiver $\gamma_c$ of $\G_c$ whose set of vertices consists of all green and red vertices together with the 
$n$ vertices immediately above the highest red vertices. These last $n$ vertices are in fact frozen vertices,
and we will paint them in blue. The lowest green vertex in each of the $n$ columns is also considered as frozen.
Hence the quiver $\gamma_c$ is an ice-quiver with $2\ell(w_0)+n = \dim G$ vertices, $2n$ of them being frozen.
Since the arrows connecting frozen vertices play no rôle, we will usually omit them.

The cluster variables of the seed $\SS_c$ are generalized minors, described as follows. For every $i\in I$, column $i$ of the quiver $\gamma_c$ contains the cluster variables:
\[
 \Delta_{\,c^k(\varpi_i),\ c^\ell(\varpi_i)}, \qquad (0\le k,\ell\le m_i,\ \ m_i -1 \le k+\ell\le m_i).
\]
More precisely, the top vertex corresponds to $\Delta_{\,c^{m_i}(\varpi_i),\ \varpi_i}$, and each vertical arrow
is of the form
\[
\def\lablestyle{\scriptstyle}
\xymatrix@-1.0pc{
\green{\Delta_{\,c^k(\varpi_i),\ c^{m_i-k}(\varpi_i)}} && \red{\Delta_{\,c^k(\varpi_i),\ c^{m_i-k-1}(\varpi_i)}}\ar[dd]
\\
&\qquad
\mbox{or}
\qquad
\\
\red{\Delta_{\,c^{k-1}(\varpi_i),\ c^{m_i-k}(\varpi_i)}}\ar[uu]&&\green{\Delta_{\,c^{k}(\varpi_i),\ c^{m_i-k}(\varpi_i)}}
}
\quad
\]

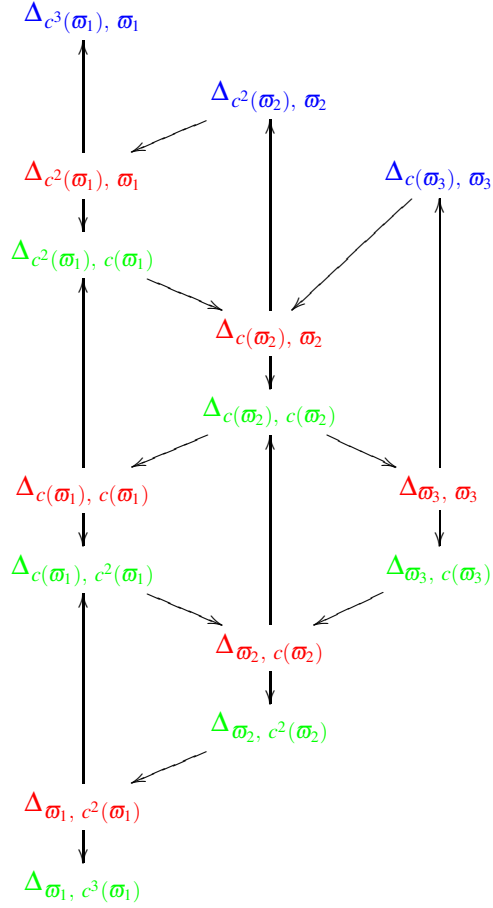
\begin{figure}[t]
\[
\def\lablestyle{\scriptstyle}
\xymatrix@-1.0pc{
\blue{\Delta_{\,c^3(\varpi_1),\ \varpi_1}}&
\\
&\blue{\Delta_{\,c^2(\varpi_2),\ \varpi_2}}\ar[ld] &&
\\
{\ar[uu]}\red{\Delta_{\,c^2(\varpi_1),\ \varpi_1}}\ar[d]&&\ar[ldd] \blue{\Delta_{\,c(\varpi_3),\ \varpi_3}} 
\\
{\green\Delta_{\,c^2(\varpi_1),\ c(\varpi_1)}}\ar[rd]
\\
&\ar[uuu] \red{\Delta_{\,c(\varpi_2),\ \varpi_2}}\ar[d]&&
\\
&\ar[ld] \green{\Delta_{\,c(\varpi_2),\ c(\varpi_2)} } \ar[rd]&&
\\
\ar[uuu]\red{\Delta_{\,c(\varpi_1),\ c(\varpi_1)} } \ar[d] && \ar[d] \red{\Delta_{\varpi_3,\ \varpi_3} }\ar[uuuu]
\\
\green{\Delta_{\,c(\varpi_1),\ c^2(\varpi_1)} } \ar[rd] &&\ar[ld] 
\green{\Delta_{\varpi_3,\ c(\varpi_3)} } 
\\
& \ar[uuu]\red{\Delta_{\,\varpi_2,\ c(\varpi_2)} } \ar[d] &&
\\
&\ar[ld]\green{\Delta_{\,\varpi_2,\ c^2(\varpi_2)}} &&
\\
\ar[uuu]\red{\Delta_{\,\varpi_1,\ c^2(\varpi_1)} } \ar[d]  
\\
\green{\Delta_{\,\varpi_1,\ c^3(\varpi_1)} } &&
}
\]
\caption{\label{fig-17} {\it An initial seed of $\C[G^{w_0,w_0}]$ in type $A_3$.}}
\end{figure}

\begin{example}\label{exa9-2}
{\rm
In type $A_3$, we choose $c= s_1s_2s_3$. Then $m_1 = 3$, $m_2 = 2$ and $m_1 = 1$. The initial seed $\SS_c$ of
$\C[G^{w_0,w_0}]$ is displayed in Figure~\ref{fig-17}. 
\cqfd
}
\end{example}

\subsection{Relations between $\C[G^{w_0,w_0}]$ and $K_\Z$}

The following proposition follows immediately from the above discussion.

\begin{Prop}\label{prop-9-3}
Let $c$ be a Coxeter element, and let $\BB_c$ be the cluster subalgebra of $\AA_{w_0}$ with initial seed 
given by the subquiver $\gamma_c$ (in which the $2n$ vertices lying at the top rim and the bottom rim are frozen). Then $\C\otimes_\Z\BB_c$ is isomorphic to the Berenstein-Fomin-Zelevinsky cluster algebra $\C[G^{w_0,w_0}]$. \cqfd
\end{Prop}

Using Theorem~\ref{main-Thm}, we deduce from Proposition~\ref{prop-9-3} that the ring $\C\otimes_\Z K_\Z$ contains 
infinitely many cluster subalgebras isomorphic to 
the cluster algebra $\C[G^{w_0,w_0}]$. Moreover, in these isomorphisms, the $Q$-variables of the initial seeds of $K_\Z$ get identified with certain generalized minors on $G$. It turns out that the $QQ$-system relations
of Proposition~\ref{Prop-renorm-QQ}, which can be regarded as certain initial exchange relations for the cluster algebra structure of Theorem~\ref{main-Thm}, correspond in this identification to some well-known algebraic identities between generalized minors discovered by Fomin and Zelevinsky \cite[Theorem 1.17]{FZ-Bruhat}.

\begin{example}
{\rm
We continue Example~\ref{exa9-2}. Comparing the seeds of Figure~\ref{Fig9} and Figure~\ref{fig-17}, we see 
that the $QQ$-system relation
\[
\bQ_{s_1s_2(\varpi_2), q^{-1}} \bQ_{\varpi_2, q^{-3}} =
\bQ_{s_1s_2(\varpi_2), q^{-3}} \bQ_{\varpi_2, q^{-1}}
+\ \bQ_{s_1(\varpi_1), q^{-2}} \bQ_{\varpi_3, q^{-2}},
\]
which is the exchange relation corresponding to the mutation of the cluster variable $\bQ_{s_1s_2(\varpi_2), q^{-1}}$, translates into the generalized minor identity
\[
\Delta_{c(\varpi_2), c(\varpi_2)}\Delta_{\varpi_2, \varpi_2} =
\Delta_{c(\varpi_2), \varpi_2}\Delta_{\varpi_2, c(\varpi_2)}
+\ \Delta_{c(\varpi_1), c(\varpi_1)}\Delta_{\varpi_3, \varpi_3},
\]
which is \cite[Theorem 1.17]{FZ-Bruhat} with $i=2$ and $u=v=s_1$. \cqfd
}
\end{example}

\begin{remark}
{\rm
(a)\ Similarly, it is easy to check that the algebraic relation between $Q$-variables proved in
Proposition~\ref{prop-5.5} corresponds precisely to the generalized minor identity of \cite[Theorem 1.16]{FZ-Bruhat}.

\smallskip
(b)\ By definition, the Weyl group $W$ is a quotient of the normalizer $N$ of a maximal torus
in~$G$. Lifting the elements of $W$ to $N$, and then letting them act by right or left translations
on $G$, one gets induced left and right actions of a finite covering $\overline{W}$ of $W$ on $G$, and therefore also on $\C[G]$. 
In these actions, the lifts of the generators $s_i$ of $W$ to $\overline{W}$ act on generalized minors by:
\[
 \Delta_{u(\varpi_i),\,v(\varpi_i)} \mapsto \Delta_{u(\varpi_i),\,s_iv(\varpi_i)},\qquad
 \Delta_{u(\varpi_i),\,v(\varpi_i)} \mapsto \Delta_{s_iu(\varpi_i),\,v(\varpi_i)},\qquad
 (u,v \in W),
\]
if $l(s_iu) = l(u)+1$ and $l(s_iv) = l(v)+1$. 
Now, as explained above, we can embed $\C[G]$ in $\C\otimes_\Z K_\Z$, in such a way that
the generalized minors $\Delta_{c^k(\varpi_i),\,c^\ell(\varpi_i)}$ of the initial seed $\SS_c$ get mapped to $Q$-variables of the form $\bQ_{c^\ell(\varpi_i),\, q^r}$, where $r$ is an integer depending linearly on $k$. This suggests that the left action of $\overline{W}$ 
on $\C[G]$ is related by this embedding with the action introduced in \S\ref{subsec-finite-cov} of a finite covering of $W$ on $\Pi'_\Z$ (see Remark~\ref{rem-7-3}). As to the right action of $\overline{W}$ 
on $\C[G]$, its counterpart in $K_\Z$ seems more hidden. In the next sections we will look at it from a different angle,
using the idea of quantum Wronskian borrowed from \cite{KSZ,KZ}. \cqfd
}
\end{remark}

\subsection{Quantum Wronskian matrices}

In the classical theory of ordinary linear differential equations, the Wronskian matrix of 
$n$ smooth functions $f_1(z), \ldots, f_n(z)$ is by definition
\[
 W(f_1,\ldots,f_n) := 
 \pmatrix{
 f_1(z)& \cdots & f_n(z)\cr
 f_1'(z)&\cdots & f_n'(z)\cr
 \vdots&     & \vdots\cr
 f_1^{(n-1)}(z)& \cdots & f_n^{(n-1)}(z)
 },
\]
where $f^{(k)}(z)$ denotes the $k$th derivative. Given a fixed number $q$, we can replace
derivatives by $q$-differences and define by analogy the \emph{quantum Wronskian matrix}
\[
 qW(f_1,\ldots,f_n) := 
 \pmatrix{
 f_1(z)& \cdots & f_n(z)\cr
 f_1(q^2z)&\cdots & f_n(q^2z)\cr
 \vdots&     & \vdots\cr
 f_1(q^{2(n-1)}z)& \cdots & f_n(q^{2(n-1)}z)
 }.
\]
It turns out that such matrices arise naturally in our topic.
Indeed, as explained in \cite{KSZ}, given a solution of the 
$QQ$-system relations in type $A_n$, one can form a $q$-Wronskian matrix of size $n+1$ such that every $Q$-variable indexed
by an element of the form $w(\varpi_i)$ occurs as a minor of size $i$ of this matrix. Moreover, this matrix has determinant 1. Let us illustrate this using the solution $\bQ_{w(\varpi_i),a}\in K_\Z$ of Proposition~\ref{Prop-renorm-QQ}.

\begin{example}\label{ex-qw-A1}
{\rm
Consider the commutative ring $K_\Z$ in type $A_1$. The following matrix in $M(2,K_\Z)$:
\[
 g(q^{r}) := 
 \pmatrix{\bQ_{\varpi_1,\, q^r}& \bQ_{s_1(\varpi_1),\, q^r}\cr
 \bQ_{\varpi_1,\, q^{r+2}}& \bQ_{s_1(\varpi_1),\, q^{r+2}}
 },
 \qquad (r\in 2\Z),
\]
can be regarded as the quantum Wronskian matrix of the two functions $z \mapsto \bQ_{\varpi_1,\,z}$ and
$z \mapsto \bQ_{s_1(\varpi_1),\,z}$, evaluated at $z = q^{r}$.
The determinant of this matrix is
\[
 \det(g(q^r)) = \bQ_{\varpi_1,\, q^r}\bQ_{s_1(\varpi_1),\, q^{r+2}} - 
 \bQ_{s_1(\varpi_1),\, q^r}\bQ_{\varpi_1,\, q^{r+2}}
 = 1
\]
because of the $QQ$-system relation of type $A_1$. Therefore $g(q^r) \in SL(2,K_\Z)$.
}
\end{example}

\begin{example}\label{ex-qw-A2}
{\rm
Consider the commutative ring $K_\Z$ in type $A_2$. The following matrix in $M(3,K_\Z)$:
\[
 g(q^{r}) := 
 \pmatrix{\bQ_{\varpi_1,\, q^r}& \bQ_{s_1(\varpi_1),\, q^r} & \bQ_{s_2s_1(\varpi_1),\, q^r}\cr
 \bQ_{\varpi_1,\, q^{r+2}}& \bQ_{s_1(\varpi_1),\, q^{r+2}} & \bQ_{s_2s_1(\varpi_1),\, q^{r+2}}\cr
 \bQ_{\varpi_1,\, q^{r+4}}& \bQ_{s_1(\varpi_1),\, q^{r+4}} & \bQ_{s_2s_1(\varpi_1),\, q^{r+4}}
 },
 \qquad (r\in 2\Z),
\]
can be regarded as the quantum Wronskian matrix of the three functions $z \mapsto \bQ_{\varpi_1,\,z}$, 
$z \mapsto \bQ_{s_1(\varpi_1),\,z}$ and $z \mapsto \bQ_{s_2s_1(\varpi_1),\,z}$
evaluated at $z = q^{r}$.
The $QQ$-system relations of type $A_2$ imply that
\[
\left|
\matrix{\bQ_{\varpi_1,\, q^r}& \bQ_{s_1(\varpi_1),\, q^r}\cr
\bQ_{\varpi_1,\, q^{r+2}}& \bQ_{s_1(\varpi_1),\, q^{r+2}} 
}
\right|
= \bQ_{\varpi_2,\, q^{r+1}},
\qquad
\left|
\matrix{\bQ_{s_1(\varpi_1),\, q^r}& \bQ_{s_2s_1(\varpi_1),\, q^r}\cr
\bQ_{s_1(\varpi_1),\, q^{r+2}}& \bQ_{s_2s_1(\varpi_1),\, q^{r+2}} 
}
\right|
= \bQ_{s_1s_2(\varpi_2),\, q^{r+1}}.
\]
Finally, the determinant of this matrix is also 1. 
Indeed, using the Lewis Caroll identity, we have
\[
\left|
\begin{array}{cc}
\bQ_{\varpi_2,\, q^{r+1}}
&
\bQ_{s_1s_2(\varpi_2),\, q^{r+1}}
\\
\bQ_{\varpi_2,\, q^{r+3}}
&
\bQ_{s_1s_2(\varpi_2),\, q^{r+3}}
\end{array}
\right|
=
\bQ_{s_1(\varpi_1),\, q^{r+2}}
\left|
\matrix{\bQ_{\varpi_1,\, q^r}& \bQ_{s_1(\varpi_1),\, q^r} & \bQ_{s_2s_1(\varpi_1),\, q^r}\cr
 \bQ_{\varpi_1,\, q^{r+2}}& \bQ_{s_1(\varpi_1),\, q^{r+2}} & \bQ_{s_2s_1(\varpi_1),\, q^{r+2}}\cr
 \bQ_{\varpi_1,\, q^{r+4}}& \bQ_{s_1(\varpi_1),\, q^{r+4}} & \bQ_{s_2s_1(\varpi_1),\, q^{r+4}}
 }
\right|.
\]
Now, the $QQ$-system relations also imply that 
\[
\left|
\begin{array}{cc}
\bQ_{\varpi_2,\, q^{r+1}}
&
\bQ_{s_1s_2(\varpi_2),\, q^{r+1}}
\\
\bQ_{\varpi_2,\, q^{r+3}}
&
\bQ_{s_1s_2(\varpi_2),\, q^{r+3}}
\end{array}
\right|
=
\bQ_{s_1(\varpi_1),\, q^{r+2}},
\]
thus by comparison, we get that $\det(g(q^r)) = 1$. Hence, $g(q^r) \in SL(3,K_\Z)$.
}
\end{example}

\subsection{Quantum Wronskian elements in ${G(K_\Z)}$} \label{subsec-qW}

We will now explain how Proposition~\ref{prop-9-3} allows to generalize the previous examples
and formulate a general statement valid for every group $G$ of type $A, D, E$. 
This discussion is a simple translation in our algebraic setting of ideas of Koroteev and Zeitlin 
formulated in the context of $q$-opers \cite{KZ}. What we would like to emphasize is that 
the cluster algebra structure of $K_\Z$ is the natural combinatorial tool for dealing with these formulas.

We first note the following consequence of the cluster algebra structure on the coordinate ring of the open double Bruhat cell.

\begin{Prop}\label{prop-9-7}
Let $k$ be a commutative field of characteristic 0. Let $G = G(k)$ denote the $k$-points of the algebraic group $G$. 
Let $\{\phi_1,\ldots, \phi_d\}$ be a fixed cluster of the cluster algebra structure on the coordinate ring $k[G^{w_0,w_0}]$.
For every $(a_1,\ldots,a_d)\in (k^\times)^d$, there is 
a unique $g\in G$ such that $\phi_1(g) = a_1, \ldots, \phi_d(g) = a_d$.
\end{Prop}
\begin{proof}
Since $\{\phi_1,\ldots, \phi_d\}$ is a cluster, it is algebraically independent, so for every $(a_1,\ldots,a_d)\in (k^\times)^d$, there exists $g\in G^{w_0,w_0}$ such that $\phi_1(g) = a_1, \ldots, \phi_d(g) = a_d$.
Now, because of the Laurent phenomenon, $k[G^{w_0,w_0}]$ is contained in the ring of Laurent polynomials
in the variables $\{\phi_1,\ldots, \phi_d\}$. Hence for every $f\in k[G^{w_0,w_0}]$, the evaluation
$f(g)$ can be expressed as a Laurent polynomial in terms of the evaluations $\phi_1(g)=a_1, \ldots, \phi_d(g)=a_d$ only, which proves unicity.
\cqfd
\end{proof}

\begin{example}
{\rm
We continue Example~\ref{example-9-1} in type $A_2$. Let 
\[
g := \pmatrix{
  a&b&c\cr
  d&e&f\cr
  h&i&j
  }\in G^{w_0,w_0}.
\]
The cluster of the coordinate ring of $G^{w_0,w_0}$ considered in Example~\ref{example-9-1} is:
\[
\{\De_{3,1},\ \De_{2,1},\ \De_{2,2},\ \De_{1,2},\ \De_{1,3},\ \De_{23,12},\ \De_{12,12},\ \De_{12,23}\}.     
\]
First we have:
\[
 h = \De_{3,1}(g),\quad d = \De_{2,1}(g),\quad e = \De_{2,2}(g),\quad b = \De_{1,2}(g),\quad c = \De_{1,3}(g).  
\]
Then we calculate easily:
\[
 a = \left(\frac{\De_{12,12} + \De_{1,2}\De_{2,1}}{\De_{2,2}}\right)(g),\quad
 f = \left(\frac{\De_{12,23} + \De_{2,2}\De_{1,3}}{\De_{1,2}}\right)(g),
\]
\[
 i = \left(\frac{\De_{23,12} + \De_{3,1}\De_{2,2}}{\De_{2,1}}\right)(g),\quad 
 \left|
\matrix{
e&f\cr
i&j
 }
\right| 
 = \left(\frac{\De_{2,2} + \De_{12,23}\De_{23,12}}{\De_{12,12}}\right)(g),
\]
where the last equality uses the fact that $\det(g) = 1$. Since $e = \De_{2,2}(g)$ is invertible, the entry $j$ can also be calculated from these equations. Hence we see that $g$ is entirely
determined by the values 
\[
\{\De_{3,1}(g),\ \De_{2,1}(g),\ \De_{2,2}(g),\ \De_{1,2}(g),\ \De_{1,3}(g),\ \De_{23,12}(g),\ \De_{12,12}(g),\ \De_{12,23}(g)\},     
\]
which can be arbitrary invertible elements in $k$. 
\cqfd
}
\end{example}

We now introduce a generalization of the definition of a quantum Wronskian matrix.
Let $q\in\C^*$, not a root of unity. Write $q^{2\Z} := \{q^r\mid r\in 2\Z\}$.

\begin{Def}\label{def-qW}
Let $c$ be a fixed Coxeter element. Recall the integers $m_i$ attached to $c$ in §\ref{sec-9-1}.
Let $\gamma : z \mapsto g(z)$ be a map from $q^{2\Z}$ to $G$.
We say that $\gamma$ is a quantum $(G,c)$-Wronskian if the following system of
equations is fulfilled:
\[
\Delta_{c^k(\varpi_i),\,c^\ell(\varpi_i)}(g(z)) = \Delta_{c^{k-1}(\varpi_i),\,c^\ell(\varpi_i)}(g(q^2z)),
\quad (i\in I,\ 1\le k\le m_i,\ 0\le \ell \le m_i,\ z\in q^{2\Z}).
\]
\end{Def}

\begin{example}
{\rm
The map $q^r \mapsto g(q^r)$ of Example~\ref{ex-qw-A2} is a quantum $(G,c)$-Wronskian for $c = s_1s_2$ and $G=SL(3,K_\Z)$.
Indeed, by definition of the matrix $g(q^r)$, its entries $\Delta_{c^k(\varpi_1),\,c^\ell(\varpi_1)}(g(q^r))$ satisfy the relations :
\[
\Delta_{c^k(\varpi_1),\,c^\ell(\varpi_1)}(g(q^r)) = \bQ_{c^l(\varpi_1),\, q^{r+2k}}
= \Delta_{c^{k-1}(\varpi_1),\,c^\ell(\varpi_1)}(g(q^{r+2})),
\quad (0< k\le 1,\ 0\le\ell \le 2,\ r\in 2\Z).
\]
As to the fundamental weight $\varpi_2$, we have $m_2=1$, and for $0\le \ell\le 1$, the generalized minor $\Delta_{\varpi_2,c^\ell(\varpi_2)}(g(q^r))$ (\resp
$\Delta_{c(\varpi_2),c^\ell(\varpi_2)}(g(q^r))$) is a $2\times 2$ minor taken on the first and second rows (\resp second and third rows) of $g(q^r)$, so the desired identity follows immediately from the one for $\varpi_1$.
In fact the only nontrivial property in this example is that $g(q^r) \in G$,
that is $\det(g(q^r)) = 1$.

Note, that $q^r \mapsto g(q^r)$ is \emph{not} a quantum $(G,\tilde{c})$-Wronskian for the Coxeter element $\tilde{c} = s_2s_1$, since for instance
\[
\Delta_{\tilde{c}(\varpi_1),\,\varpi_1}(g(q^r)) = \bQ_{\varpi_1,\, q^{r+4}} \quad \mbox{  and  } \quad
\Delta_{\varpi_1,\,\varpi_1}(g(q^r)) = \bQ_{\varpi_1,\, q^r}. 
\]
\cqfd
}
\end{example}

Let $k$ denote the fraction field of $K_\Z$. Let $c$ be a fixed Coxeter element. 
Recall the height function $l_c$ defined in \S\ref{subsec4-3-1}. 
Using the description of the initial seed $\SS_c$ given in \S\ref{sec-9-1} and 
Proposition~\ref{prop-9-7} we can define the following elements of $G(k)$.
\begin{Def}\label{def-g}
For $r\in 2\Z$, we denote by $g_c(q^r)$ the unique element of $G(k)$ such that:
\[
\Delta_{c^k(\varpi_i),\,c^\ell(\varpi_i)}(g_c(q^r)) =
\bQ_{c^\ell(\varpi_i),\, q^{2+r+2(k-m_i)-l_c(i)}}, \quad
(i\in I,\ \ 0\le k,\ell\le m_i,\ \ m_i -1 \le k+\ell\le m_i).
\]
\end{Def}
Note that, by Proposition~\ref{prop-invertible}, the elements $\bQ_{w(\varpi_i),q^r}$ are 
invertible in $K_\Z$ for every $(i,r)\in V$. Therefore the elements $g_c(q^r)$ belong in fact
to $G(K_\Z)$.

We can now state the main result of this section.

\begin{Prop}\label{thm-qw}
The map $q^r \mapsto g_c(q^r)$ is a quantum $(G,c)$-Wronskian. 
\end{Prop}

\begin{proof}
The formula of Definition~\ref{def-g} is obtained in the following way. 
Consider the seed $\uSS_c$ of $K_\Z$ given by Theorem~\ref{main-Thm}.
As already noted in Proposition~\ref{prop-9-3}, the infinite quiver $\Gamma_c$ of $\uSS_c$ contains 
a finite subquiver $\gamma_c$ isomorphic to the quiver of the initial seed
$\SS_c$ of $k[G^{w_0,w_0}]$. The cluster variables of $\SS_c$ are exactly the generalized
minors occurring in Definition~\ref{def-g}, and their evaluations at $g(q^r)$ 
are the corresponding $Q$-variables located at the same position on the subquiver $\gamma_c$
of $\uSS_c$, except that their spectral parameter has been multiplied uniformly by $q^r$.
In fact, this spectral parameter shift by $q^r$ can be implemented by performing $|r|/2$ times
the canonical sequence of $N$ mutations at all red vertices if $r>0$, or at all green vertices
if $r<0$. We shall denote this seed by $\uSS_c[r]$.

Let us now perform the same sequence of $N$ mutations at all red vertices of $\SS_c$. (Note that
none of the red vertices of $\SS_c$ is frozen.) Using the generalized minor identities \cite[Theorem 1.17]{FZ-Bruhat}, this will replace every red cluster variable 
$\De_{c^{k-1}(\varpi_i),\,c^{m_i-k}(\varpi_i)}$ by $\De_{c^{k}(\varpi_i),\,c^{m_i-k+1}(\varpi_i)}$.On the other hand, because of the $QQ$-relations, the same sequence of $N$ mutations at all red vertices of $\uSS_c[r]$ will replace every red cluster variable $\bQ_{c^{m_i-k}(\varpi_i),\, q^{r+2(k-m_i)-l_c(i)}}$ by $\bQ_{c^{m_i-k+1}(\varpi_i),\, q^{r+2+2(k-m_i)-l_c(i)}}$.
Now, by definition of $g_c(q^{r+2})$, we have
\[
\De_{c^{k-1}(\varpi_i),\ c^{m_i-k+1}(\varpi_i)}(g_c(q^{r+2})) =
\bQ_{c^{m_i-k+1}(\varpi_i),\, q^{r+2+2(k-m_i)-l_c(i)}}, 
\]
hence we obtain 
\[
\De_{c^{k}(\varpi_i),\,c^{m_i-k+1}(\varpi_i)}(g_c(q^r)) = 
\De_{c^{k-1}(\varpi_i),\ c^{m_i-k+1}(\varpi_i)}(g_c(q^{r+2})).
\]
This is the quantum $(G,c)$-Wronskian property of Definition~\ref{def-qW} when $k+\ell = m_i+1$.

Let us call $\mu^{(1)}$ the sequence of $N$ mutations that we have performed. In the new seed 
$\mu^{(1)}(\SS_c)$ we consider all non-frozen vertices which are sinks for the vertical arrows and
call them red vertices of $\mu^{(1)}(\SS_c)$. There are $N-n$ of them. Let $\mu^{(2)}$ denote the sequence of mutations at these red vertices of $\mu^{(1)}(\SS_c)$. Repeating the same reasoning for 
$\mu^{(2)}$ as we did for $\mu^{(1)}$, we will deduce the quantum $(G,c)$-Wronskian property for 
$k+\ell = m_i+2$. Iterating, we will verify in this way the quantum $(G,c)$-Wronskian property for 
$k+\ell > m_i$.

Finally, the same reasoning but replacing everywhere red vertices by green vertices will give the 
proof of the quantum $(G,c)$-Wronskian property for 
$k+\ell \le m_i$.
\cqfd
\end{proof}

\begin{remark}\label{rem-qW}
{\rm
Definition~\ref{def-qW} is an algebraic version of the notion of $(G,q)$-Wronskian introduced by Koroteev and Zeitlin in \cite{KZ} to generalize the $(SL(n),q)$-Wronskian of \cite{KSZ}. A $(G,q)$-Wronskian is a certain type of section of a principal $G$-bundle on the projective line equipped with a $q$-connection depending on the choice of a Coxeter element $c$.  
Proposition~\ref{thm-qw} corresponds to \cite[Proposition 4.17]{KZ}. In our situation, the rôle of the $q$-connection is played by the sequence of cluster mutations at red vertices.

Note that in \cite{KSZ, KZ}, the $QQ$-systems have twisting parameters $\xi_i$, $\widetilde{\xi_i}$, and singularities encoded by certain polynomials $\L_i(z)$ (see \cite[\S 3.6]{KZ}). In contrast, our $QQ$-system equations  do not involve any additional parameters (see Proposition~\ref{Prop-renorm-QQ}). A consequence of this is that they cannot have polynomial solutions. The solutions $\bQ_{w(\varpi_i),q^r}\in K_\Z$ that we consider in this paper are all formal power series for $w\not = e$. We were informed by Koroteev and Zeitlin that $QQ$-systems without twisting parameters and singularities naturally occur in relation with $q$-opers on a disc (instead of a projective line).  

}
\end{remark}


\bigskip
\small
\noindent
\begin{tabular}{ll}
Christof {\sc Geiss} & Instituto de Matem\'aticas,\\
&Universidad Nacional Aut\'onoma de M\'exico,\\
&Ciudad Universitaria, 04510
Ciudad de M\'exico, M\'exico\\
& email : {\tt christof.geiss@im.unam.mx}
\\[5mm]
David {\sc Hernandez}  
& Université Paris Cité and Sorbonne Université, \\
&CNRS, IMJ-PRG, \\
& F-75006, Paris, France\\
& email : {\tt david.hernandez@imj-prg.fr}\\
[5mm]
Bernard {\sc Leclerc}  & Universit\'e de Caen Normandie,\\
&CNRS UMR 6139 LMNO,\\ &14032 Caen, France\\
&email : {\tt bernard.leclerc@unicaen.fr}
\end{tabular}

\end{document}